\pdfmapline{+dmjhira dmjhira <dmjhira.pfb}
\documentclass[a4paper, english]{amsart}

\usepackage[utf8]{inputenc}
\usepackage{babel}

\usepackage{amsmath}
\usepackage{amssymb}
\usepackage{mathtools}
\usepackage{thmtools}
\usepackage{stmaryrd}
\usepackage{enumitem}
\usepackage{hyperref}
\usepackage{cleveref}
\usepackage{tikz-cd}
\usepackage{nth}
\usepackage{csquotes}
\usepackage[style=alphabetic,
    minalphanames=4,maxalphanames=4,
    maxbibnames=4]{biblatex}
\usepackage{bbold}
\numberwithin{equation}{section}

\declaretheorem[style=plain,sibling=equation]{theorem}
\declaretheorem[style=plain,sibling=equation]{lemma}
\declaretheorem[style=plain,sibling=equation]{corollary}
\declaretheorem[style=plain,sibling=equation]{proposition}

\declaretheorem[style=definition,sibling=equation]{definition}
\declaretheorem[style=definition,sibling=equation]{example}
\declaretheorem[style=definition,sibling=equation]{construction}

\declaretheorem[style=remark,sibling=equation]{remark}
\declaretheorem[style=definition, sibling=equation]{notation}

\newcounter{mainresults}

\declaretheorem[style=plain,numberlike=mainresults,
name=Theorem,
refname={theorem,theorems},
Refname={Theorem,Theorems}
]{mainresult}

\tikzcdset{scale cd/.style={every label/.append style={scale=#1},
    cells={nodes={scale=#1}}}}

\newcommand{\field}{\mathbb{k}}
\newcommand{\category}[1]{\mathsf{#1}}

\newcommand{\catMod}{\category{Mod}}
\newcommand{\catmod}{\category{mod}}

\newcommand{\kos}{/\!\!/}
\newcommand{\bnd}{\mathsf{b}}
\newcommand{\per}{\mathsf{per}}

\DeclareMathOperator{\Hom}{Hom}

\newcommand{\limit}{\varprojlim}
\newcommand{\colim}{\varinjlim}

\makeatletter
\def\varholim@#1#2{%
    \vtop{\m@th\ialign{##\cr
    \hfil$#1\operator@font holim$\hfil\cr
    \noalign{\nointerlineskip\kern1.5\ex@}#2\cr
    \noalign{\nointerlineskip\kern-\ex@}\cr}}%
}
\def\hocolim{%
  \mathop{\mathpalette\varholim@{\rightarrowfill@\textstyle}}\nmlimits@
}
\def\holim{%
  \mathop{\mathpalette\varholim@{\leftarrowfill@\textstyle}}\nmlimits@
}
\makeatother

\DeclareMathOperator{\id}{Id}

\DeclareMathOperator{\cone}{cone}

\DeclareMathOperator{\emorp}{End}
\DeclareMathOperator{\thick}{thick}
\DeclareMathOperator{\Loc}{Loc}

\DeclareMathOperator{\Spec}{Spec}
\DeclareMathOperator{\Tor}{Tor}
\DeclareMathOperator{\Ext}{Ext}
\DeclareMathOperator{\add}{add}
\DeclareMathOperator{\Add}{Add}

\DeclareMathOperator{\im}{im}
\DeclareMathOperator{\coker}{coker}

\DeclareFontFamily{U}{dmjhira}{}
\DeclareFontShape{U}{dmjhira}{m}{n}{
  <-> dmjhira
}{}
\DeclareFontSubstitution{U}{dmjhira}{m}{n}

\newcommand{\ideal}[1]{\mathfrak{#1}}

\newcommand{\ida}{\ideal{a}}
\newcommand{\idb}{\ideal{b}}
\newcommand{\comp}[1]{#1^{\wedge}}
\newcommand{\intHom}{\underline{\Hom}}
\newcommand{\SW}[1]{{#1}^{\mathrm{SW}}}

\title[Torsion and complete dualizables in Noetherian tt-categories]%
{Torsion and complete dualizable objects in tensor-triangulated categories over a Noetherian ring}
\author{Jun Maillard and Jan Šťovíček}

\thanks{The research was supported by the Czech Science Foundation grant GA~{\v C}R 23-05148S}

\begin{document}

\begin{abstract}
We study categories of dualizable torsion and complete objects for compactly-rigidly generated tensor-triangulated categories $\category{T}$ with a Noetherian central action of a graded commutative Noetherian ring $R$. We show that they always admit a natural Noetherian action of the completed graded ring $\comp{R}$ and that the categories of dualizable torsion and complete objects can be abstractly reconstructed as tensor-triangulated $\comp{R}$-linear categories from the category of compact torsions objects with the corresponding structure.
If the category of compact objects of $\category{T}$ in addition admits a strong generator $g$, we show that the torsion coreflection (resp.\ complete reflection) of $g$ is a strong generator for the category of dualizable torsion (resp.\ dualizable complete) objects.
In that case, we also show that the categories of dualizable torsion and compact torsion objects determine each other in terms of Brown-type representability theorems.
\end{abstract}

\maketitle
\tableofcontents

\section{Introduction}\label{sec:Introduction}

In the seminal paper \autocite{balmerTTGeometryFoundation}, Balmer defined a ringed space, nowadays called the Balmer spectrum, associated with any small tensor-triangulated category. This evolved into the field of tensor-triangular geometry \autocite{balmerTensorTriangularGeometry2010}, which allows one to study a large class of seemingly unrelated examples from algebraic geometry, homotopy theory, representation theory or motivic theory using geometric intuition and unified homological tools.
However, one useful and common technique in the study of commutative Noetherian rings was not available at the level of small tensor-triangulated categories---the adic completion which would abstract the completion functor $\category{D}^\per(R)\to\category{D}^\per(\comp{R}_\ida)$ for an ideal $\ida$ in a ring $R$ (albeit it was known how to model various aspects of completion by recollements of big tensor-triangulated categories, e.g.\ in~\autocite[\S3.3]{hoveyPalmieriStricklandMemoir1997} or~\autocite[\S2]{greenleesTate2001}). In their recent work, Benson, Iyengar, Krause and Pevtsova \autocite{bensonLocallyDualisableLocalAlgebra2024,bensonLocallyDualisableModular2024}, Balmer and Sanders \autocite{balmerSandersPerfectComplexesCompletion2024,balmerSandersTateIVT2025}, and Nauman, Pol and Ramzi \autocite{naumanPolRamziFractureSquare2024} suggested and supported by classes of well behaved examples a small tensor-triangulated version of the completion, which we are going to explain.

\subsection{Adic completion in tensor-triangular geometry}

Rigid (in the sense that every object is dualizable) small tensor-triangulated categories $(\category{K},\otimes,\mathbb{1})$ arise in practice, up to idempotent completion, as the categories of compact objects in compactly-rigidly generated tensor-triangulated categories $(\category{T},\otimes,\mathbb{1})$. In the realm of stable $\infty$-categories, where the authors of \autocite{naumanPolRamziFractureSquare2024} worked, this is always the case since we can Ind-complete (the underlying stable $\infty$-category of) $\category{K}$ to $\category{T}$. If we fix any Zariski closed subset $\mathcal{V}$ with a quasi-compact complement in the Balmer spectrum of $\category{K}$ (which is the same as choosing a thick $\otimes$-ideal of $\category{K}$ with a single generator by \autocite[Proposition~2.14]{balmerTTGeometryFoundation}), we obtain a recollement
\[
    \begin{tikzcd}
        \Gamma_{\mathcal{V}}\category{T} 
            \arrow[r, bend left,  tail, "i_!", pos=0.4]
            \arrow[r, bend right, tail, "i_*"', pos=0.4] &
        \category{T}
            \arrow[l, "i^*"' description] 
            \arrow[r, bend left, "p^*", pos=0.6]
            \arrow[r, bend right, "p^!"', pos=0.6] &
        L_{\mathcal{V}}\category{T}
            \arrow[l, tail, "p_*"' description] 
    \end{tikzcd}
\]
In analogy to algebraic geometry, $\Gamma_\mathcal{V}\category{T}$ is the category of objects `supported on $\mathcal{V}$', also known as `$\mathcal{V}$-torsion objects', while $L_\mathcal{V}\category{T}$ is the localization to the open complement of $\mathcal{V}$. Just by formal properties of the recollement, there is one more natural full subcategory to consider, $\Lambda^{\mathcal{V}} \category{T}=i_*i^*\category{T}$, whose objects are usually referred to as `$\mathcal{V}$-complete'. The terminology here is motivated by the Matlis-Greenlees-May duality \autocite{greenleesMayLocalHomology1992}, whose commutative-algebraic aspects will be discussed in \S\ref{subsec:torsion-complete-modules}.
This is all standard and we refer to \autocite[\S3]{balmerSandersTateIVT2025} for a nice account.

The interesting part is that the subcategories $\Gamma_{\mathcal{V}} \category{T}$ and $\Lambda^{\mathcal{V}} \category{T}$ are naturally compactly generated tensor-triangulated and equivalent as such, but their monoidal units are generally not compact. In particular, their subcategories of compact objects, which are even equal as subcategories of $\category{T}$, are not tensor-triangulated. The natural `small' tensor-triangulated subcategory, which is a candidate for the $\mathcal{V}$-adic completion in the tensor-triangular geometry, is the category of dualizable objects, $(\Lambda^{\mathcal{V}}\category{T})^d$ (or equivalently $(\Gamma_{\mathcal{V}}\category{T})^d$), together with the completion functor $\Lambda^{\mathcal{V}}=i_*i^*\colon\category{T}^c\to(\Lambda^{\mathcal{V}}\category{T})^d$ (resp.\ $\Gamma_{\mathcal{V}}=i_!i^*\colon\category{T}^c\to(\Gamma_{\mathcal{V}}\category{T})^d$)

\subsection{Problems about complete dualizable objects}

The main challenge which is treated in the references above is how to understand the complete (or torsion) dualizable objects and how exactly they are related to $\category{K}=\category{T}^c$. First of all, one may hope to reconstruct the category of complete dualizable objects $(\Lambda^{\mathcal{V}}\category{T})^d$ from the category of torsion/complete compact objects $(\Lambda^{\mathcal{V}}\category{T})^c=(\Gamma_{\mathcal{V}}\category{T})^c\subset\category{T}$ by some sort of completion, similarly as the adic completion of a ring $\comp{R}_\ida$ is obtained as a limit of its torsion quotients $R/\ida^n$. Albeit this might seem naive at first, there are instances of metric completion of triangulated categories initiated by Krause~\autocite{krauseCompletingDper2020} and perfected by Neeman in~\autocite{neemanMetrics2020}. Unfortunately, these techniques heavily depend on the presence of a non-degenerate $t$-structure and interesting classes of examples, e.g.\ coming from modular representation theory of finite groups, do not admit any of these (as it often happens that some power of the suspension functor is isomorphic to the identity). However, as we will explain below, one still can achieve complete reconstruction of $(\Lambda^{\mathcal{V}}\category{T})^d$ from $(\Lambda^{\mathcal{V}}\category{T})^c$ at the level of tensor-triangulated categories (so without reference to any model such as $\infty$-categories) even in a $t$-structure unfriendly environment.

Another previously intensely studied problem is about thick generation.
As mentioned, since the completion functor $\Lambda^\mathcal{V}=i_*i^*\colon\category{T}\to\Lambda^\mathcal{V}\category{T}$ is monoidal, the images of the compact objects of $\category{T}$ are dualizable in $\Lambda^{\mathcal{V}}\category{T}$. We hence have sequence of inclusions of categories
\[
    (\Lambda^{\mathcal{V}}\category{T})^c \subsetneq \thick(\Lambda^{\mathcal{V}}\category{T}^c) \subset (\Lambda^{\mathcal{V}}\category{T})^d,
\]
where the first inclusion is typically proper. On the contrary, the second inclusion is known to be an equality in at least the following cases:
\begin{enumerate}[label=(\arabic*)]
\item \autocite{mathewGaloisHomotopy2016} verified the equality for Morava E-theories in~\autocite[Proposition~10.11]{mathewGaloisHomotopy2016}.

\item \autocite[\S11]{naumannPolSeparableAlgebras2024} considered the case the homotopy category $\category{T}$ of a connective $\mathbb{E}_\infty$-ring spectrum $A$ which is complete with respect to a finitely generated ideal $\ida\subset\pi_0(A)$.

\item \autocite[\S5]{stricklandLocCompSpectra} established this for the $p$-completion of the stable homotopy category of finite spectra; see~\autocite[Proposition~8.6]{balmerSandersTateIVT2025} for more  explanation.

\item \autocite{balmerSandersPerfectComplexesCompletion2024} verified the equality in the case of the derived category $\category{T} = \category{D}(R)$ over a commutative ring $R$ satisfying a so-called Koszul-completeness condition, which is always true for commutative Noetherian rings.

\item\label{item:modularCompletionExpl} \autocite{bensonLocallyDualisableModular2024} proved the equality for a compactly-rigidly generated tensor-triangulated category $\category{T}$ with an action of a graded commutative local Noetherian ring $R$ and completion with respect to the maximal ideal $\mathfrak{m}\subset R$, under the additional assumption of a so-called local regularity condition.
\end{enumerate}

Despite all these encouraging results, the second inclusion still can be proper. An example from the homotopy theory, in the category of $K(1)$-local spectra at an odd prime was presented in~\autocite[\S15.1]{hoveyStricklandMoravaKTheories1999} (and translated to our language in~\autocite[\S5]{bensonLocallyDualisableModular2024}).
Besides being much less tractable, this also brings conceptual issues; see \autocite[Remark~5.4]{balmerSandersPerfectComplexesCompletion2024} or \autocite[Remark~3.20]{balmerSandersTateIVT2025}. 
So some conditions for the equality are needed and we will exhibit a natural and often satisfied one in terms of strong generation of $\category{T}^c$.

\subsection{The triangulated reconstruction result}
Let us first explain our setting, which is inspired by that of~\autocite{bensonLocallyDualisableModular2024}.
We will assume an action of a graded commutative Noetherian (but not necessarily local) ring $R$ on $\category{T}$ so that (graded) Hom-groups for all pairs of compact objects are finitely generated $R$-modules. This encompasses both $\category{T}=\category{D}(\catMod R)$ for a commutative Noetherian ring $R$ and $\category{T}=\category{KInj}(\field G)$ where $\field$ is a field, $G$ is a finite group and the acting ring is the group cohomology $R=H^*(G,\field)$ (see \autocite[\S10]{bensonLocalCohomologySupport2008}). In fact, there is a far-reaching common generalization of these two settings: the Ind-completion $\category{T}=\category{Rep}(G,S)$ of the bounded derived category $\category{D}^\bnd(G,S)$ of $G$-lattices for a finite flat group scheme over a commutative Noetherian ring $S$, which is acted on again by $R=H^*(G,S)$. This category was the main character of \autocite{barthelLattices2023}, where the required finiteness conditions come from \autocite{vdKallenNoetherianBase2023} (see \autocite[Proposition 3.15]{lauBSpectrumStacks2023} for an easier special case).
But we can also consider the derived category of any dg algebra or the homotopy category of any ring spectrum whose cohomology (respectively homotopy) ring is Noetherian. For instance, any graded commutative Noetherian algebra viewed as a formal dg algebra with zero differential is an example.

Compared to \autocite{bensonLocallyDualisableModular2024}, the completion which we consider is more general, not only at maximal ideals of $R$, but at arbitrary homogeneous ideals $\ida\subset R$.
In fact, the assumptions from the previous paragraph ensure, using~\autocite{balmerSpectraCube2010}, the existence of a continuous map from the Balmer spectrum of $\category{T}^c$ to $\Spec R$, the homogeneous spectrum of the graded ring $R$. If $R=\emorp(\mathbb{1})$ is the homogeneous endomorphism ring of the tensor unit, this map is even surjective by~\autocite[Theorem 7.3]{balmerSpectraCube2010}, and in many cases it is known to be bijective. In our paper, we will restrict our attention only to closed sets in the Balmer spectrum which are preimages of Zariski closed sets $V(\ida)\subset\Spec R$ under this map.

Under these assumptions, we prove that the categorical completion of $\category{T}$ very nicely interacts with completion of the acting ring. 

\begin{mainresult}[see \Cref{thm:finite-gen-hom-dualizables} and \Cref{corollary:completion-on-homs}]
\label{main:completion-on-homs}
Let $\category{T}$ be an $R$-linear, Noetherian, tensor-triangulated, compactly-rigidly generated category and let $\ida\subset R$ be a homogeneous ideal. Then:

\begin{enumerate}[label=(\alph*)]
\item both $(\Gamma_\ida\category{T})^d$ and $(\Lambda^\ida\category{T})^d$ have a canonical action of the $\ida$-adic completion $\comp{R}_\ida$ and all Hom-groups are finitely generated over $\comp{R}_\ida$;
\item if $c,d\in\category{T}^c$, we have canonical isomorphisms
\[
\Hom_\category{T}(c,d)\otimes_R\comp{R}_\ida\cong
\Hom_{\Gamma_\ida\category{T}}(\Gamma_\ida c,\Gamma_\ida d)\cong
\Hom_{\Lambda^\ida\category{T}}(\Lambda^\ida c,\Lambda^\ida d).
\]
\end{enumerate}
\end{mainresult}

Now we can consider the restricted Yoneda functor $h\colon \Gamma_\ida\category{T}\to\catMod(\Gamma_\ida\category{T})^c$, which has proved to be a powerful tool to study compactly generated triangulated categories. In fact, as noted in~\autocite{balmerKrauseStevensonSmashingFrame2020}, the tensor product $\otimes\colon (\Gamma_\ida\category{T})^c\times(\Gamma_\ida\category{T})^c\to(\Gamma_\ida\category{T})^c$ extends through the Day convolution construction to $\catMod(\Gamma_\ida\category{T})^c$ in such a way that $h$ becomes a strong monoidal functor. An additional technical issue arises in our situation as the tensor unit is not contained in $(\Gamma_\ida\category{T})^c$, but $(\Gamma_\ida\category{T})^c$ still contains a countable direct system of `approximate units' as explained in \Cref{remark:approximate-unitality}, and this suffices. The key point for our next main result is that the restriction of $h$ to the category $(\Gamma_\ida\category{T})^d$ of torsion dualizable objects is fully faithful by~\Cref{prop:yoneda-restricted-fully-faithful}. This leads to one of our main results.

\begin{mainresult}[see \Cref{thm:reconstruction}]
\label{main:reconstruction}
Let $\category{T}$ be an $R$-linear, Noetherian, tensor-tri\-an\-gu\-lated, compactly-rigidly generated category, and fix a homogeneous ideal $\ida\subset R$.
Then we can recover $(\Gamma_\ida\category{T})^d$ with its tensor-triangulated $\comp{R}_\ida$-linear category structure from the approximate unital $R$-linear tensor-triangulated category $(\Gamma_\ida\category{T})^c$.

More precisely, the restricted Yoneda functor $h\colon \Gamma_\ida\category{T}\to\catMod(\Gamma_\ida\category{T})^c$ restricts to an equivalence between $(\Gamma_\ida\category{T})^d$ and the full subcategory of $\catMod(\Gamma_\ida\category{T})^c$ formed by the dualizable objects with respect to the Day convolution product. The triangles in $(\Gamma_\ida\category{T})^d$ are recovered as suitable colimits of triangles in $(\Gamma_\ida\category{T})^c$.
\end{mainresult}

If we assume in addition that $\category{T}^c$ is strongly generated, or in other terminology, it has a finite dimension in the sense of Rouquier \autocite{rouquierDimensions2008}, there is another way in which $(\Gamma_\ida\category{T})^c$ and $(\Gamma_\ida\category{T})^d$ determine each other. In order to state it, note that given any $c\in(\Gamma_\ida\category{T})^c$ and $d\in(\Gamma_\ida\category{T})^d$, the Hom-group $\Hom_\category{T}(c,d)$ is annihilated by some power of $\ida$. If we denote the category of finitely generated $\ida$-power torsion $R$-modules by $\category{tors} R$, we obtain the following result which is reminiscent of~\autocite[Theorem~0.4]{neeman2Brown2025}.

\begin{mainresult}[see \Cref{thm:perfectPairing}]
\label{main:perfectPairing}
Suppose that $\category{T}$ is an $R$-linear, Noetherian, tensor-tri\-an\-gu\-lated, compactly-rigidly generated category such that $\category{T}^c$ admits a strong generator, and that $\ida\subset R$ is any homogeneous ideal. Then $\Hom_\category{T}\colon (\Gamma_\ida\category{T})^{c,\mathrm{op}}\times(\Gamma_\ida\category{T})^d\to\category{tors} R$ is a perfect pairing in the sense that
\begin{enumerate}[label=(\arabic*)]
\item the restricted Yoneda functor
$(\Gamma_\ida\category{T})^d \to \category{Fun}_R((\Gamma_\ida\category{T})^{c,\mathrm{op}}, \category{tors} R)$
is fully faithful and the essential image consists precisely of the cohomological functors, and

\item the restricted Yoneda functor
$(\Gamma_\ida\category{T})^{c,\mathrm{op}} \to \category{Fun}_R((\Gamma_\ida\category{T})^d, \category{tors} R)$
is fully faithful and the essential image consists precisely of the homological functors.
\end{enumerate}
\end{mainresult}

\subsection{The thick generation result}
Our next aim is to give a comparatively simple proof of the equality $\thick(\Lambda^\ida\category{T}^c)=(\Lambda^\ida\category{T})^d$.
To do so, we need an important additional assumption which was already mentioned in the previous section: that $\category{T}^c$ admits a strong generator.
 As explained in \Cref{example:strongly-generated=regular-finite-dim}, this condition can be viewed as a regularity condition. Compared to the local regularity of \autocite[Definition 7.1]{bensonLocallyDualisableModular2024}, it is a condition of geometrically global nature, which is expected since we deal with non-local rings. In fact, we prove in~\S\ref{subsec:local-regularity} that the existence of a strong generator always implies local regularity in our setup, and the advantage is that strong generators are much better studied and in many cases easier to construct. All the assumptions that we require are met in the following situations, as well for any categories obtained by localization of these at a multiplicative set of the acting ring $R$ in the sense of \autocite[Theorem~3.6]{balmerSpectraCube2010}:

\begin{enumerate}
\item For $\category{T}=\category{KInj}(\field G)$ where $\field$ is a field, $G$ is a finite group and the acting ring is the group cohomology ring $R=H^*(G,\field)$ (see~\S\ref{subsec:KInjkG}).
\item For $\category{T}=\category{Rep}(G,S)$, where $G$ is a commutative finite flat group scheme over a regular commutative Noetherian ring $S$, assuming also that $S$ is of finite type over an excellent base ring $\field$ of Krull dimension $\le2$. In this case, the `group algebra' $SG$ (in the sense of \autocite[\S4]{barthelLattices2023}) is a commutative algebra of finite type over $\field$ and $\category{T}^c=\category{D}^\bnd(\catmod SG)$ since any finitely generated $S$-module has finite projective dimension over $S$ (see \Cref{example:strictly-finite=regular}). The strong generation assumption then follows from~\autocite[Theorem~0.15 and Reminder~0.13]{neemanStrongGenerators2021}, while the finiteness properties of the action of $R=H^*(G,S)$ come from~\autocite{vdKallenNoetherianBase2023}, as already mentioned.
\item For $\category{T}=\category{D}(R)$, the derived category of a graded commutative dg ring or an $\mathbb{E}_\infty$-ring spectrum whose cohomology (respectively homotopy) ring $H^*(R)$ is a regular Noetherian ring of finite (graded) Krull dimension. We stress that we assume nothing about the connectivity of $R$, so this includes in particular the above mentioned example from~\autocite{mathewGaloisHomotopy2016}. In this case, $R$ itself is a strong generator of $\category{T}^c=\category{D}^\per(R)$ by \autocite[Propositions~1.1 and~1.3]{hoveyLockridge2011} (the setup is almost that of affine weakly regular tt-categories in the sense of~\autocite{dellAmbrogioStanleyWeaklyRegular2016}, but we need the additional assumption of finiteness of the Krull dimension).
\end{enumerate}

Now, we can state our final main result applicable to all the classes of examples above.

\begin{mainresult}[see \Cref{thm:strongGenDual,thm:characDualizableComp}]
\label{main:strongGenConsequences}
    Let $\category{T}$ be an $R$-linear, Noetherian, tensor-tri\-an\-gu\-lated, compactly-rigidly generated category. Assume that $g$ is a strong generator of $\category{T}^c$ and $\ida$ a homogeneous ideal of $R$. Then:
\begin{enumerate}[label=(\alph*)]
\item $\Gamma_{\ida}g$ is a strong generator of $(\Gamma_{\ida}\category{T})^d$ and $\Lambda^{\ida}g$ is a strong generator of $(\Lambda^{\ida}\category{T})^d$;
\item an object $x\in\Lambda^{\ida}\category{T}$ is dualizable if and only if $\Hom_{\category{T}}(c, x)$ is a finitely generated $\comp{R}_\ida$-module for each $c\in\category{T}^c$.
\end{enumerate}
\end{mainresult}

Note that this proves more then the promised equality $\thick(\Lambda^{\ida}\category{T}^c)=(\Lambda^{\ida}\category{T})^d$. The categories of torsion and complete dualizable objects turn out to be even strongly generated. This agrees with the fact that an adic completion of a commutative Noetherian regular ring of finite Krull dimension is again such. 

Although this part is heavily influenced by \autocite{bensonLocallyDualisableModular2024} and some results could be either directly transposed to our setting or only slightly adapted (\Cref{prop:complAllCompactNoeth}, \Cref{thm:characDualizableComp}), several tools are simply not available outside of the local setting:
\begin{itemize}
\item Matlis and Brown-Comenetz dualities;
\item any description in terms of Artinian or finite length modules.
\end{itemize}
To replace them, we rely on a part of commutative algebra, which is perhaps not that well known - the theory of derived $\ida$-complete modules that we quickly review in~\S\ref{subsec:torsion-complete-modules}. This in particular allows us to directly detect $\ida$-complete objects in $\category{T}$ in terms of the homomorphism $R$-modules into them (\Cref{thm:charPTandDC}). We are not aware of any such characterization in the existing literature.

\subsection{The structure of the article}

\begin{itemize}
\item \Cref{sec:ttSetting} introduces the categorical notions used throughout the article (compactness, dualizability, linearity, torsion and completion) and presents necessary preliminary material. Most of this is taken from the existing literature, but we also include facts about the interaction of the monoidal structure with sequential homotopy colimits in~\S\ref{subsec:hocolim} for which we were unable to find a reference, as well as a criterion for the existence of a strong generator in \Cref{prop:finHomGen}.

\item \Cref{sec:torClosedIdeal} focuses on basic properties of dualizable objects in $\Gamma_\ida\category{T}$ and $\Lambda^\ida\category{T}$, where $\ida\subset R$ is a homogeneous ideal of $R$. After introducing the key concepts of Koszul objects and derived $\ida$-complete $R$-modules, we cohomologically characterize in \Cref{thm:charPTandDC} the objects of $\Lambda^\ida\category{T}$ as those for which the Hom-groups from the compact objects of $\category{T}$ are derived $\ida$-complete $R$-modules. The existing theory around derived $\ida$-complete modules and derived completion also allows us to prove \Cref{main:completion-on-homs} (see~\Cref{thm:finite-gen-hom-dualizables} and \Cref{corollary:completion-on-homs}) as well as a criterion for thick generation of one object of $(\Gamma_\ida\category{T})^d$ from another in~\Cref{thm:genDual}.

\item \Cref{sec:reconstruction} is devoted to the proof of \Cref{main:reconstruction} and explaining all its steps.

\item \Cref{sec:strong-gen} collects all results that depend on strong generators. This involves showing that the existence of a strong generator is inherited from $\category{T}^c$ to both $(\Gamma_\ida\category{T})^d$ and $(\Lambda^\ida\category{T})^d$ in \Cref{thm:strongGenDual}, and a cohomological characterization of dualizable torsion/complete objects inside the categories of all torsion/complete objects in \Cref{thm:characDualizableTor,thm:characDualizableComp}. A combination of these results constitutes \Cref{main:strongGenConsequences}. We also prove the two representability results from \Cref{main:perfectPairing} (see \Cref{thm:perfectPairing}) here, and conclude by a short discussion of the local regularity condition from \autocite{bensonLocallyDualisableModular2024} and an illustration of the theory in representation theory of finite groups.
\end{itemize}

\subsection*{Acknowledgment}

We are grateful to Paul Balmer, Srikanth Iyengar, Henning Krause, Amnon Neeman, Thomas Peirce, Leonid Positselski and Jordan Williamson for helpful discussions and references.

\section{Preliminaries on the tensor-triangulated setting}\label{sec:ttSetting}

A triangulated category $\category{T}$ is \emph{tensor-triangulated} if it is endowed with a symmetric monoidal structure $(\otimes, \mathbb{1})$ such that $- \otimes -$ is exact in both variable (see~\autocite[Definition~3]{balmerTensorTriangularGeometry2010} or~\autocite[Definition~A.2.1]{hoveyPalmieriStricklandMemoir1997}). We will also assume that $- \otimes -$ preserves any existing coproduct in each variable. An internal hom-object $\intHom(y,z)$ between two objects $y,z \in \category{T}$ is defined, when it exists, by the adjunction isomorphism
\[
    \Hom_{\category{T}}(x, \intHom(y,z)) \cong \Hom_{\category{T}}(x \otimes y, z)
\]
natural in $x$. As a consequence of Brown Representability Theorem~\autocite[Theorem~3.1]{neemanBousfieldTechniques1996}, internal hom-objects exist whenever $\category{T}$ is compactly generated.

\subsection{Grading} Throughout this paper, all rings and modules are \emph{graded}, and ideals are \emph{homogeneous}. Recall that a ring $R$ is called \emph{graded-commutative} if $yx=(-1)^{|x||y|}xy$ for all homogeneous elements $x,y\in R$. Such rings behave analogously to ordinary commutative rings; see~\autocite[\S1.5]{brunsHerzogCMBook1993} for specific aspects of the grading.
If $R$ is graded-commutative, we denote by $\Spec R$ the homogeneous spectrum, that is, the set of all prime homogeneous ideals. If $\category{T}$ is a triangulated category, 
\[
    \Hom_{\category{T}}(x, y) = \bigoplus_{i \in \mathbb{Z}} \Hom_{\category{T}}^i(x,y)
\]
denotes the \emph{graded} Hom-group, with $\Hom_{\category{T}}^i(x,y) = \Hom_{\category{T}}^0(\Sigma^{-i} x,y)$. Morphisms appearing in our commutative diagrams are always degree $0$ morphisms.

\subsection{Compactness and dualizability}\label{subsec:compacts} Compact and dualizable objects are classical classes of `small' objects definable in a tensor-triangulated category.

A \textit{compact object} of a triangulated category $\category{T}$ with arbitrary coproducts is an object $c \in \category{T}$, whose representable functor $\Hom_{\category{T}}(c, -)$ commutes with arbitrary coproducts in the sense that the following canonical map is an isomorphism:
\[
		\bigoplus_{i \in \mathcal{I}} \Hom_{\category{T}}(c, x_i) \overset{\sim}\to
    \Hom_{\category{T}}(c, \bigoplus_{i \in \mathcal{I}}x_i)
\]
A set $\category{G}\subset\category{T}$ of compact objects is called \emph{generating} if $\category{T}$ is the smallest triangulated subcategory of $\category{T}$ closed under arbitrary coproducts and containing $\category{G}$.
A compact object $c$ is a \emph{generator} of $\category{T}$ if $\category{G}=\{c\}$ is generating.
This can equivalently be tested by the following criterion; see~\autocite[Theorem~2.1]{neemanBousfieldTechniques1996}:

\begin{proposition}
    A set of compact objects $\category{G}\subset\category{T}^c$ is generating if and only if, for any object $x$ in $\category{T}$, $\Hom_{\category{T}}(c, x) = 0$ for each $c\in\category{G}$ implies that $x = 0$.
\end{proposition}

The full subcategory of compact objects is then the smallest thick subcategory of $\category{T}$ containing $\category{G}$, also by~\autocite[Theorem~2.1]{neemanBousfieldTechniques1996}. We denote it by $\category{T}^c$.
Note also that compactly generated triangulated categories have products since for any family of objects $x_i\in\category{T}$, $i\in I$, the functor $\prod_{i\in I}\Hom(-,x_i)$ is representable by the  Brown Representability Theorem~\autocite[Theorem~3.1]{neemanBousfieldTechniques1996}.

A \emph{dualizable} (sometimes also called \emph{rigid}) object in a tensor-triangulated category $\category{T}$ is an object $d$ such that the canonical map
\[
    \intHom_{\category{T}}(d, \mathbb{1}) \otimes x \overset{\sim}\to \intHom_{\category{T}}(d, x)
\]
is an isomorphism. The object $\SW{d} = \intHom_{\category{T}}(d, \mathbb{1})$ is then called the \emph{Spanier-Whitehead dual} of $d$, and is dualizable as well, with Spanier-Whitehead dual $d$.
We refer to e.g.\ to~\autocite[\S III.1]{lewisMaySteinbergerEquivariantHomotopy1986} for this and other basic facts about dualizable (there called `finite') objects.
 The full subcategory of dualizable objects, denoted by $\category{T}^d$, is a thick subcategory of $\category{T}$. We will say that $\category{T}$ is \emph{rigidly generated} if $\category{T}$ is the smallest triangulated subcategory of $\category{T}$ closed under arbitrary coproducts and containing $\category{T}^d$.

Basic relations between these notions are recorded in \autocite[\S2.1]{hoveyPalmieriStricklandMemoir1997}.
Dualizable and compact objects are also related by the following tensor duality.

\begin{proposition}\label{prop:compactRigid}
    Let $\category{T}$ be a tensor-triangulated category.
    \begin{enumerate}[label=(\alph*)]
    \item\label{item:prodCR} For any compact object $c \in \category{T}^c$ and any dualizable object $d \in \category{T}^d$, $c \otimes d$ is compact.
    \item\label{item:prodbyCC} Assume $\category{T}$ is both compactly generated and rigidly generated. Let $d$ be an object in $\category{T}$ such that $- \otimes d$ preserves compactness. Then $d$ is dualizable.
    \item\label{item:prodbyRC} Let $c$ be an object in $\category{T}$ such that $c \otimes -$ maps dualizable objects to compact ones. Then $c$ is compact.
    \end{enumerate}
\end{proposition}
\begin{proof}
    If $c$ is a compact object and $d$ is a dualizable object, then
    \[
        \Hom_{\category{T}}(c \otimes d, -) \cong \Hom_{\category{T}}(c, \SW{d} \otimes -)
    \]
    Both functors $\Hom_{\category{T}}(c, -)$ and $\SW{d} \otimes -$ preserve coproducts, hence $c \otimes d$ is compact. This proves \ref{item:prodCR}.

    Fix an object $d$ in $\category{T}$, such that $c \otimes d$ is compact for any compact object $c$. Then $\intHom(d, -)$ preserves coproducts. Indeed, fix a family of objects $(x_i)_i$. Then for any compact object $c$, we have a commutative diagram
		\[
				\begin{tikzcd}[column sep=small]
						\bigoplus_i \Hom_{\category{T}}(c, \intHom(d, x_i))
						\arrow[d, swap, "\sim"] \arrow[r, "\sim"]
						& \Hom_{\category{T}}(c, \bigoplus_i \intHom(d, x_i))
						\arrow[r, ]
						& \Hom_{\category{T}}(c, \intHom(d, \bigoplus_i x_i))
						\arrow[d, "\sim"] \\
						\bigoplus_i \Hom_{\category{T}}(c \otimes d, x_i)
						\arrow[rr, "\sim"]
						&&\Hom_{\category{T}}(c \otimes d, \bigoplus_i x_i),
				\end{tikzcd}
		\]
		where the vertical isomorphisms are given by the adjunction and the horizontal morphisms are the canonical ones. Two of these are  also isomorphisms as indicated, since $c$ and $c\otimes d$ are compact.
    So, if $\phi\colon \bigoplus_i \intHom(d, x_i)\to \intHom(d, \bigoplus_i x_i)$ is the canonical homomorphism, we have proved that $\Hom_{\category{T}}(c,\phi)$ is an isomorphism for each $c$ compact.
    Since $\category{T}$ is compactly generated, it follows that $\intHom(d, -)$ preserves coproducts, as claimed.
		
		Now fix a dualizable object $g$ of $\category{T}$ and consider the following commutative diagram:
    \[
        \begin{tikzcd}
            \intHom(d, \mathbb{1}) \otimes g \arrow[d, "\sim"] \arrow[rr] & & \intHom(d, g) \arrow[d, "\sim"] \\
            \intHom(\SW{g}, \intHom(d, \mathbb{1})) \arrow[r, "\sim"] &
            \intHom(\SW{g} \otimes d, \mathbb{1}) \arrow[r, "\sim"] & 
            \intHom(d, \intHom(\SW{g}, \mathbb{1}))
        \end{tikzcd}
    \]
    The top arrow must thus be an isomorphism. The natural transformation
    \[
        \intHom(d, \mathbb{1}) \otimes - \to \intHom(d, -)
    \]
    is a natural transformation between exact, coproducts-preserving functors and an isomorphism on a generating set of objects, hence it is an isomorphism. This proves~\ref{item:prodbyCC}.

    Statement \ref{item:prodbyRC} is a direct consequence of the fact that $\mathbb{1}$ is always dualizable.
\end{proof}

All the `large' tensor-triangulated categories considered in this article are compactly generated, and the compact objects are dualizable. In particular, the hypotheses of \Cref{prop:compactRigid}~\ref{item:prodbyCC} are satisfied.

\begin{definition}
A tensor-triangulated category $\category{T}$ is called \emph{rigidly-compactly generated} if its compact and dualizable objects coincide, $\category{T}^c = \category{T}^d$, and they generate~$\category{T}$.
\end{definition}

\begin{remark}
If the monoidal unit $\mathbb{1}$ of $\category{T}$ is compact, then $\category{T}^d\subset\category{T}^c$; see the proof of~\autocite[Theorem~2.1.3(d)]{hoveyPalmieriStricklandMemoir1997}.
So for tensor-triangulated $\category{T}$ to be compactly-rigidly  generated, it suffices to verify that it is compactly generated and $\mathbb{1}\in\category{T}^c\subset\category{T}^d$.
\end{remark}

\subsection{Sequential homotopy colimits}\label{subsec:hocolim} 

If $\category{T}$ be a compactly generated triangulated category, we can define the sequential homotopy colimit $\hocolim_i x_i$ of a chain morphisms
\[
  x_0 \xrightarrow{\phi_0} x_1 \xrightarrow{\phi_1} x_2 \xrightarrow{\phi_2} \cdots
\]
as a cone of the morphism $\id-\phi\colon \bigoplus_i x_i\to\bigoplus_i x_i$, see~\autocite[Definition~2.2.3]{hoveyPalmieriStricklandMemoir1997} or~\autocite[\S1.6]{neemanBook2001}. Of course, one cannot expect this construction to be functorial, but using the notion of minimal weak colimits from~\autocite[Definition~2.2.1]{hoveyPalmieriStricklandMemoir1997}, we can see that it is determined uniquely up to a non-canonical isomorphism just by the underlying category of $\category{T}$, independently of the triangulation.

\begin{definition}
If $F\colon I\to\category{T}$ is a diagram, we say that a cocone $(f_i\colon F(i)\to x)$ is a \emph{weak colimit} if for any other cocone $(g_i\colon F(i)\to y)$, there exists a (not necessarily unique) morphism $h\colon x\to y$ such that $g_i=hf_i$ for each $i$. In other words, we require that the canonical morphism
\[ \Hom_\category{T}(x,y)\to\limit_{i\in I}\Hom_\category{T}(F(i),y) \]
is surjective for each $y\in\category{T}$.
A weak colimit $(f_i\colon F(i)\to x)$ is said to be \emph{minimal} if it induces for each $c\in\category{T}^c$ an isomorphism of abelian groups
\[ \colim_{i\in I}\Hom_\category{T}(c,F(i))\to\Hom_\category{T}(c,x). \]
\end{definition}

\begin{remark}
\autocite[Definition~2.2.1]{hoveyPalmieriStricklandMemoir1997} appears more restrictive as it requires that $\colim_{i\in I} H(F(i))\to H(x)$ be an isomorphism for each coproduct preserving homological functor $H\colon\category{T}\to\category{Ab}$. However, this is actually an equivalent condition by~\autocite[Corollary 2.4]{krauseTelescope2000}.
\end{remark}

\begin{lemma} Let $F\colon I\to\category{T}$ be a diagram in compactly generated triangulated category.
\begin{enumerate}
\item If $(f_i\colon F(i)\to x)$ and $(g_i\colon F(i)\to x)$ are both minimal weak colimits, there exists a (non-canonical) isomorphism $h\colon x\to y$ such that $g_i=hf_i$ for each~$i$.
\item
If $x_0 \xrightarrow{\phi_0} x_1 \xrightarrow{\phi_1} x_2 \xrightarrow{\phi_2} \cdots$ is a sequence of maps and $(f_i\colon x_i\to\hocolim x_i)$ is the cocone coming from a defining triangle, then $(f_i\colon x_i\to\hocolim x_i)$ is a minimal weak colimit.
\end{enumerate}
\end{lemma}

\begin{proof}
See~\autocite[Proposition~2.2.4(c) and~(d)]{hoveyPalmieriStricklandMemoir1997}.
\end{proof}

Suppose now that we have a commutative diagram in $\category{T}$ of the shape $\mathbb{N}\times\mathbb{N}$:
\[
\begin{tikzcd}
x_{00} \arrow[r, "\phi_{00}"] \arrow[d, "\psi_{00}"] & x_{10} \arrow[r, "\phi_{10}"] \arrow[d, "\psi_{10}"] & x_{20} \arrow[r, "\phi_{20}"] \arrow[d, "\psi_{20}"] & \cdots
\\
x_{01} \arrow[r, "\phi_{01}"] \arrow[d, "\psi_{01}"] & x_{11} \arrow[r, "\phi_{11}"] \arrow[d, "\psi_{11}"] & x_{21} \arrow[r, "\phi_{21}"] \arrow[d, "\psi_{21}"] & \cdots
\\
x_{02} \arrow[r, "\phi_{02}"] \arrow[d, "\psi_{02}"] & x_{12} \arrow[r, "\phi_{12}"] \arrow[d, "\psi_{12}"] & x_{22} \arrow[r, "\phi_{20}"] \arrow[d, "\psi_{22}"] & \cdots
\\
\vdots & \vdots & \vdots
\end{tikzcd}
\]
Denoting $z_{i}=\hocolim_j x_{ij}$, we obtain a compatible chain of morphisms
\begin{equation}\label{eq:partial-hocolim}
\begin{tikzcd}
z_{0} \arrow[r, "\chi_{0}"] & z_{1} \arrow[r, "\chi_{1}"] & z_{2} \arrow[r, "\chi_{2}"] & \cdots
\end{tikzcd}
\end{equation}
and a cocone $(f_{ij}\colon x_{ij}\to\hocolim z_i)$ given by the compositions $x_{ij}\to z_i\to \hocolim_j z_j$. The cocone of course restricts to the cocone of the diagonal $(f_{ii}\colon x_{ii}\to\hocolim z_i)$.
On the other hand, we can consider the diagonal $x_{00}\to x_{11}\to\cdots$.

\begin{lemma} 
For any choice of minimal weak colimit cocones and a compatible chain~\eqref{eq:partial-hocolim} as above,
there is a (non-canonical) isomorphism $h\colon \hocolim_i x_{ii}\overset{\sim}\to\hocolim_i z_i$ such that $f_{ii}=hg_i$ for each $i$.
\end{lemma}

\begin{proof}
Since $(g_i\colon x_{ii}\to \hocolim_i x_{ii})$ is a minimal weak colimit cocone, there exists a (possibly non-unique) morphism $h\colon \hocolim_i x_{ii}\to\hocolim_i z_i$ such that $f_{ii}=hg_i$ for each $i$.
Note that if $c\in\category{T}^c$, then there are canonical isomorphisms
\begin{multline*}
\colim_{i}\Hom_\category{T}(c,x_{ii})\cong
\colim_{i,j}\Hom_\category{T}(c,x_{ij})\cong
\colim_{i}\colim_{j}\Hom_\category{T}(c,x_{ij})\cong
\\
\cong
\colim_{i}\Hom_\category{T}(c,z_{i})\cong
\Hom_\category{T}(c,\hocolim_{i}z_i).
\end{multline*}
Since $(g_i\colon x_{ii}\to \hocolim_i x_{ii})$ is minimal, we obtain that $\Hom_\category{T}(c,h)$ is an isomorphism for each $c\in\category{T}^c$. Hence, $h$ is an isomorphism.
\end{proof}

If $\category{T}$ is in addition a tensor-triangulated category, we can deduce the following compatibility between the tensor product and sequential homotopy colimits.

\begin{lemma}\label{lemma:tensor-and-hocolim}
If
$x_0 \xrightarrow{\phi_0} x_1 \xrightarrow{\phi_1} x_2 \xrightarrow{\phi_2} \cdots$ and 
$y_0 \xrightarrow{\psi_0} y_1 \xrightarrow{\psi_1} y_2 \xrightarrow{\psi_2} \cdots$
are two sequences of morphisms in $\category{T}$, then there is a (non-canonical) isomorphism
\[ \hocolim (x_i\otimes y_i)\cong(\hocolim x_i)\otimes(\hocolim y_j). \]
Moreover, when $(f_i\colon x_i\to\hocolim_i x_i)$ and $(g_i\colon y_i\to\hocolim_i y_i)$ are two minimal weak colimit cocones, then the tensor product
\[ (f_i\otimes g_i\colon x_i\otimes y_i\to(\hocolim x_i)\otimes(\hocolim y_j)) \]
is a minimal weak colimit cocone.
\end{lemma}

\begin{proof}
This follows by applying the previous lemma to the $\mathbb{N}\times\mathbb{N}$-indexed diagram $x_i\otimes y_j$ in $\category{T}$ with the specific choice of~\eqref{eq:partial-hocolim} given by $z_i=x_i\otimes\hocolim_j y_j$ and $\chi_i=\phi_i\otimes\hocolim_j y_j$ for each $i$.
\end{proof}

\subsection{Action of a graded ring and Noetherian objects}
Let $\category{T}$ be a tensor-triangulated category and $R$ a graded-commutative Noetherian ring. A map of graded rings $R \to \emorp(\mathbb{1})$ endows every graded Hom-set with a structure of graded $R$-module:
\[
    R \otimes_{\mathbb{Z}} \Hom_{\category{T}}(x, y) \to \Hom_{\category{T}}(\mathbb{1}, \mathbb{1}) \otimes_{\mathbb{Z}} \Hom_{\category{T}}(x, y) \xrightarrow{\otimes} \Hom_{\category{T}}(x, y)
\]
Under these actions, the compositions are $R$-linear.
If $\category{T}$ is compactly generated, we can follow the idea from~\autocite[\S4]{bensonLocallyDualisableModular2024} of transferring finiteness properties from $R$-modules to objects of $\category{T}$ using homomorphism $R$-modules from the compacts.

\begin{definition}\label{def:finiteness-conditions}
If $\category{T}$ is a compactly generated tensor-triangulated category with an action of a graded-commutative Noetherian ring $R$, we say that:
\begin{itemize}
\item an object $x\in\category{T}$ is \emph{Noetherian} if $\Hom_\category{T}(c,x)$ is a Noetherian $R$-module for each $c\in\category{T}^c$,
\item the category $\category{T}$ is \emph{Noetherian} if all compact objects are Noetherian, i.e.\ $\Hom_{\category{T}}(c, c')$ is Noetherian over $R$ for any pair of compact objects $c, c' \in \category{T}^c$;
\item \emph{strictly Noetherian} if the classes of compact and Noetherian objects coincide.
\end{itemize}
\end{definition}

\begin{remark}
If $\category{T}$ has a compact generator $g\in\category{T}^c$, then $x\in\category{T}$ is Noetherian if and only if $\Hom_\category{T}(g,x)$ is a Noetherian $R$-module.
\end{remark}

\begin{example}\label{example:strictly-finite=regular}
It is instructive to consider the $\category{T}=\category{D}(\catMod R)$, where $R$ is commutative Noetherian, as an example. Then $\category{T}^c$ is known to coincide with the class $\category{D}^\per(R)$ of perfect complexes, while the class of Noetherian objects is easily seen to coincide with $\category{D}^\bnd(\catmod R)$. Hence $\category{T}=\category{D}(\catMod R)$ is strictly Noetherian if and only if each finitely generated $R$-modules has finite projective dimension, which is equivalent to $R$ being regular (see e.g.\ \autocite[Theorem~20.2.11]{foxbyHolmJoergensen2024}).
\end{example}

\subsection{Abstract torsion and completion}\label{subsec:torsion-and-completion} Let $\category{T}$ be a tensor-triangulated category, with arbitrary coproducts, and endowed with an action of a graded-commutative Noetherian ring $R$. We assume $\category{T}$ is compactly-rigidly  generated. For any specialization closed subset $\mathcal{V}$ of $\Spec(R)$, one can consider the full subcategory of $\mathcal{V}$-torsion objects of~$\category{T}$,
\[
    \Gamma_{\mathcal{V}}\category{T} = \{ x \in \category{T} \: | \:
        \forall c \in \category{T}^c, \forall \ideal{p} \in \Spec(R) \backslash \mathcal{V},
        \Hom_{\category{T}}(c, x)_{\ideal{p}} = 0
    \}.
\]
Its right Hom-orthogonal is the category of $\mathcal{V}$-local objects
\[
    L_{\mathcal{V}}\category{T} = \{ x \in \category{T} \: | \:
        \forall t \in \Gamma_{\mathcal{V}}\category{T},
        \Hom_{\category{T}}(t, x) = 0
    \}.
\]
The categories $\Gamma_{\mathcal{V}}(\category{T})$ and $L_{\mathcal{V}}(\category{T})$ fit in a recollement; see~\autocite[Proposition~2.3]{bensonColocalizingSubcategoriesCosupport2012} (or also \autocite[\S3.3]{hoveyPalmieriStricklandMemoir1997} or~\autocite[\S2]{greenleesTate2001}):
\[
    \begin{tikzcd}
        \Gamma_{\mathcal{V}}\category{T} 
            \arrow[r, bend left,  tail, "i_!"]
            \arrow[r, bend right, tail, "i_*"'] &
        \category{T}
            \arrow[l, "i^*"'] 
            \arrow[r, bend left, "p^*"]
            \arrow[r, bend right, "p^!"'] &
        L_{\mathcal{V}}\category{T}
            \arrow[l, tail, "p_*"'] 
    \end{tikzcd}
\]
where $i_!$ and $p_*$ are the canonical inclusions.
We will focus our attention on the endofunctors of $\category{T}$:
\[
    \Gamma_{\mathcal{V}} = i_!i^* \qquad L_{\mathcal{V}} = p_*p^* \qquad V^{\mathcal{V}} = p_*p^! \qquad \Lambda^{\mathcal{V}} = i_*i^*
\] 
In particular, we have $\Gamma_\mathcal{V}\category=\im i_!$ and $L_\mathcal{V}=\im p_*$. We will denote the essential image of the third fully faithful functor in the recollement by
\[
\Lambda^\mathcal{V}\category{T} = \im i_*.
\]
Inspired by~\autocite{greenleesMayLocalHomology1992} and also by similar terminology of other references of ours, we name the objects in the just defined classes as follows.

\begin{definition}
The objects in $\Gamma_{\mathcal{V}}\category{T}$ are called \emph{$\mathcal{V}$-torsion}, the objects of $L_{\mathcal{V}}\category{T}$ are said to be \emph{$\mathcal{V}$-local}, and the objects of $\Lambda^{\mathcal{V}}\category{T}$ will be called \emph{$\mathcal{V}$-complete}.
\end{definition}

There is an adjunction $\Gamma_{\mathcal{V}} \dashv \Lambda^{\mathcal{V}}$. Moreover, for any $x \in \category{T}$, the pair $(\Gamma_{\mathcal{V}}x, L_{\mathcal{V}}x)$ is the unique pair $(x^{\text{tor}}, x^{\text{loc}})$, with $x^{\text{tor}} \in \Gamma_{\mathcal{V}}\category{T}$ and $x^{\text{loc}} \in L_{\mathcal{V}}X$, fitting in an exact triangle
\[
    x^{\text{tor}} \to x \to x^{\text{loc}} \to \Sigma x^{\text{tor}}.
\]
Similarly, the pair $(V^{\mathcal{V}}x, \Lambda^{\mathcal{V}}x)$ is the unique pair $(x^{\text{coloc}}, x^{\text{comp}})$, with $x^{\text{celoc}} \in L_{\mathcal{V}}\category{T}$ and $x^{\text{comp}} \in \Lambda^{\mathcal{V}}\category{T}$, fitting in an exact triangle
\[
    x^{\text{coloc}} \to x \to x^{\text{comp}} \to \Sigma x^{\text{coloc}}.
\]

The calculus with these functors is controlled by certain tensor idempotents in~$\category{T}$ by \autocite[\S3]{balmerFaviTelescope2011}, which allows us to easily deduce certain rules of interaction between these functors. Most notably, if $\mathcal{V},\mathcal{W}\subset\Spec{R}$ are two closed subsets, then there are natural isomorphisms
\[
\Gamma_{\mathcal{V}}\Gamma_{\mathcal{W}}\cong
\Gamma_{\mathcal{V}\cap\mathcal{W}} \cong
\Gamma_{\mathcal{W}}\Gamma_{\mathcal{V}}
\quad\text{and}\quad
\Gamma_{\mathcal{V}}L_{\mathcal{W}}\cong
L_{\mathcal{W}}\Gamma_{\mathcal{V}},
\]
and by passing to right adjoints, we also obtain
\[
\Lambda^{\mathcal{V}}\Lambda^{\mathcal{W}}\cong
\Lambda^{\mathcal{V}\cap\mathcal{W}} \cong
\Lambda^{\mathcal{W}}\Lambda^{\mathcal{V}}
\quad\text{and}\quad
\Lambda^{\mathcal{V}}V^{\mathcal{W}}\cong
V^{\mathcal{W}}\Lambda^{\mathcal{V}}.
\]

The closed monoidal structure on $\category{T}$ also induces one on both $\Gamma_\mathcal{V}\category{T}$ and $\Lambda^\mathcal{V}\category{T}$.

\begin{proposition}\label{prop:tensorCompTor}
    For any $\mathcal{V} \subset \Spec(R)$ specialization closed:
    \begin{enumerate}[label=(\alph*)]
    \item\label{item:tensorCompTor-torsion} The category $\Gamma_{\mathcal{V}}\category{T}$ is a compactly generated $\otimes$-ideal of $\category{T}$. Its tensor product is the restriction of the one of $\category{T}$, and it has an internal Hom-functor $\intHom_{\Gamma_{\mathcal{V}}\category{T}}(x, y) = \Gamma_{\mathcal{V}} \intHom_{\category{T}}(x,y)$. The class of compact objects is given by $(\Gamma_\mathcal{V}\category{T})^c=\category{T}^c\cap\Gamma_\mathcal{V}\category{T}$.
    \item\label{item:tensorCompTor-complete} The category $\Lambda^{\mathcal{V}}\category{T}$ is a compactly generated thick subcategory of $\category{T}$. It has a tensor product $x \otimes_{\Lambda^{\mathcal{V}}\category{T}} y = \Lambda^{\mathcal{V}}(x \otimes_{\category{T}} y)$ and an internal Hom-functor $\intHom_{\Lambda^{\mathcal{V}}\category{T}}(x,y) = \intHom_{\category{T}}(x,y)$.
    \item\label{item:tensorCompTor-dualizables} For any compact object $c \in \category{T}^c$, $\Gamma_{\mathcal{V}}c$ is dualizable in $\Gamma_{\mathcal{V}}\category{T}$ and $\Lambda^{\mathcal{V}}c$ is dualizable in $\Lambda^{\mathcal{V}}\category{T}$.
    \end{enumerate}
\end{proposition}

\begin{proof}
Part~\ref{item:tensorCompTor-torsion} and the claim about torsion objects in part~\ref{item:tensorCompTor-dualizables} can be found in~\autocite[Proposition~3.6]{bensonLocallyDualisableModular2024}.
Part~\ref{item:tensorCompTor-complete} and the other half of~\ref{item:tensorCompTor-dualizables} is obtained by transporting the closed monoidal structure on $\Gamma_{\mathcal{V}}\category{T}$ using the category equivalence $\Gamma_{\mathcal{V}}\category{T}\simeq\Lambda^{\mathcal{V}}\category{T}$ from \autocite[Proposition 2.4]{bensonLocallyDualisableModular2024} and the natural isomorphism $\Lambda^{\mathcal{V}}\Gamma_\mathcal{V}\cong\Lambda^{\mathcal{V}}$ from~\autocite[Proposition~2.3]{bensonColocalizingSubcategoriesCosupport2012}. This is completely analogous to~\autocite[Corollary~3.8]{bensonLocallyDualisableModular2024}.
\end{proof}

\begin{remark}\label{remark:torsion-compacts}
In part~\ref{item:tensorCompTor-complete} of the previous proposition, we have the inclusion
$(\Lambda^\mathcal{V}\category{T})^c\subset\category{T}^c\cap\Lambda^\mathcal{V}\category{T}$.
This follows from the inclusion $(\Lambda^\mathcal{V}\category{T})^c=(\Gamma_\mathcal{V}\category{T})^c\subset\category{T}^c$ as subcategories of $\category{T}$ by the same argument as for~\autocite[Lemma~2.7]{bensonLocallyDualisableModular2024}; see also \Cref{prop:compactKoszul} below.
However, unlike in part~\ref{item:tensorCompTor-torsion}, this is not an equality in general. 
We will see an example of this phenomenon in the modular representation theory of finite groups in \S\ref{subsec:KInjkG}.
\end{remark}

Note that by \autocite[Proposition 2.4]{bensonLocallyDualisableModular2024} and the above argument, $\Gamma_{\mathcal{V}}$ and $\Lambda^{\mathcal{V}}$ are idempotent and restrict to mutual inverse monoidal equivalence

\[
    \begin{tikzcd}
        \Gamma_{\mathcal{V}}\category{T} 
            \arrow[r, "\Lambda^{\mathcal{V}}", "\sim"', shift left=0.12cm] &
        \Lambda^{\mathcal{V}}\category{T}.
            \arrow[l, "\Gamma_{\mathcal{V}}", shift left=0.12cm]
    \end{tikzcd}
\]

\subsection{Bounded generation}

We fix a triangulated category $\category{T}$ with arbitrary coproducts. For any full subcategory $\category{S} \subset \category{T}$, we denote by $\add(\category{S})$ the full subcategory of $\category{T}$ closed under finite coproducts and summands, and by $\Add(\category{S})$ the full subcategory of $\category{T}$ closed under arbitrary coproducts and summands. If $\category{A}$ and $\category{B}$ are two full subcategories, we denote by $\category{A} * \category{B}$ the full subcategory whose objects are the $t \in \category{T}$ fitting in an exact triangle $a \rightarrow t \rightarrow b \rightarrow \Sigma a$, with $a \in \category{A}$ and $b \in \category{B}$. For a full subcategory $\category{S} \subset \category{T}$, we define inductively $\thick^n(\category{S})$ by
\[
    \thick^0(\category{S}) = \add(\Sigma^i\category{S}, i \in \mathbb{Z}), \qquad \thick^{n+1}(\category{S}) = \add(\thick^0(\category{S}) * \thick^n(\category{S}))
\]
and $\Loc^n(\category{S})$
\[
    \Loc^0(\category{S}) = \Add(\Sigma^i\category{S}, i \in \mathbb{Z}), \qquad \Loc^{n+1}(\category{S}) = \Add(\Loc^0(\category{S}) * \Loc^n(\category{S})).
\]

\begin{definition} \label{definition:strong-generator}
    Let $\category{T}$ be a triangulated category. An object $g \in \category{T}$ is a \emph{strong generator} if there is an integer $D$ such that 
    \[
        \thick^D(g) = \category{T}.
    \]
\end{definition}

\begin{construction}\label{constr:genGammaGen}
    Let $\category{T}$ be an $R$-linear triangulated category. Fix $g \in \category{T}$ such that $\emorp(g)$ is finitely generated over $R$. We denote
    \[
        \category{F} = \{ x \in \category{T} \: | \: \Hom(g, x) \text{ is finitely generated} \}
    \]
    
    Let $x \in \category{F}$. Then there is a sequence of maps
    \[
        x = x_0 \xrightarrow{\phi_0} x_1 \xrightarrow{\phi_1} x_2 \xrightarrow{\phi_2} \cdots
    \]
    such that:
    \begin{itemize}
    \item for any $i \in \mathbb{N}$, $x_i$ belongs to $\category{F}$;
    \item for any $i \in \mathbb{N}$, $\Hom(\thick^0(g), \phi_i) = 0$;
    \item for any $i \in \mathbb{N}$, $\cone(\phi_i)$ belongs to $\thick^0(g)$.
    \end{itemize}
\end{construction}
\begin{proof}
    We construct the objects $x_i$ and maps $\phi_i$ inductively starting with $x_0 = x$. Assume, by induction, that $x_i \in \category{F}$ has been constructed. Let $f_1, \ldots, f_n$ be generators of the $R$-module $\Hom(g, x_i)$. We consider the triangle:
    \[
        \bigoplus_{k = 0}^n \Sigma^{-|f_k|}g \xrightarrow{(f_k)_k} x_i \xrightarrow{\phi_i} x_{i+1} \rightarrow
    \]
    Clearly $x_{i+1}$ belongs to $\category{F}$ and $\cone(\phi_i)$ to $\thick^0(g)$. Moreover, any map $h \colon g \to x_i$ factors through $(f_k)_k$, thus $\phi_i \circ h = 0$. Hence $\phi_i$ acts by zero on $\Hom(\thick^0(g), x_i)$.
\end{proof}

\begin{lemma}\label{lemma:constrGenGamma}
    In \Cref{constr:genGammaGen},
    \begin{enumerate}[label=(\arabic*)]
    \item\label{item:genGammaGen-orthogonal} for any $i \leq j$, $\Hom(\Loc^{j-i}(g), \phi_j\phi_{j-1}\ldots\phi_i) = 0$;
    \item\label{item:genGammaGen-thick} for any $i$, $\cone(\phi_i\ldots\phi_0)$ belongs to $\thick^i(g)$.   
    \end{enumerate}
\end{lemma}

\begin{proof}
Clearly $\Hom(\Loc^{0}(g), \phi_j) = 0$ for any $j$. The first part then follows by a standard argument, see~\autocite[Proposition~3.3]{christensenIdeals1998}. The second part follows from the octahedral axiom.
\end{proof}

\begin{proposition}\label{prop:finHomGen}
    Let $\category{T}$ be a $R$-linear triangulated category. Let $\category{U}$ be a thick subcategory of $\category{T}$ and $g \in \category{U}$ be such that:
    \begin{enumerate}
    \item $\category{U} \subset \bigcup_{n \in \mathbb{N}}\Loc^n(g)$
    \item for any $u \in \category{U}$, $\Hom(g,u)$ is finitely generated over $R$. 
    \end{enumerate}
    Then $\category{U} = \thick(g)$.

    Moreover, if $\category{U} \subset \Loc^D(g)$ for some $D \in \mathbb{N}$, then $g$ is a strong generator of $\category{U}$:
    \[
        \thick^D(g) = \thick(g) = \category{U}.
    \]
\end{proposition}
\begin{proof}
    Since $\category{U}$ is a thick subcategory of $\category{T}$ and $g \in \category{U}$, $\thick(g) \subset \category{U}$. Conversely, fix $x \in \category{U}$. Since $\Hom(g, x)$ is finitely generated, we obtain a sequence $(x_i, \phi_i)$ from \Cref{constr:genGammaGen}. Since $x$ belongs to $\Loc^n(g)$, for some $n \in \mathbb{N}$, $\phi_{n}\ldots\phi_{0}$ acts by zero on $\Hom(x,x)$ by \Cref{lemma:constrGenGamma}~\ref{item:genGammaGen-orthogonal}, hence $\phi_{n}\ldots\phi_{0}=0$. Thus $x$ is direct summand of $\Sigma^{-1}\cone(\phi_n\ldots\phi_0)$ and $\cone(\phi_n\ldots\phi_0) \in \thick^n(g)$ by \Cref{lemma:constrGenGamma}~\ref{item:genGammaGen-thick}. This proves the first part of the statement.

    The second part of the statement follows from the fact that, under the hypothesis $\category{U} \subset \Loc^D(g)$, $n$ can be uniformly chosen to be equal to $D$ in the previous paragraph. 
\end{proof}

\begin{proposition}\label{prop:strictlyRightNoetherian}
    Let $\category{T}$ be a $R$-linear, Noetherian, tensor-triangulated and compactly generated category. Suppose that $\category{T}^c$ has a strong generator $g$. Then $\category{T}$ is strictly Noetherian.
\end{proposition}
\begin{proof}
    Let $\category{N}$ be the full subcategory of Noetherian objects of $\category{T}$, i.e.\ those $y\in\category{T}$ such that for any $c \in \category{T}^c$, the $R$-module $\Hom_{\category{T}}(c,y)$ is Noetherian. We fix $x \in \category{N}$. From \Cref{constr:genGammaGen}, we build inductively a sequence
    \[
        x = x_0 \xrightarrow{\phi_0} x_1 \xrightarrow{\phi_1} x_2 \xrightarrow{\phi_2} \cdots
    \]
    with $x_i \in \category{N}$, $\cone(\phi_i) \in \thick^0(g)$, $\Hom_{\category{T}}(\thick^0(g), \phi_i) = 0$. Denoting $c_i = \Sigma^{-1}\cone(\phi_i \ldots \phi_0)$, $c_i$ is compact.
		
		We claim that $x = \hocolim c_i$. Indeed, there is a direct system of distinguished triangles of the form
		\begin{equation}\label{eq:systemOfTrianglesForGeneration}
				\begin{tikzcd}[column sep=3em]
						\Sigma^{-1}x_{i} \arrow[d, swap, "\Sigma^{-1}\phi_i"] \arrow[r] & c_{i-1} \arrow[d] \arrow[r, "\psi_{i-1}"] & x \arrow[d, equal] \arrow[r, "\phi_{i-1}\cdots\phi_0"] & x_{i} \arrow[d, "\phi_i"] \\
						\Sigma^{-1}x_{i+1} \arrow[r] & c_i \arrow[r, "\psi_i"] & x \arrow[r, "\phi_i\cdots\phi_0"] & x_{i+1},
				\end{tikzcd}
		\end{equation}
which induces a homomorphism $\psi\colon\hocolim_i c_i \to x$. Since $g\in\category{T}$ is a generator, it suffices to prove that $\Hom(g, \psi)$ is an isomorphism. To that end, notice that the canonical homomorphism $\colim_i \Hom(g, x_i) \to \Hom(g, \hocolim_i x_i)$ is an isomorphism by~\autocite[Lemma~5.8(1)]{beligiannisRelativeHA2000}, so $\Hom(g, \psi)$ identifies with $\colim_i (\psi_{i})_*$, where $(\psi_i)_*=\Hom(g, \psi_i)$. The latter map is an isomorphism since an application of $\Hom(g, -)$ to~\eqref{eq:systemOfTrianglesForGeneration} yields a direct system of exact sequences of abelian groups of the form
		\[
				\begin{tikzcd}
						\Hom(g, \Sigma^{-1}x_{i}) \arrow[d, swap, "0"] \arrow[r] & \Hom(g, c_{i-1}) \arrow[d] \arrow[r, "(\psi_{i-1})_*"] & \Hom(g, x) \arrow[d, equal] \arrow[r] & \Hom(g, x_{i}) \arrow[d, "0"] \\
						\Hom(g, \Sigma^{-1}x_{i+1}) \arrow[r] & \Hom(g, c_i) \arrow[r, "(\psi_{i})_*"] & \Hom(g, x) \arrow[r] & \Hom(g, x_{i+1}).
				\end{tikzcd}
		\]
This proves the claim.

    Now fix $D \in \mathbb{N}$, such that for any $i$, $c_i$ is in $\thick^D(g)$. Then, by definition of the homotopy colimit,
    \[ 
        x = \cone(\bigoplus_i c_i \to \bigoplus_i c_i) \in \Loc^{2D + 1}(g)
    \]
    Hence by \Cref{prop:finHomGen}, $x$ is in $\thick^{2D+1}(g)$, hence compact. This shows that $\category{T}$ is strictly Noetherian.
\end{proof}

\begin{example}\label{example:strongly-generated=regular-finite-dim}
As in \Cref{example:strictly-finite=regular}, it is instructive to understand what precisely \Cref{prop:strictlyRightNoetherian} says for $\category{T}=\category{D}(\catMod R)$, where $R$ is commutative Noetherian. We already know from the above-mentioned example that $\category{T}$ being strictly Noetherian is equivalent to $R$ being regular. On the other hand, $\category{T}^c$ being strongly generated is equivalent to $R$ being regular of finite Krull dimension. Indeed, $\category{T}^c$ has a strong generator if and only if $R$ has finite global dimension by~\autocite[Theorem~6]{stevensonRouquierGlobalDimension2025} (which corrects \autocite[Proposition~7.26]{rouquierDimensions2008}), and the condition that $R$ has finite global dimension is equivalent to $R$ being regular of finite Krull dimension by~\autocite[\href{https://stacks.math.columbia.edu/tag/00OE}{Tag 00OE}]{thestacksprojectauthorsStacksProject}.

All in all, we see that the implication \Cref{prop:strictlyRightNoetherian} cannot be reversed in general: $\category{T}^c$ having a strong generator is a stronger condition than $\category{T}$ being strictly Noetherian.
\end{example}

\subsection{Compact reflections and cogeneration}

In \Cref{prop:strictlyRightNoetherian}, we have given a criterion for recognition of compact objects in terms of smallness of Hom-groups from the compacts. Here we will discuss the possibility of recognizing compacts in terms of Hom-groups to the compacts.
As one might suspect, the situation is more complicated than in the previous section, but we obtain a criterion which is still useful for some classes of examples. The key observation is the following.

\begin{proposition} \label{prop:compact-reflection}
Let $\category{T}$ be an $R$-linear, Noetherian, compactly generated tensor-triangulated category. If $\category{T}^c$ has a strong generator $g$, then the following are equivalent for an object $x\in\category{T}$:
\begin{enumerate}[label=(\alph*)]
\item\label{item:left-noetherian_Tc} $\Hom_\category{T}(x,c)$ is a Noetherian $R$-module for each $c\in\category{T}^c$;
\item\label{item:left-noetherian_generator} $\Hom_\category{T}(x,g)$ is a Noetherian $R$-module;
\item\label{item:compact-reflection} there exists a triangle $y\to x\to c\to$ such that $c\in\category{T}^c$ and $\Hom(y,-)_{|\category{T}^c}=0$ (or in other words, $x$ admits a reflection with respect to $\category{T}^c\subset\category{T}$).
\end{enumerate}
\end{proposition}

\begin{proof}
    The fact that~\ref{item:left-noetherian_Tc} implies~\ref{item:left-noetherian_generator} is obvious, while the converse follows from the fact that $\category{T}=\thick(g)$. Clearly also~\ref{item:compact-reflection} implies~\ref{item:left-noetherian_generator} as then $\Hom_\category{T}(x,g)\cong\Hom_\category{T}(c,g)$ is Noetherian.

    It remains to prove that~\ref{item:left-noetherian_Tc} implies~\ref{item:compact-reflection}. 
    Fix $x \in \category{T}$ such that $\Hom_{\category{T}}(x,g)$ is finitely generated over $R$. The functor $\Hom_{\category{T}}(x, -)_{|\category{T}^c}$ is homological and objectwise finitely-generated over $R$. Then, by a result of Letz~\autocite{letzBrownRepresentabilityTriangulated2023}, which is a variant of~\autocite[Theorem~1.3]{bondalVanDenBergh2003} suitable for our setting, it is representable by a compact object $c \in \category{T}^c$: 
    \[
        \Hom_{\category{T}^c}(c, -) \cong \Hom_{\category{T}}(x, -)_{|\category{T}^c}.
    \]
    This isomorphism is induced by precomposition by a map $f \colon x \to c$, and by completing $f$ to a triangle, we obtain~\ref{item:compact-reflection}.
\end{proof}

\begin{corollary}\label{corollary:compactness-by-cogeneration}
Let $\category{T}$ be as in \Cref{prop:compact-reflection} and $x\in\category{T}$ be an object such that $\Hom_\category{T}(y,x)=0$ for each $y$ such that $\Hom(y,-)_{|\category{T}^c}=0$ (e.g.\ if $x$ is in the smallest triangulated subcategory of $\category{T}$ containing $\category{T}^c$ and closed under products). Then $x$ is compact if and only if $\Hom_\category{T}(x,g)$ is a Noetherian $R$-module.
\end{corollary}

The orthogonality condition on $x$ from the corollary is automatically satisfied if $\category{T}$ is compactly cogenerated in the following sense.

\begin{definition}\label{definition:compactly-cogenerated}
A triangulated category $\category{T}$ with coproducts is \emph{compactly cogenerated} if for any $y \in \category{T}$, we have $y=0$ whenever $\Hom_{\category{T}}(y, -)_{|\category{T}^c} = 0$.
\end{definition}

The obvious question is which tensor-triangulated categories are compactly cogenerated. On the positive side, we have the following classes of examples.

\begin{example}\label{examples:compactly-cogenerated}
Let $\category{T}=\category{D}(\catMod R)$, where $R$ is commutative Noetherian ring such that
\begin{enumerate}
\item $R$ is countable, or
\item $R$ is a finitely generated algebra over a field.
\end{enumerate}
Then the smallest colocalizing subcategory of $\category{T}$ (i.e.\ the smallest triangulated subcategory of $\category{T}$ closed under products) which contains $R$ is all of $\category{T}$. Indeed, Neeman \autocite{neemanColocalizing2011} classified colocalizing subcategories of $\category{T}=\category{D}(\catMod R)$ in terms of subsets of $\Spec(R)$. More precisely, each colocalizing subcategory is determined by the residue fields $k(\ideal{p})$, $\ideal{p}\in\Spec(R)$, which it contains, and for each subset of $\Spec(R)$ there exists a colocalizing subcategory containing precisely the residue fields parametrized by the elements of that subset. In particular, proving that the smallest colocalizing subcategory containing $R$ is all of $\category{T}$ reduces to proving that $\mathbf{R}\!\Hom_R(k(\ideal{p}),R)\ne 0$ for all $\ideal{p}\in\Spec(R)$ (see also \autocite[Remark~1.3]{nakamuraCosupports2019}). This was done in the first case above (and for some other classes of rings) in~\autocite[Theorem~1]{thompsonCosupportComputations2018} and for the second case in \autocite[Corollary~1.2]{nakamuraCosupports2019}.
\end{example}

On the negative side, there are many natural compactly generated (and often even tensor-triangulated, $R$-linear and strictly Noetherian) categories which are not compactly cogenerated.

\begin{example} ~
\begin{enumerate}
\item If $R$ is a complete local commutative Noetherian ring, then $\category{T}=\category{D}(\catMod R)$ is not compactly cogenerated unless $R$ is Artinian. Indeed, if $\ideal{p}\in\Spec(R)$ is non-maximal, then $\mathbf{R}\!\Hom_R(k(\ideal{p}),R)=0$ (and so $\Hom_\category{T}(k(\ideal{p}),c)=0$ for each $c\in\category{T}^c$) by \autocite[Proposition~4.19]{bensonColocalizingSubcategoriesCosupport2012} or \autocite[Example~5.3]{thompsonCosupportComputations2018}.

\item Suppose $R$ is a (not necessarily commutative) Noetherian ring and $\category{T}=\category{KInj}(R)$. If there exists a finitely generated $R$-module of infinite projective dimension, then $\category{KInj}(R)$ is not compactly cogenerated. This in particular applies to $R=\field G$, where $G$ is a finite group and $\field$ is a field of positive characteristic which divides the order of $G$ (see \S\ref{subsec:KInjkG} for a more detailed discussion of this situation).

Indeed, in this case Krause~\autocite[Proposition~2.3]{krauseStableDerived2005} proved that, up to homotopy equivalence, the compact objects of $\category{T}$ are precisely injective resolutions of bounded complexes of finitely generated modules. In particular, $\Hom(y,-)_{|\category{T}^c}=0$ for any acyclic complex of injectives $y\in\category{T}$. We only have to make sure that there is an acyclic complex $y$ which is not contractible, i.e.\ such that $y\ne 0$ in $\category{T}$. However, the full subcategory $\category{S}\subset\category{T}$ consisting of acyclic complexes is also compactly generated by~\autocite[Corollary~5.4]{krauseStableDerived2005} and the category $\category{S}^c$ of compacts is up to direct factors equivalent to the Verdier quotient $\category{D}^\bnd(\catmod R)/\category{D}^\per(\catmod R)$. So, if there exists a finitely generated module of infinite projective dimension, then $\category{D}^\bnd(\catmod R)/\category{D}^\per(\catmod R)\ne 0$, so $\category{S}\ne 0$ and, consequently, $\category{T}$ cannot be compactly cogenerated.
\end{enumerate}
\end{example}

\section{Torsion and completion at an ideal}\label{sec:torClosedIdeal}

Throughout this section, we fix a compactly-rigidly generated category $\category{T}$.
We assume that a graded commutative Noetherian ring $R$ acts on $\category{T}$, and $\category{T}$ is Noetherian under this action. We fix a closed subset $\mathcal{V}$ of $\Spec(R)$. Equivalently, it can be written, for some (not necessarily prime) ideal $\ideal{a}$ of $R$, as
\[
    \mathcal{V} = \mathcal{V}(\ideal{a}) = \{ \ideal{p} \in \Spec(R) \: | \: \ideal{a} \subset \ideal{p} \}
\]
We will use the shorthand notations $\Gamma_{\ideal{a}} = \Gamma_{\mathcal{V}(\ideal{a})}$, $L_{\ideal{a}} = L_{\mathcal{V}(\ideal{a})}$ and $\Lambda^{\ideal{a}} = \Lambda^{\mathcal{V}(\ideal{a})}$, and call the objects of $\category{T}$ in the essential images of these functors \emph{$\ida$-torsion}, \emph{$\ida$-local} and \emph{$\ida$-complete}, respectively.
Beware that when $\ideal{a}$ is a non-maximal prime ideal, there is an unfortunate clash of notations with those used in~\autocite{bensonLocallyDualisableModular2024} and our notation does not match with theirs.

\subsection{Koszul objects}\label{subsec:Koszul-obj}

Let $x$ be an object in $\category{T}$. For any $r \in R$ homogeneous of degree $d$, the Koszul object $x \kos r$ is the cone of the action of $r$ on $x$:
\[
    x \kos r = \cone(\Sigma^{-d}x \xrightarrow{r} x)
\]
For any sequence $\mathbf{r} = (r_1, \ldots, r_k)$ of homogeneous elements whose sum of degrees amounts to degree $d$, $x \kos \mathbf{r}$ is the iterated Koszul object $(\ldots(x \kos r_1) \kos r_2 \ldots ) \kos r_k$. Note that
\[
x \kos \mathbf{r} \cong x \otimes (\mathbb{1} \kos \mathbf{r}) \cong
x \otimes (\mathbb{1} \kos r_1) \otimes \cdots \otimes (\mathbb{1} \kos r_k)
\]
for any $x\in\category{T}$. Consequently, Koszul objects always come equipped with maps $x\to x\kos\mathbf{r}$ and $\Sigma^{d-k}x\kos\mathbf{r}\to x$ which can be chosen to be functorial in $x$. 
For any ideal $\ideal{a} = (r_1, \ldots, r_k)$, we write
\[
    x \kos \ideal{a} = x \kos (r_1, \ldots, r_k) \qquad x \kos \ideal{a}^{[n]} = x \kos (r_1^n, \ldots, r_k^n)
\]
Both of these notations depend on the choice of the sequence of generators of $\ideal{a}$. Therefore, we always implicitly assume that we have chosen a set of generators for any ideal that we use in the context of Koszul objects.

\begin{lemma} \label{lemma:koszul-generation-time}
Suppose that $\ida=(a_1,\ldots,a_k)\subset R$ is a homogeneous ideal generated by a sequence of $k$ elements and $x\in\category{T}$. Then $x\kos\ida\in\thick^{2^k-1}(x)$.
\end{lemma}

\begin{remark}
The constant $2^k-1$ can be very often improved considerably by~\autocite[Theorem~A]{letzStephanGenerationTime2025}, but this plays no role in our discussion here.
\end{remark}

\begin{proof}
This is an easy inductive argument on $k$, taking into account that $z\kos r\in\thick_1(z)$ for each $z\in\category{T}$ and $r\in R$.
\end{proof}

We will record the following consequence for a later use.

\begin{lemma}\label{lemma:koszul-split}
Let $x\in\category{T}$ and $\ida\subset R$ be an ideal acting by zero on $x$ (that is, $\ida\cdot\id_x=0$) and suppose that the chosen generating sequence for $\ida$ is of length $k$ and the degrees sum to $d$. Then for each $n\ge 2^k-1$, the above maps $x\to x\kos\ida^{[n]}$ and $\Sigma^{nd-k}x\kos\ida^{[n]}\to x$ are a split monomorphism and a split epimorphism, respectively.
\end{lemma}

\begin{proof}
Let $\mathbf{r} = (r_1, \ldots, r_k)$ be the chosen generating sequence for $\ida$ and denote $\mathbf{s} = (s_1, \ldots, s_k)$, where $s_i=r_i^{n}$. Then $\mathbf{s}$ has total degree $nd$ and generates the ideal $\ida^{[n]}$. Moreover, if $n\ge 2^k-1$, then $\ida^{[n]}$ acts by zero on all objects of $\thick^{2^k-1}(x)$ and in particular on the ``partial'' Koszul objects $x_l := x \otimes (\mathbb{1} \kos s_1) \otimes \cdots \otimes (\mathbb{1} \kos s_l)$ for $l\le k$. Hence all the maps $x_l\to x_{l+1}\cong x_l\kos s_{l+1}$ are split monomorphisms and the map $x \to x\kos\ida$ is a just a composition of them. The argument for $\Sigma^{nd-k}x\kos\ida^{[n]} \to x$ is similar.
\end{proof}

Koszul objects also provide a description of compacts in $\Gamma_{\ideal{a}}\category{T}$ and $\Lambda^{\ideal{a}}\category{T}$
\begin{proposition}\label{prop:compactKoszul}
    The compact objects of $\Gamma_{\ideal{a}}\category{T}$ and $\Lambda^{\ideal{a}}\category{T}$ coincide and
    \[
        (\Gamma_{\ideal{a}}\category{T})^c = (\Lambda^{\ideal{a}}\category{T})^c = \Gamma_{\ideal{a}}\category{T} \cap \category{T}^c = \thick(c \kos \ideal{a},  c \in \category{T}^c) 
    \]
    Moreover, for any $c \in (\Gamma_{\ideal{a}}\category{T})^c$, there exists $n \in \mathbb{N}$ such that
    \[
        \ida^n \cdot \id_c = 0
    \]
\end{proposition}
\begin{proof}
    This is proven in \autocite[Lemma 2.7]{bensonLocallyDualisableModular2024} for a point $\ideal{p} \in \Spec R$, but the proof holds as well for an arbitrary ideal $\ideal{a}$. 
		
		For the convenience of the reader, we mention that the argument for that $\Lambda^{\ida}(x \kos \ida) \cong x \kos \ida$ for any $x\in\category{T}$ can be streamlined to some extent. Indeed, suppose that $\ida=(r_1,\ldots,r_k)$ and $x\in\category{T}$, and we will prove that $\Lambda^{\ida}(x \kos \ida) \cong x \kos \ida$ by induction on $k$. For $k=1$, we use that for any $y\in L_\ida\category{T}$, the $R$-module $\Hom_\category{T}(y,x\kos\ida)$ is on the one hand naturally an $R[r_1^{-1}]$-module (since $r_1$ acts by an isomorphism on $y$) and on the other hand it is annihilated by a certain power of $r_1$ by \autocite[Lemma 5.11]{bensonLocalCohomologySupport2008}. So $\Hom_\category{T}(y,x\kos\ida)=0$ for each $y\in L_\ida\category{T}$, or in other words $x\kos\ida\in\Lambda^\ida\category{T}$.
		
If $k>1$, we put $\ida'=(r_1,\dots,r_{k-1})$ and $\idb=(r_k)$ and $z=x\kos\ida'$. As $\Lambda^{\ida'}(z)\cong z$ by the inductive hypothesis and as $\Lambda^{\ida'}$ is an exact functor, we also have $\Lambda^{\ida'}(z\kos\idb)\cong z\kos\idb$. Using the inductive hypothesis for $\idb$ this time, we obtain that
\[
\Lambda^\ida(x\kos\ida) \cong
\Lambda^\idb(\Lambda^{\ida'}(z\kos\idb)) \cong
\Lambda^\idb(z\kos\idb) \cong
z\kos\idb =
x\kos\ida,
\]
as required.
\end{proof}

The functors $\Gamma_{\ida}$ and $\Lambda^{\ida}$ can be presented as a homotopy colimit and limit, respectively, of Koszul objects, as pointed out in similar situations in~\autocite[Lemma~3.6]{greenleesMayCompletions1995} or~\autocite[Proposition~2.9]{bensonStratifyingTria2011}.

\begin{construction}\label{constr:transitionKoszul}
    Let $x$ be an object in $\category{T}$ and write $\ida = (a_1, \ldots, a_k)$, with $a_i$ of degree $d_i$. For each $m\ge 1$ and $i=1,\dots,k$, consider a morphism $\psi_m^i$ fitting into the diagram
\[
  \begin{tikzcd}
	  \Sigma^{md_i-1}\mathbb{1}
		  \arrow[r]
			\arrow[d, "a_i"] &
    \Sigma^{md_i-1}\mathbb{1} \kos a_i^m
		  \arrow[r]
			\arrow[d, dotted, "\psi_m^{i}"] &
	  \mathbb{1}
		  \arrow[r, "a_i^m"]
		  \arrow[d, equal] &
	  \Sigma^{md_i}\mathbb{1}
			\arrow[d, "a_i"] \\
	  \Sigma^{(m+1)d_i-1}\mathbb{1}
		  \arrow[r] &
    \Sigma^{(m+1)d_i-1}\mathbb{1} \kos a_i^{(m+1)}
		  \arrow[r] &
	  \mathbb{1}
		  \arrow[r, "a_i^{(m+1)}"] &
	  \Sigma^{(m+1)d_i}\mathbb{1}
  \end{tikzcd}
\]
Then, if we denote $d=\sum_{i=1}^kd_i$ the sum of the degrees of the chosen generators of $\ida$, the morphisms $\phi_m=x\otimes \psi_m^1\otimes\cdots\phi_m^k$ form a compatible cocone of a countable direct system in $\category{T}$:
\begin{equation}\label{eq:Koszul-telescope}
  \begin{tikzcd}
	  \Sigma^{d-k} x\kos\ida
		  \arrow[r, "\phi_1"]
			\arrow[rrrrd, bend right=7, shift right=2] &
	  \Sigma^{2d-k} x\kos\ida^{[2]}
		  \arrow[r, "\phi_2"]
			\arrow[rrrd, bend right=3, shift right=1] &
	  \Sigma^{3d-k} x\kos\ida^{[3]}
		  \arrow[r, "\phi_3"]
			\arrow[rrd] &
		\cdots
		\\
		&&&& x
  \end{tikzcd}
\end{equation}
\end{construction}

\begin{proposition}\label{prop:torsionColimit}
    Let $x$ be an object in $\category{T}$ and fix the direct system and its cocone
    as in \Cref{constr:transitionKoszul}. Then homotopy colimit of the system~\eqref{eq:Koszul-telescope}
    together with a suitable map $\phi$ compatible with the cocone fits into a distinguished triangle
    \[
        \hocolim_m \Sigma^{md-k}x\kos\ida^{[m]} \overset{\phi}\to x \to l \to,
    \]
    where $\hocolim_m(\Sigma^{md-k}x\kos\ida^{[m]})\in\Gamma_{\ida}\category{T}$ and $l \in L_{\ida}\category{T}$.
    Consequently, the homotopy colimit is isomorphic to $\Gamma_{\ida}x$.
\end{proposition}
\begin{proof}
    We proceed by induction on $k\ge 1$. For $k=1$, this is~\autocite[Proposition~2.9]{bensonStratifyingTria2011}.
		So assume that $k \geq 1$, write $\ida' = (a_1, \ldots, a_{k-1})$, $\ida''=(a_k)$, and let $d'$ be the sum of the degrees of the generators of $\ida'$ and $d''$ be the degree of $a_k$. Note that then, by \Cref{lemma:tensor-and-hocolim}, we have
		\[
		  \hocolim_m \Sigma^{md-k}x\kos\ida^{[m]} \cong
			(\hocolim_m \Sigma^{md'-k+1}x\kos\ida'^{[m]}) \otimes (\hocolim_m \Sigma^{md''-1}\mathbb{1}\kos\ida''^{[m]}).
		\]
		By the inductive hypothesis, we have triangles
    \begin{align*}
        \hocolim_m \Sigma^{md'-k+1}x\kos\ida'^{[m]} \overset{\phi'}\to x \to l' \to,
				\\
        \hocolim_m \Sigma^{md''-1}\mathbb{1}\kos\ida''^{[m]} \overset{\phi''}\to \mathbb{1} \to l'' \to
    \end{align*}
with $\hocolim_m(\Sigma^{md'-k+1}x\kos\ida'^{[m]})\in\Gamma_{\ida'}\category{T}$, $\hocolim_m(\Sigma^{md''-1}\mathbb{1}\kos\ida''^{[m]})\in\Gamma_{\ida''}\category{T}$, $l'\in L_{\ida'}\category{T}$, and $l'\in L_{\ida''}\category{T}$. We claim that we can take $\phi=\phi'\otimes\phi''$. The domain of $\phi$ is contained in $(\Gamma_{\ida'}\category{T})\otimes(\Gamma_{\ida''}\category{T})\subset\Gamma_\ida\category{T}$ and we can express $\phi$ as the composition
\[ \phi=\phi'\circ((\hocolim_m \Sigma^{md'-k+1}x\kos\ida'^{[m]})\otimes\phi''). \]
Thus, the cone of $\phi$ fits by the octahedral axiom into a triangle of the form
\[  (\hocolim_m \Sigma^{md'-k+1}x\kos\ida'^{[m]})\otimes l''\to \cone(\phi) \to l' \to. \]
Since $L_{\ida''}\category{T}$ is a tensor ideal and both $L_{\ida''}\category{T}$ and $L_{\ida'}\category{T}$ are contained in $L_\ida\category{T}$, we have $l:=\cone(\phi)\in L_\ida\category{T}$, as required.
\end{proof}

\subsection{Power-torsion and derived complete modules}
\label{subsec:torsion-complete-modules}

The key task we will need to achieve to understand dualizable $\ida$-torsion or $\ida$-complete objects in $\category{T}$ is to get a better control of Hom-modules of $\Lambda^\ida\category{T}$. This is achieved via the concept of derived $\ida$-complete $R$-module, going back to Greenlees and May~\autocite{greenleesMayLocalHomology1992}. 
Our definition of derived complete modules is very similar to~\autocite{salch2023}, while a thorough and much more general discussion of derived completion and related homological conditions is given in \autocite{positselskiRemarksDerivedComplete2023}.
Along with this, we will discuss the more familiar theory of $\ida$-power torsion modules, to stress the analogies and point up differences.

\begin{definition}
    Let $M$ be a $R$-module.
    \begin{enumerate}[label=(\alph*)]
    \item $M$ is $\ida$-\emph{power torsion} if on any element $m \in M$, some power $n$ of $\ida$ acts trivially: $\ida^n m = 0$.
    \item $M$ is $\ida$-\emph{complete} if the canonical map
    \[
        M \to \limit_n M/\ida^nM
    \]
    is an isomorphism.
    \item $M$ is \emph{derived $\ida$-complete} if it is a cokernel of a morphism of $\ida$-complete modules.
    \end{enumerate}
\end{definition}

The following characterization of power torsion modules is well known. We will denote by $\comp{R}=\limit_n{R/\ida^n}$ the usual $\ida$-adic completion of $R$.

\begin{lemma}\label{lemma:powerTorsionSupport}
    Let $M$ be an $R$-module. The following statements are equivalent:
    \begin{enumerate}[label=(\roman*)]
    \item $M$ is $\ida$-power torsion;
    \item for all $\ideal{p} \in \Spec(R) \backslash \mathcal{V}(\ida)$, $M_{\ideal{p}} = 0$;
    \item for all $r \in \ida$, $\Tor_{*}^R(R[r^{-1}], M) = 0$.
    \item for all $r \in \ida$, $R[r^{-1}]\otimes_R M = 0$.
    \end{enumerate}
		The class of $\ida$-power torsion modules is an hereditary torsion class (i.e.\ a Serre subcategory closed under coproducts) in $\catMod{R}$.
		
		Moreover, the $R$-module structure of any $\ida$-power torsion $R$-module uniquely extends to an $\comp{R}$-module structure, and the restriction functor induces an equivalence between the categories of $\ida\comp{R}$-power torsion $\comp{R}$-modules and $\ida$-power torsion $R$-modules.
\end{lemma}
\begin{proof}
    See~\autocite[\href{https://stacks.math.columbia.edu/tag/0953}{Tag 0953}]{thestacksprojectauthorsStacksProject}.
\end{proof}

Quite non-trivially, derived complete $R$-modules admit an almost perfectly formally dual characterization.
 We mostly rely on \autocite{positselskiDedualizingMGM2016,positselskiRemarksDerivedComplete2023} and explain where to find the claims of the following lemma there (see also \autocite[Propositions~3.10 and~7.2]{williamsonPolHtpyComplete2022} for a similar summary).

\begin{lemma}\label{lemma:derivedComplete}
    Let $M$ be an $R$-module. The following statements are equivalent:
    \begin{enumerate}[label=(\roman*)]
    \item\label{item:derived-a-complete} $M$ is derived $\ida$-complete;
    \item\label{item:a-contramodules-all-Ext} for all $r \in \ida$, $\Ext^{*}_R(R[r^{-1}], M) = 0$.
    \item\label{item:a-contramodules} for all $r \in \ida$, $\Hom_R(R[r^{-1}], M) = 0 = \Ext^{1}_R(R[r^{-1}], M) = 0$.
    \end{enumerate}
		The class of derived $\ida$-complete modules is an exact Abelian subcategory of $\catMod{R}$ which is closed under extensions and products.

		Moreover, the $R$-module structure of any derived $\ida$-complete $R$-module uniquely extends to an $\comp{R}$-module structure, and the restriction functor induces an equivalence between the categories of derived $\ida\comp{R}$-complete $\comp{R}$-modules and derived $\ida$-complete $R$-modules.
\end{lemma}

\begin{proof}
The equivalence between conditions \ref{item:derived-a-complete}--\ref{item:a-contramodules} follows from \autocite[\S2]{positselskiDedualizingMGM2016}, where the modules satisfying condition~\ref{item:a-contramodules} are called $\ida$-contramodule $R$-modules. Adopting this terminology temporarily, note that the class of $\ida$-contramodule modules contains all $\ida$-complete modules by the discussion at the beginning of \autocite[\S2]{positselskiDedualizingMGM2016}, and is coreflective by \autocite[Proposition~2.1]{positselskiDedualizingMGM2016}. In particular, this class is closed under taking cokernels and so contains all derived $\ida$-complete modules in our sense. On the other hand, given any $\ida$-contramodule module $M$, we can take a projective presentation $P_1\to P_0\to M\to 0$ and apply the $\ida$-contramodule coreflector, which we denote by $\Delta_\ida$, to obtain $\Delta_\ida(P_1)\to\Delta_\ida(P_0)\to \Delta_\ida(M)\cong M\to 0$. Since $\Delta_\ida(P)\cong\comp{P}$ for any projective (or ever flat) $R$-module $P$ by~\autocite[Lemma~2.5]{positselskiDedualizingMGM2016}, this shows the equivalence between~\ref{item:derived-a-complete} and~\ref{item:a-contramodules}. The equivalence between~\ref{item:a-contramodules-all-Ext} and~\ref{item:a-contramodules} follows from the fact that the localizations $R[r^{-1}]$ have projective dimension at most one over $R$.

The properties of the full subcategory of derived $\ida$-complete $R$-modules follows from \autocite[Proposition~1.1]{geigleLenzingPerpendicular1991}, using the fact that $R[r^{-1}]$ has projective dimension at most one over $R$ for any $r\in R$.

    Finally, the moreover part follows from \autocite[Proposition~1.5]{positselskiRemarksDerivedComplete2023}. Regarding the terminology used there, we only remark that each $\ida$-contramodule $R$-module is quotseparated (i.e.\ a quotient of an $\ida$-complete $R$-module) in our case since $R$ is Noetherian; cf.~\autocite[Corollary 3.7]{positselskiRemarksDerivedComplete2023}.
\end{proof}

Modules annihilated by a fixed power of $\ida$ are both $\ida$-power torsion and (derived) $\ida$-complete.

\begin{lemma}\label{lemma:annPowTorsionComplete}
    Let $M$ be an $R$-module such that $\ida^n M = 0$, for some $n \geq 1$. Then $M$ is $\ida$-power torsion and $\ida$-complete.
\end{lemma}
\begin{proof}
    The module $M$ is obviously $\ida$-power torsion.
    Furthermore, for any $k \geq n$, $M / \ida^kM = M$ and the transition map $M / \ida^{k+1}M \to M / \ida^{k}M$ is the identity. Hence the infinite linear system $(M/\ida^{k}M)_{k \geq 1}$ is equivalent to a finite subsystem, whose limit is $M$ by initiality. 
\end{proof}

A key property of derived $\ida$-complete modules is that they satisfy a variation of Nakayama's lemma. Note that it holds for \emph{all} derived-complete modules, without any assumption of finite generation.

\begin{lemma}[Nakayama for derived $\ida$-complete modules]\label{lemma:nakayamaDC}
    Let $M$ be a derived $\ida$-complete module. Then $M = 0$ if and only if $M / \ida M = 0$.
\end{lemma}

\begin{proof}
Since any derived $\ida$-complete module is a cokernel of completions of free modules, one case reduce the problem to the Nakayama lemma for $\ida$-complete modules; see \autocite[Corollary~0.3]{portaShaulYekutieli2015} for this argument.

Alternatively, one can use that if $M$ is derived $\ida$-complete, $r\in\ida$ and $m_0, m_1, \dots $ is any sequence of elements of $M$, one can naturally define infinite sums $z_k=\sum_{n\ge 0}r^nm_{n+k}$ for all $k\ge 0$. More precisely, the countable system of equations $z_k-rz_{k+1}=m_k$, $k\ge 0$, has a unique solution $z_0, z_1, z_2, \cdots \in M$; see \autocite[Lemma~2.1]{positselskiContraadjustedModulesContramodules2017}. Now the Nakayama lemma follows easily by~\autocite[Lemmas~3.2, 4.2 and Remark~4.3]{positselskiContraadjustedModulesContramodules2017}.
\end{proof}

\begin{corollary}\label{cor:derivedCompleteQuotNoeth}
    Assume $R$ is $\ida$-complete. Let $M$ be a derived $\ida$-complete $R$-module and $r$ be an element of $\ida$. If $M/rM$ is Noetherian, so is $M$. 
\end{corollary}
\begin{proof}
    Let $m_1, \ldots, m_k \in M$ be representatives of generators of $M / rM$, and let $N = \coker(R^k \xrightarrow{(m_1, \ldots, m_k)} M)$. We see that $N / rN = 0$, hence $N / \ida N = 0$. Since $M$ and $R$ are derived $\ida$-complete, $N$ is derived $\ida$-complete, thus $N = 0$; hence $M$ is finitely generated.
\end{proof}

Now we can characterize $\ida$-power torsion, but more importantly also $\ida$-complete objects of $\category{T}$ as follows.

\begin{theorem} \label{thm:charPTandDC}
Let $\category{T}$ be a compactly-rigidly generated tensor-triangulated category, Noetherian over a graded commutative ring Noetherian ring $R$.
    Let $y$ be an object in $\category{T}$. Then:
    \begin{enumerate}[label=(\alph*)]
    \item\label{item:charPowerTorsion} $y\in\Gamma_{\ida}\category{T}$ if and only if $\Hom_{\category{T}}(c, y)$ is $\ida$-power torsion for each $c\in\category{T}^c$;
    \item\label{item:charDerivedComplete} $y\in\Lambda^{\ida}\category{T}$ if and only if $\Hom_{\category{T}}(c, y)$ is derived $\ida$-complete for each $c\in\category{T}^c$ if and only if $\Hom_{\category{T}}(x, y)$ is derived $\ida$-complete for each $x\in\category{T}$.
\end{enumerate}
\end{theorem}
\begin{proof}
    Statement \ref{item:charPowerTorsion} is a direct consequence of \Cref{lemma:powerTorsionSupport}.

    In order to prove \ref{item:charDerivedComplete}, assume first that $y\in\Lambda^\ida\category{T}$ and $x\in\category{T}$, and we will prove that $\Hom_{\category{T}}(x,y)$ is derived $\ida$-complete. In order to do this, we fix an injective cogenerator $I$ of $\catMod(R)$. For any compact object $c \in \category{T}^c$, we consider the object $t_{c \kos \ida}(I)$ defined using Brown Representability:
    \[
        \Hom_{\category{T}}(-, t_{c \kos \ida}(I)) \cong \Hom_R(\Hom_{\category{T}}(c \kos \ida, -), I).
    \]
    These objects have been studied in \autocite{bensonColocalizingSubcategoriesCosupport2012}. 
		In particular, by an adaptation of \autocite[Proposition 5.4]{bensonColocalizingSubcategoriesCosupport2012}, the set $\{ t_{c \kos \ida}(I) \mid c \in \category{T}^c \}$ perfectly cogenerates $\Lambda^{\ida}\category{T}$ in the sense of~\autocite[\S1]{krauseBrownRepresentabilityCoherent2002}. Consequently, \autocite[Proposition 5.2]{bensonColocalizingSubcategoriesCosupport2012} tells us that the smallest triangulated subcategory of $\category{T}$ containing the set $\{ t_{c \kos \ida}(I) \mid c \in \category{T}^c \}$and closed under products is precisely $\Lambda^\ida\category{T}$.
		
 Since by \Cref{lemma:derivedComplete}, the category of derived $\ida$-complete modules is an extension closed exact Abelian subcategory of the category of $R$-modules which is also closed under products, it is sufficient to check that $\Hom_{\category{T}}(c', t_{c \kos \ida}(I))$ is (derived) $\ida$-complete for any compact objects $c, c' \in \category{T}^c$. The $R$-module $\Hom_{\category{T}}(c \kos \ida, c')$ is annihilated by a power of $\ida$, hence so is
    \[
        \Hom_{\category{T}}(c', t_{c \kos \ida}(I)) \cong \Hom_{R}(\Hom_{\category{T}}(c \kos \ida, c'), I)
    \]
    By \Cref{lemma:annPowTorsionComplete}, $\Hom_{\category{T}}(c', t_{c \kos \ida}(I))$ is $\ida$-complete.
    
		Suppose conversely that $y\in\category{T}$ is such that $\Hom_{\category{T}}(c,x)$ is derived $\ida$-complete for each $c\in\category{T}^c$, and consider a triangle
		\[ V^\ida y \to y \to \Lambda^\ida y \to \Sigma V^\ida y \]
		with $V^\ida y\in L_\ida\category{T}$ and $\Lambda^\ida y\in \Lambda^\ida\category{T}$ as in \S\ref{subsec:torsion-and-completion},
		and denote $z_0=V^\ida y$ for brevity.
		Since $\Hom_{\category{T}}(c,y)$ is derived $\ida$-complete by assumption and $\Hom_{\category{T}}(c,\Lambda^\ida y)$ is derived $\ida$-complete by the previous part, it follows that $\Hom_{\category{T}}(c,z_0)$ is derived $\ida$-complete. If we write $\ida=(a_1, \dots, a_k)$ and $z_i=z\kos(a_1,\dots, a_i)$ for all $1\le i\le k$ and apply $\Hom_{\category{T}}(c, -)$ to the exact triangles
    \[
        \Sigma^{-|a_i|}z_{i-1} \xrightarrow{a_i} z_{i-1} \rightarrow z_i \rightarrow \Sigma^{-|a_i|+1}z_{i-1},
    \]
    we obtain long exact sequences of derived $\ida$-complete $R$-modules
    \[
        \Hom_{\category{T}}(c, \Sigma^{-|a_i|}z_{i-1}) \xrightarrow{a_i} \Hom_{\category{T}}(c, z_{i-1}) \rightarrow \Hom_{\category{T}}(c, z_i) \rightarrow \Hom_{\category{T}}(c, \Sigma^{-|a_i|+1}z_{i-1}).
    \]

		Now $z_k=z_0\kos(a_1, \dots, a_k)\in\Gamma_\ida\category{T}\cap L_\ida\category{T}=\{0\}$, and using the Nakayama \Cref{lemma:nakayamaDC} and the exact sequences, we prove by induction for $i=k, k-1, \dots, 0$ that $\Hom_{\category{T}}(c, z_{i-1})=0$. In particular, $\Hom_{\category{T}}(c, z_0)=0$ for all $c\in\category{T}^c$, so $z_0=0$ and $y\cong \Lambda^\ida y\in\Lambda^\ida\category{T}$, as required.
\end{proof}

\begin{corollary}\label{cor:tPTcDC}
    Let $x, y$ be two objects in $\category{T}$ and $c$ be a compact object in $\category{T}^c$.
    \begin{enumerate}[label=(\alph*)]
    \item\label{item:torsionPowerTorsion} $\Hom_{\category{T}}(c, \Gamma_{\ida}y)$ is $\ida$-power torsion.
    \item\label{item:completeDerivedComplete} $\Hom_{\category{T}}(x, \Lambda^{\ida}y)$ is derived $\ida$-complete.
    \end{enumerate}
    In particular, the action of $R$ on $\Hom_{\category{T}}(c, \Gamma_{\ida}y)$ and $\Hom_{\category{T}}(x, \Lambda^{\ida}y)$ lifts uniquely to an action of $\comp{R}$.
\end{corollary}

\begin{proof}
    Parts \ref{item:torsionPowerTorsion} and \ref{item:completeDerivedComplete} are immediate from \Cref{thm:charPTandDC}.
    The action of $R$ uniquely lifts to an action of the completed ring $\comp{R}$ by \Cref{lemma:powerTorsionSupport,lemma:derivedComplete}.
\end{proof}

\begin{remark}\label{remark:homFromTorsionIsComplete}
    The adjunction isomorphism, for any $x,y \in \category{T}$,
    \[
        \Hom_{\category{T}}(x, \Lambda^{\ida}y) \cong \Hom_{\category{T}}(\Gamma_{\ida}x, y)
    \]
    is $R$-linear, hence $\Hom_{\category{T}}(\Gamma_{\ida}x, y)$ is derived $\ida$-complete, and the adjunction isomorphism induces an isomorphism of $\comp{R}$-modules by \Cref{lemma:derivedComplete}.
\end{remark}

Now we can use Nakayama's lemma to lift finite generation properties from $\ida$-complete objects in $\category{T}$.

\begin{proposition}\label{prop:complAllCompactNoeth}
    Let $x$ be an object in $\Lambda^{\ida}\category{T}$. The following statements are equivalent:
    \begin{enumerate}[label=(\roman*)]
    \item\label{stmt:complCompactNoeth} for any $c\in(\Lambda^{\ida}\category{T})^c$, $\Hom_{\category{T}}(c, x)$ is Noetherian over $R$;
    \item\label{stmt:allCompactNoeth} for any $c\in\category{T}^c$, $\Hom_{\category{T}}(c, x)$ is Noetherian over $\comp{R}$.
    \end{enumerate} 
\end{proposition}
\begin{proof}
    Assume \ref{stmt:complCompactNoeth} holds. Fix $c \in \category{T}^c$. Then $c \kos \ida$ is in $(\Lambda^\ida\category{T})^c$ and by assumption $\Hom_{\category{T}}(c \kos \ida, x)$ is Noetherian over $R$, thus over $\comp{R}$. By induction, it suffices to check that, if for some $r \in \ida$, $\Hom_{\category{T}}(c \kos r, x)$ is Noetherian over $\comp{R}$, so is $\Hom_{\category{T}}(c, x)$.
    From the exact triangle
    \[
        \Sigma^{-|r|}c \xrightarrow{r} c \rightarrow c \kos r \rightarrow \Sigma^{-|r|+1}c
    \]
    we obtain a long exact sequence of derived $\ida$-complete $\comp{R}$-modules
    \[
        \Hom_{\category{T}}(\Sigma c, x) \xrightarrow{r} \Hom_{\category{T}}(\Sigma^{-|r|+1}c, x) \to \Hom_{\category{T}}(c \kos r, x).
    \]
    Taking $M = \Hom_{\category{T}}(c, x)$, $M / rM$ is isomorphic to a submodule of the Noetherian $\comp{R}$-module $\Hom_{\category{T}}(c \kos r, x)$, hence is Noetherian itself. Thus by \Cref{cor:derivedCompleteQuotNoeth}, $M$ is Noetherian over $\comp{R}$.

    Conversely, assume \ref{stmt:allCompactNoeth} holds. Then for any compact object $c \in (\Lambda^{\ida}T)^c$, the module $\Hom_{\category{T}}(c, x)$ is annihilated by some power of $\ida$, by \Cref{prop:compactKoszul}. Thus, if a family of elements $\mathcal{F}$ generates $\Hom_{\category{T}}(c, x)$ as an $\comp{R}$-module, it also generates it as an $R$-module.    
\end{proof}

The dual proof holds for $\ida$-torsion modules.

\begin{proposition}\label{prop:torAllCompactNoeth}
    Let $x$ be an object in $\Gamma_{\ida}\category{T}$. Then the following propositions are equivalent:
    \begin{enumerate}[label=(\roman*)]
    \item for any $c \in (\Gamma_{\ida}\category{T})^c$, $\Hom_{\category{T}}(x, c)$ is finitely generated over $R$;
    \item for any $c \in \category{T}^c$, $\Hom_{\category{T}}(x, c)$ is finitely generated over $\comp{R}$.
    \end{enumerate}
\end{proposition}

\subsection{Finiteness on homomorphisms from dualizability}

In this section we collect finiteness properties involving of homomorphism modules involving torsion and complete dualizable objects.
In particular, we will show that the homomorphism groups between torsion or complete dualizable modules are finitely generated over the adic completion $\comp{R}$ of $R$.

\begin{proposition}\label{prop:dualNoethComp}
    Let $d \in (\Lambda^{\ida}\category{T})^d$ be a dualizable object in $\Lambda^{\ida}\category{T}$. For any compact object $c \in (\Lambda^{\ida}\category{T})^c$, $\Hom_{\category{T}}(c, d)$ and $\Hom_{\category{T}}(d,c)$ are Noetherian over $R$.
\end{proposition}
\begin{proof}
    Note that the compact object $c$ is dualizable in each of $\Gamma_{\ida}\category{T}$, $\Lambda^{\ida}\category{T}$ and $\category{T}$, and the three Spanier-Whitehead duals coincide. Hence $\SW{c} \in \Gamma_{\ida}\category{T} \cap \category{T}^c = (\Gamma_{\ida}\category{T})^c = (\Lambda^{\ida}\category{T})^c$. Moreover,
    \[
        \Hom_{\category{T}}(c, d) \cong \Hom_{\category{T}}(\Lambda^{\ida}\mathbb{1}, \SW{c} \otimes_{\Lambda^{\ida}\category{T}} d).
    \]
    Since $c' = \SW{c} \otimes_{\Lambda^{\ida}\category{T}} d$ is a product of a compact and a dualizable object of $\Lambda^{\ida}\category{T}$, it is itself compact in $\Lambda^{\ida}\category{T}$ by \Cref{prop:compactRigid}, and
    \[
        \Hom_{\category{T}}(\Lambda^{\ida}\mathbb{1}, c') 
        \cong \Hom_{\category{T}}(\Gamma_{\ida}\mathbb{1}, \Gamma_{\ida}c') 
        \cong \Hom_{\category{T}}(\mathbb{1}, c')
    \]
    since $c'$ is in $\Lambda^{\ida}\category{T}$. Moreover, both $\mathbb{1}$ and $c'$ are compact in $\category{T}$, hence $\Hom_{\category{T}}(\mathbb{1}, c')$ is Noetherian over $R$.
		
		The argument for $\Hom_{\category{T}}(d, c)$ is similar, using the isomorphism
    \[
        \Hom_{\category{T}}(d, c) \cong \Hom_{\category{T}}(\Lambda^{\ida}\mathbb{1}, \SW{d} \otimes_{\Lambda^{\ida}\category{T}} c) \cong \Hom_{\category{T}}(\mathbb{1}, \SW{d} \otimes c),
    \]
		and the fact that $\SW{d} \otimes c\in\category{T}^c$.
\end{proof}

\begin{corollary}\label{cor:dualNoethTor}
    Let $d \in (\Gamma_{\ida}\category{T})^d$ be a dualizable object in $\Gamma_{\ida}\category{T}$. For any compact object $c \in (\Gamma_{\ida}\category{T})^c$, $\Hom_{\category{T}}(c, d)$ and $\Hom_{\category{T}}(d, c)$ are Noetherian over $R$.
\end{corollary}
\begin{proof}
    Since $\Lambda^{\ida}d$ is dualizable in $\Lambda^{\ida}\category{T}$, for any $c \in (\Gamma_{\ida}\category{T})^c$,
    \[
        \Hom_{\category{T}}(c, d) \cong \Hom_{\category{T}}(\Lambda^{\ida}c, \Lambda^{\ida}d) \cong \Hom_{\category{T}}(c, \Lambda^{\ida}d)
    \]
    is Noetherian over $R$.
\end{proof}

As a consequence of the Nakayama \Cref{lemma:nakayamaDC} for derived complete $R$-modules, we have the following finiteness result.

\begin{corollary}\label{cor:finiteGenOverCompR}
    For any compact object $c \in \category{T}^c$ and dualizable torsion object $d \in (\Gamma_{\ida}\category{T})^d$, the $\comp{R}$-module $\Hom_{\category{T}}(\Gamma_{\ida}c, d)$ is finitely generated.
\end{corollary}

\begin{proof}
By adjunction, we can equivalently prove that $\Hom_{\category{T}}(c, \Lambda^{\ida}d)$ is finitely generated over $\comp{R}$. Since $\Lambda^{\ida}d$ is dualizable in $\Lambda^\ida\category{T}$ (with respect to the monoidal structure described in \Cref{prop:tensorCompTor}), $\Hom_{\category{T}}(c, \Lambda^{\ida}d)$ is Noetherian over $R$ for any $c$ in $(\Lambda^{\ida}\category{T})^c$ by~\Cref{prop:dualNoethComp}. The conclusion follows from \Cref{prop:complAllCompactNoeth}.
\end{proof}

Now we are in a position to prove that homomorphism groups between torsion or complete dualizable objects are finitely generated over the completed ring.

\begin{theorem}\label{thm:finite-gen-hom-dualizables}
    For any dualizable compact objects $d,d' \in (\Gamma_{\ida}\category{T})^d$, the $\comp{R}$-module $\Hom_{\category{T}}(d,d')$ is finitely generated. The same is true for $d,d' \in (\Lambda^{\ida}\category{T})^d$.
\end{theorem}
\begin{proof}
    There is an isomorphism of $\comp{R}$-modules:
    \[
        \Hom(d,d') \cong \Hom(\Gamma_{\ida}\mathbb{1}, \SW{d} \otimes d')
    \]
    By \Cref{cor:finiteGenOverCompR}, the second module is finitely generated over $\comp{R}$, hence so is the first one.
\end{proof}

\subsection{Completion for objects and for homomorphism modules}

Since the functors $\Gamma_\ida\colon\category{T}\to\Gamma_\ida\category{T}$ and $\Lambda^\ida\colon\category{T}\to\Lambda^\ida\category{T}$, they send objects of $\category{T}^c=\category{T}^d$ to the corresponding dualizable objects. So we have the following diagram of functors, which commutes up to natural equivalence:
\[
    \begin{tikzcd}
		    & \category{T}^c
				    \arrow[dl, "\Gamma_{\ida}"'] \arrow[dr, "\Lambda^{\ida}"]
		    \\
        (\Gamma_{\ida}\category{T})^d
            \arrow[rr, "\Lambda^{\ida}", "\sim"', shift left=0.12cm] &&
        (\Lambda^{\ida}\category{T})^d.
            \arrow[ll, "\Gamma_{\ida}", shift left=0.12cm]
    \end{tikzcd}
\]
We will show that the diagonal functors
induce the $\ida$-adic completions on homomorphism modules.

Here and also in \Cref{sec:reconstruction}, we will the notion of Mittag-Leffler inverse system in an abelian category.

\begin{definition}[{\autocite[13.1.2]{grothendieckEGA3a}}]\label{def:ML}
An inverse system $\mathcal{M}=(M_i, f_{ij}\colon M_j\to M_i)_{i,j\in I, j\ge i}$ in an abelian category is called a \emph{Mittag-Leffler} system if it satisfies the following condition: For each $m\in I$ there exists $l\ge m$ such that for each $k\ge l$, we have $\im(f_{mk}) = \im(f_{ml})$.
\end{definition}

\begin{lemma}[{\autocite[Proposition~13.2.1]{grothendieckEGA3a}}]\label{lem:ML-closure}
Suppose we have a directed partially ordered set $(I,\le)$ and a short exact sequence $0\to\mathcal{K}\to\mathcal{L}\to\mathcal{M}\to 0$ of three inverse systems indexed by $I$ in an abelian category. Then:
\begin{enumerate}[label=(\alph*)]
\item if $\mathcal{L}$ is Mittag-Leffler, so is $\mathcal{M}$;
\item if $\mathcal{K}$ and $\mathcal{M}$ are Mittag-Leffler, so is $\mathcal{L}$.
\end{enumerate}
\end{lemma}

If we specialize the abelian category to the one of abelian groups, or more generally modules over a ring or a small preadditive category (since these are just additive functors into abelian groups), countably infinite Mittag-Leffler systems control exactness of the inverse limit functor. Here we will denote by $\limit^{(1)}$ the first right derived functor of the inverse limit functor (we refer to \autocite{jensenDerivedLim1972} or~\cite[\S3.5]{weibelIntro1994} for details).

\begin{lemma}\label{lem:ML-and-lim1}
Let $\category{A}$ be a small preadditive category, $\mathcal{K}=(K_i,f_{ij})_{i,j\in I,j\ge i}$ be a Mittag-Leffler inverse system in $\catMod\category{A}$ indexed by a countably infinite partially ordered set $(I,\le)$, and  denote by $f_i\colon\limit\mathcal{K}\to K_i$ the limit maps. Then:
\begin{enumerate}[label=(\alph*)]
\item for each $m\in I$ there exists $l\ge m$ such that $\im(f_{m}) = \im(f_{ml})$;
\item $\limit^{(1)}\mathcal{K}=0$;
\item any short exact sequence $0\to\mathcal{K}\to\mathcal{L}\to\mathcal{M}\to 0$ of inverse systems indexed by $(I,\le)$ in $\catMod\category{A}$ induces a short exact sequence of inverse limits
\[
  0 \to
	\limit\mathcal{K} \to
  \limit\mathcal{L} \to
  \limit\mathcal{M} \to 0.
\]
\end{enumerate}
\end{lemma}

\begin{proof}
For the first statement, note we can reduce the problem to the case $I=\mathbb{N}$ by passing to a cofinal subsystem of $\mathcal{K}$. Then the result follows by an easy induction argument, see the proof of~\autocite[Lemma~4.5]{sarochStovicekTelescope2008}.
We refer to~\autocite[Proposition~3.5.7]{weibelIntro1994} for the second statement, while the third one is a consequence and goes back at least to~\autocite[Proposition~13.2.2]{grothendieckEGA3a}.
\end{proof}

The main result of the section and some consequences follow.

\begin{proposition}\label{prop:completion-on-homs}
    Let $\category{T}$ a tensor-triangulated category with small coproducts, and endowed with an action of a graded commutative Noetherian ring $R$. Let $x, y \in \category{T}$ be such that the $R$-module $\Hom_\category{T}(x,y)$ is finitely generated. Finally, let $\ida$ be a homogeneous ideal in $R$ and consider a direct system as in \Cref{constr:transitionKoszul} for $x$ (with respect to some chosen set generators of $\ida$ of cardinality $k$ whose degrees sum up to $d$).
Then the diagram~\eqref{eq:Koszul-telescope} induces natural isomorphisms of $\comp{R}$-modules
    \[ 
    \comp{\Hom_{\category{T}}(x,y)}_\ida \overset{\sim}\to
		\limit_m \Hom_{\category{T}}(\Sigma^{md - k} x \kos \ida^{[m]},y) \overset{\sim}\leftarrow
		\Hom_{\category{T}}(\hocolim_m \Sigma^{md - k} x \kos \ida^{[m]},y).
    \]
\end{proposition}

\begin{proof}
To ease the notation and inspired by \Cref{prop:torsionColimit}, we denote $\Gamma_\ida x:=\hocolim_m \Sigma^{md - k} x \kos \ida^{[m]}$.

We will first prove the statement for a principal ideal $\ida=(r)$, where $r\in R$ is a homogeneous element of degree $d$.
Then the direct system from Construction~\ref{constr:transitionKoszul} is given by the following direct system of triangles:
    \[
        \begin{tikzcd}
						\Sigma^{-1}x \arrow[r, "r"] \arrow[d, equal] 
                        & \Sigma^{d-1} x \arrow[r] \arrow[d, "r"] 
                        & \Sigma^{d-1} x \kos r \arrow[r] \arrow[d] 
                        & x \arrow[d, equal]
						\\
						\Sigma^{-1}x \arrow[r, "r^2"] \arrow[d, equal] 
                        & \Sigma^{2d-1} x \arrow[r] \arrow[d, "r"] 
                        & \Sigma^{2d-1} x \kos r^2 \arrow[r] \arrow[d] 
                        & x \arrow[d, equal]
						\\
						\Sigma^{-1}x \arrow[r, "r^3"] \arrow[d, equal] 
                        & \Sigma^{3d-1} x \arrow[r] \arrow[d, "r"] 
                        & \Sigma^{3d-1} x \kos r^3 \arrow[r] \arrow[d] 
                        & x \arrow[d, equal]
						\\
						\vdots & \vdots & \vdots & \vdots
				\end{tikzcd}
		\]
If we apply $\Hom_{\category{T}}(-,y)$ to this system, we obtain an inverse system of long exact sequences of $R$-modules of the form
    \[
				\Hom_{\category{T}}(x,y) \to
				\Hom_{\category{T}}(\Sigma^{md-1}x\kos r^m,y) \to
				\Hom_{\category{T}}(\Sigma^{md-1}x,y) \xrightarrow{r^m}
				\Hom_{\category{T}}(\Sigma^{-1}x,y).
		\]
If we denote $M=\Hom_{\category{T}}(c,d)$ and $M[s]=\{m\in M\mid sm=0\}$ for a homogeneous element $s \in R$, the latter system induces an inverse system of short exact sequences of graded $R$-modules
		\begin{equation} \label{eq:completion-homs}
				\begin{tikzcd}[column sep=scriptsize]
						0 \arrow[r] & M/rM \arrow[r] \ar[d, <<-] 
                        & \Hom_{\category{T}}(\Sigma^{d-1}x\kos r, y) \arrow[r] \ar[d, <-] 
                        & \Sigma^{1-d}M[r] \arrow[r] \ar[d, <-, "r"] & 0
						\\
						0 \arrow[r] & M/r^2M \arrow[r] \ar[d, <<-] 
                        & \Hom_{\category{T}}(\Sigma^{2d-1}x\kos r^2, y) \arrow[r] \ar[d, <-] 
                        & \Sigma^{1-2d}M[r^2] \arrow[r] \ar[d, <-, "r"] & 0
						\\
						0 \arrow[r] & M/r^3M \arrow[r] \ar[d, <<-] 
                        & \Hom_{\category{T}}(\Sigma^{3d-1}x\kos r^3, y) \arrow[r] \ar[d, <-] 
                        & \Sigma^{1-3d}M[r^3] \arrow[r] \ar[d, <-, "r"] & 0
						\\
						& \vdots & \vdots & \vdots
				\end{tikzcd}
		\end{equation}
Since $M\in\catmod{R}$ and $R$ is assumed to be Noetherian, the chain
\[ M[r]\subset M[r^2]\subset M[r^3] \subset \cdots \]
stabilizes, and so there exists $D>0$ such that $r^D M[r^i] = 0$ for each $i\ge 1$. In particular, the inverse system
		\[
				\Sigma^{1-d}M[r] \overset{r}\leftarrow
				\Sigma^{1-2d}M[r^2] \overset{r}\leftarrow
				\Sigma^{1-3d}M[r^3] \overset{r}\leftarrow \dots
		\]
has a cofinal subsystem with all maps vanishing (one sometimes says that the system is pro-zero). Consequently, the system is Mittag-Leffler and
\[ \limit_i \Sigma^{1-md}M[r^m] = 0. \]
The inverse system $M/rM\twoheadleftarrow M/r^2M\twoheadleftarrow\cdots$ is also Mittag-Leffler, and so is the inverse system $(\Hom_{\category{T}}(\Sigma^{md-1}x\kos r^m, y))_{m\ge 1}$ by~\Cref{lem:ML-closure}. Furthermore, since the inverse limit functor is left exact and we know that the limit of the rightmost column in~\eqref{eq:completion-homs} vanishes, the canonical map
		\[
				\comp{M} \to \limit_m \Hom_{\category{T}}(\Sigma^{md-1}x \kos r^m,y)
		\]
is an isomorphism.

It remains to compare $\limit_m \Hom_{\category{T}}(\Sigma^{md-1}x\kos r^m,y)$ with the homomorphisms module $\Hom_{\category{T}}(\hocolim_m\Sigma^{md-1}x\kos r^m,y)$.
We apply \autocite[Proposition~2.2.4(d)]{hoveyPalmieriStricklandMemoir1997}, which says that the comparison map is a part of a short exact sequence
		\[
				0 \to
				{\limit_m}^{(1)} \Hom_{\category{T}}(\Sigma^{md}x\kos r^m,y) \to
				\Hom_{\category{T}}(\Gamma_\ida x,y) \to
				\limit_m \Hom_{\category{T}}(\Sigma^{md-1}x\kos r^m,y) \to
				0,
		\]
where $\limit^{(1)}$ is the first right derived functor of the inverse limit functor (we refer to \autocite{jensenDerivedLim1972} for details). Since the $\limit^{(1)}$ is vanishes on Mittag-Leffler inverse systems by \Cref{lem:ML-and-lim1}, we conclude that the natural homomorphisms of $R$-modules
		\begin{equation}\label{eq:completion-homs-principal}
				\comp{\Hom_{\category{T}}(x,y)} \overset{\sim}\to
				\limit_m \Hom_{\category{T}}(\Sigma^{md-1}x\kos r^m,y) \overset{\sim}\leftarrow
				\Hom_{\category{T}}(\Gamma_\ida x,y).
		\end{equation}
are bijective. Since all the modules are derived complete, they are in fact canonically homomorphisms of $\comp{R}$-modules by \Cref{lemma:derivedComplete}.
This proves the proposition for a principal ideal $\ida=(r)$.

In the general case where $\ida=(r_1, \dots, r_{k-1}, r_k)$, we proceed by induction on $k$. Let $\ida'=(r_1, \dots, r_{k-1})$, denote $d'$ the sum of the degrees of its generators, $d_k$ the degree of $r_k$, and put $x'=\Gamma_{\ida'} x$. The inductive hypothesis says applied to $x$ and $\Sigma^{ld_k+1}x\kos r_k^l$ gives a commutative diagram with isomorphisms in rows of the form
\[
  \begin{tikzcd}[column sep=1em, scale cd=.85]
    \comp{\Hom_{\category{T}}(x,y)}_{\ida'} \ar[r, "\sim"] \ar[d] &
	  \limit_m\Hom_{\category{T}}(\Sigma^{md'-k+1}x\kos\ida'^{[m]},y) \ar[r, <-, "\sim"] \ar[d] &
	  \Hom_{\category{T}}(x',y) \ar[d] 
		\\
    \comp{\Hom_{\category{T}}(\Sigma^{ld_k-1}x\kos r_k^l,y)}_{\ida'} \ar[r, "\sim"] &
	  \limit_m\Hom_{\category{T}}(\Sigma^{md'+ld_k-k}(x\kos\ida'^{[m]})\kos r_k^l,y) \ar[r, <-, "\sim"] &
	  \Hom_{\category{T}}(\Sigma^{ld_k-1}x'\kos r_k^l,y). 
	  \end{tikzcd}
	\]
Note that $\comp{R}_{\ida'}$ is Noetherian by~\autocite[\href{https://stacks.math.columbia.edu/tag/05GH}{Tag 05GH}]{thestacksprojectauthorsStacksProject} and all the terms in the diagram are finitely generated $\comp{R}_{\ida'}$-modules. So we can apply the same argument as above in $\catmod{\comp{R}_{\ida'}}$ to show that we have natural isomorphisms analogous to~\eqref{eq:completion-homs-principal} for the principal ideal $\idb=(r_k)\subset\comp{R}_{\ida'}$ and the objects $x',y$.

In particular, the $\idb$-completion of the first row in the previous diagram is canonically isomorphic to the inverse limit of the second rows over all $l\ge 1$. This gives us an isomorphism 
\[
  \comp{(\comp{\Hom_{\category{T}}(x,y)}_{\ida'})}_\idb \overset{\sim}\to
	\limit_{m,l}\Hom_{\category{T}}(\Sigma^{md'+ld_k-k}(x\kos\ida'^{[m]})\kos r_k^l,y)	
\]
The first isomorphism from the statement of the proposition for $\ida$ is obtained by composing this with the inverse of the canonical isomorphism 
$\comp{(\comp{\Hom_{\category{T}}(x,y)}_{\ida'})}_\idb \to \comp{\Hom_{\category{T}}(x,y)}_{\ida}$ on the left, and by passing to a cofinal inverse limit indexed over the diagonal $m=l$ on the right.

The second isomorphism is similar. The inductive hypothesis for $\ida'$ and the second isomorphism in~\eqref{eq:completion-homs-principal} for $\idb$ and $x',y$ combine to
\begin{align*}
  \Hom_\category{T}(\Gamma_\idb x',y) &\overset{\sim}\to
	\limit_l\Hom_\category{T}(\Sigma^{ld_k-1}x'\kos r_k^l,y) \\ 
	 & \overset{\sim}\to 
	 \limit_{m,l}\Hom_{\category{T}}(\Sigma^{md'+ld_k-k}(x\kos\ida'^{[m]})\kos r_k^l,y).
\end{align*}
Now we just use the identifications (where the last one is given by \Cref{lemma:tensor-and-hocolim})
\[
\Gamma_\idb x'= \hocolim_m \Sigma^{md - k} x' \kos r_k^m\cong x'\otimes(\hocolim_m \Sigma^{md - k} \mathbb{1} \kos r_k^m)\cong\Gamma_\ida x
\]
and pass to the cofinal diagonal of the double-indexed inverse limit.
\end{proof}

Now we can again assume that $\category{T}$ is compactly-rigidly generated tensor-triangu\-lated with a Noetherian central action of a commutative Noetherian graded ring $R$. 

\begin{corollary}\label{corollary:completion-on-homs}
Let $c,d\in\category{T}^c$ and $\ida\subset R$. Then the homomorphisms of $R$-modules
\begin{equation}\label{eq:torsion-completion}
\Hom_{\Gamma_\ida\category{T}}(\Gamma_\ida c,\Gamma_\ida d) \leftarrow
\Hom_{\category{T}}(c,d) \to
\Hom_{\Lambda^\ida\category{T}}(\Lambda^\ida c,\Lambda^\ida d)
\end{equation}
induce isomorphisms of $\comp{R}$-modules
\[
\Hom_{\Gamma_\ida\category{T}}(\Gamma_\ida c,\Gamma_\ida d) \overset{\sim}\longleftarrow
\Hom_{\category{T}}(c,d)\otimes_R\comp{R} \overset{\sim}\longrightarrow 
\Hom_{\Lambda^\ida\category{T}}(\Lambda^\ida c,\Lambda^\ida d). \]
\end{corollary}

\begin{proof}
We will discuss only the left arrows, since the claim for the right arrows follows from the equivalence of $\Gamma_\ida\category{T}$ and $\Lambda^\ida\category{T}$.
\Cref{prop:torsionColimit,prop:completion-on-homs} tell us that the left arrow in~\eqref{eq:torsion-completion} identifies with the canonical completion morphism
$\Hom_{\category{T}}(c,d) \to \comp{\Hom_{\category{T}}(c,d)}$.
 Since $\Hom_{\category{T}}(c,d)$ is finitely generated as an $R$-module, it remains to use the well-known fact that the completion functor on $\catmod{R}$ is naturally equivalent to $-\otimes_R\comp{R}$; see \autocite[\href{https://stacks.math.columbia.edu/tag/00MA}{Tag 00MA}]{thestacksprojectauthorsStacksProject}.
\end{proof}

If we specialize the proposition to $c=d=\mathbb{1}$, we obtain

\begin{corollary}\label{corollary:completion-on-homs-emorph}
Given any homogeneous ideal $\ida\subset R$, there is a ring isomorphism
$\emorp_{\category{T}}(\Gamma_\ida\mathbb{1})\cong\emorp_{\category{T}}(\mathbb{1})\otimes_R\comp{R}_\ida$.
\end{corollary}

A curious consequence of \Cref{prop:completion-on-homs} is that all dualizable objects can be realized not only as homotopy colimits (or limits) of compact $\ida$-torsion objects, but in a proper context as honest categorical colimits (or limits).
This is relatively unusual in the context of triangulated categories, but in some situations it happens; see \autocite[Lemma~5.2.9]{krauseHomologicalRT2022}.

\begin{corollary}\label{cor:koszul-colimit}
    Any dualizable object $x \in (\Gamma_{\ida}\category{T})^d$ can be presented as a classical categorical colimit of its associated Koszul objects in the category $(\Gamma_{\ida}\category{T})^d$:
    \[
        x \cong \colim_m \Sigma^{md-k}x \kos \ida^{[m]}
    \]
		(assuming that $\ida$ is generated by $k$ elements, with total degree $d$).
		Moreover, the functor $h \colon \Gamma_{\ida}\category{T} \to \catMod(\Gamma_{\ida}\category{T})^c$ preserves these colimits.
\end{corollary}

\begin{proof}
    We apply \Cref{prop:completion-on-homs} to the category $\Gamma_\ida\category{T}$, which admits an action of $\comp{R}$ by \Cref{remark:homFromTorsionIsComplete}. If $x,y\in(\Gamma_\ida\category{T})^d$, then the $\comp{R}$-module $\Hom(x,y)$ is finitely generated by \Cref{thm:finite-gen-hom-dualizables}, so $\ida$-complete. Hence, the first isomorphism established by the proposition takes the form
$\Hom_{(\Gamma\category{T})^d}(x,y) \cong \limit_m \Hom_{(\Gamma\category{T})^d}(\Sigma^{md-1}x \kos \ida^{[m]},y)$,	
which is none other than the universal property for a honest categorical colimit.
\end{proof}

We finish with a technical statement, which will be important in \Cref{sec:reconstruction}.

\begin{corollary}\label{cor:koszul-ML}
    Given any pair of dualizable objects $x,y \in (\Gamma_{\ida}\category{T})^d$, the inverse system of abelian groups $(\Hom_{(\Gamma_\ida\category{T})^d}(x\kos\ida^{[m]},y))_{m\ge 1}$ is Mittag-Leffler.
\end{corollary}

\begin{proof}
This is proved by induction on the number of generators of $\ida=(r_1,\ldots,r_k)$, where the case $k=1$ was explicitly established in the proof of \Cref{prop:completion-on-homs}.
For $k>1$, denote $\ida'=(r_1,\ldots,r_{k-1})$, by $d'$ the sum of the degrees of the generators for $\ida'$, by $d_k$ the degree of $r_k$, and
\[
M_{m,l} := \Hom_{\category{T}}(\Sigma^{md'+ld_k-k}(x\kos\ida'^{[m]})\kos r_k^l,y).
\]
Since the system in the statement is cofinal in $(M_{m,l})_{m,l\in\mathbb{N}}$, it suffices to prove the Mittag-Leffler property for $(M_{m,l})_{m,l\in\mathbb{N}}$. The inductive hypothesis says that $(M_{m,l})_{m\ge 1}$ is Mittag-Leffler for each $l\ge 1$, and also that the inverse system
\[
  \Big(\colim_m M_{m,l}\Big)_{l\ge 1}\cong
	\Big(\Hom_{\category{T}}\big((\Gamma_{\ida'}x)\kos r_k^l,y\big)\Big)_{l\ge 1}
\]
is Mittag-Leffler. So given any $m_0,l_0\in\mathbb{N}$, there are by \Cref{lem:ML-and-lim1}:
\begin{itemize}
\item $l_1\ge l_0$ such that image of $\limit_{m,l}M_{m,l}\to\limit_m M_{m,l_0}$ equals the image of $\limit_{m}M_{m,l_1}\to\limit_m M_{m,l_0}$, and
\item $m_1\ge m_0$ such that the image of $\limit_{m}M_{m,l_1}\to M_{m_0,l_1}$ equals the image of $M_{m_1,l_1}\to M_{m_0,l_0}$.
\end{itemize}
It follows that the image of $\limit_{m,l}M_{m,l}\to M_{m_0,l_0}$ coincides with the image of $M_{m_1,l_1}\to M_{m_0,l_0}$, so $(M_{m,l})_{m,l\in\mathbb{N}}$ is Mittag-Leffler as required.
\end{proof}

\begin{remark}
One may ask whether $\Hom_{\category{T}}(c,\Gamma_\ida d)$ and $\Gamma_\ida \Hom_{\category{T}}(c,d)$, where the second $\Gamma_\ida$ stands for the torsion radical $\Gamma_\ida\colon \catmod{R} \to \catmod{R}$,
compare in a similar way. Unfortunately, the situation is more complicated here, since unlike the completion functor, the torsion radical is \emph{not} exact on finitely generated modules. If $\ida=(r)$ is principal, a similar argument as the one for \Cref{prop:completion-on-homs} only gives us a short exact sequence
		\[
				0 \to 
				H^1_\ida \Hom_{\category{T}}(\Sigma^{-1}c,d) \to
				\Hom_{\category{T}}(c,\Gamma_\ida d) \to
				\Gamma_\ida \Hom_{\category{T}}(c,d) \to
				0,
		\]
where $H^1_\ida\colon \catmod{R} \to \catmod{R}$ is the first local cohomology (i.e.\ the first right derived functor of $\Gamma_\ida$). If the ideal $\ida$ is not principal, the situation gets even more complicated as also higher local cohomologies come into play.
\end{remark}

\subsection{Thick generation for dualizable objects}

We again fix a tensor-tri\-an\-gu\-la\-ted, compactly-rigidly generated category $\category{T}$, with a Noetherian action of $R$ and a homogeneous ideal $\ida\subset R$. Our aim here is characterize in terms of torsion compact objects when $x\in\thick(g)$ for a pair $x,g\in(\Gamma\category{T})^d$.

We will need two preparatory lemmas. The first one allows to detect vanishing of morphisms in $(\Gamma_\ida\category{T})^d$ in terms of their tensors in $(\Gamma_\ida\category{T})^c$.

\begin{lemma}\label{lemma:map-tensored-compacts-zero}
    Let $f \colon x \to y$ be a homomorphism in $(\Gamma_{\ida}\category{T})^d$. Then the following are equivalent:
\begin{enumerate}[label=(\alph*)]
\item $f\otimes c=0$ for each $c\in(\Gamma_\ida\category{T})^c$;
\item $f\otimes (\mathbb{1}\kos\ida^{[m]})=0$ for each $m\ge 0$;
\item $f=0$.
\end{enumerate}
\end{lemma}

\begin{proof}
The first condition obviously implies the second one.
If $f$ is such that $f\otimes (\mathbb{1}\kos\ida^{[m]})=0$ for each $m\ge 0$, then \Cref{constr:transitionKoszul} yields a direct system of maps $f\otimes (\mathbb{1}\kos\ida^{[m]})\colon x\kos\ida^{[m]}\to y\kos\ida^{[m]}$ in $\Gamma_\ida\category{T}$ with a cocone $f$, and all the compositions $x\kos\ida^{[m]}\to y\kos\ida^{[m]}\to y$ vanish by assumption. So the compositions $x\kos\ida^{[m]}\to x\to y$ vanish as well, and so does $f$ by Corollary~\ref{cor:koszul-colimit}.
\end{proof}

The second lemma shows how properties of the sequence from \Cref{constr:genGammaGen} translate to Koszul complexes.

\begin{lemma} \label{lemma:ghosts-to-koszul}
Suppose that $\idb=(b_1,\ldots,b_k)\subset R$ is a homogeneous ideal generated by a sequence of $k$ elements of total degree $d$. If we have an object $g$ and a sequence of maps
\[
    x_0 \xrightarrow{\phi_0} x_1 \xrightarrow{\phi_1} x_2 \xrightarrow{\phi_2} \cdots \xrightarrow{\phi_{2^{2k}-1}} x_{2^{2k}}
\]
in $\Gamma_\ida\category{T}$ such that $\Hom(g,\phi_i)=0$ for each $i$, then also
\[
    \Hom(g\kos\idb,(\phi_{2^{2k}-1}\circ\cdots\phi_1\circ\phi_0)\otimes\mathbb{1}\kos\idb) = 0.
\]
\end{lemma}

\begin{proof}
Since $g\kos\idb\in\thick^{2^k-1}(g)$ by \Cref{lemma:koszul-generation-time}, \autocite[Proposition~3.3]{christensenIdeals1998} implies that $\Hom(g\kos\idb,\phi_{i+2^k-1}\circ\cdots\circ\phi_i)=0$ for each $0\le i\le 2^k(2^k-1)$.
Further, note that $\mathbb{1}\kos r\cong\Sigma^{1-e}\SW{(\mathbb{1}\kos r)}$ for each $r\in R$ of degree $e$. By induction on $k$, we get that $\mathbb{1}\kos\idb\cong\Sigma^{k-d}\SW{(\mathbb{1}\kos\idb)}$ and so
\[
    \Hom(g\kos\idb,\phi_{i+2^k-1}\circ\cdots\circ\phi_i)\cong
    \Hom(\Sigma^{k-d}g,(\phi_{i+2^k-1}\circ\cdots\circ\phi_i)\otimes\mathbb{1}\kos\idb)=0
\]
for each $0\le i\le 2^k(2^k-1)$. The conclusion follows by another application of~\autocite[Proposition~3.3]{christensenIdeals1998}, using that $g\kos\idb\in\thick^{2^k-1}(\Sigma^{k-d}g)$.
\end{proof}

Now we can prove the main result of the section.

\begin{theorem} \label{thm:genDual}
Let $\category{T}$ be an {R}-linear, Noetherian, tensor-triangulated, compactly-rigidly generated category and let $\ida\subset R$ be a homogeneous ideal. Then the following are equivalent for $x,g\in(\Gamma_\ida\category{T})^d)$:
\begin{enumerate}[label=(\roman*)]
\item\label{item:dual-gen-in-Td} $x\in\thick(g)$ in $(\Gamma_\ida\category{T})^d$;
\item\label{item:dual-gen-in-Tc} there exist $n\ge 0$ such that for each $m>0$, we have $x\kos\ida^{[m]}\in\thick^n(g\kos\ida^{[m]})$ in $(\Gamma_\ida\category{T})^c$. 
\end{enumerate}
If $g$ is of the form $g\cong\Gamma_\ida c$ for $c\in\category{T}^c$, then the conditions are further equivalent to
\begin{enumerate}[resume, label=(\roman*)]
\item\label{item:dual-gen-tor-Tc} there exist $n\ge 0$ such that for each $m>0$, $x\kos\ida^{[m]}\in\thick^n(c)$ in~$\category{T}^c$.
\end{enumerate}
\end{theorem}

\begin{proof}
Suppose that $y\in\thick(x)$; then $y\in\thick^n(x)$ for some $n\ge 0$.
Since $-\otimes\mathbb{1}\kos\ida^{[m]}\colon(\Gamma_\ida\category{T})^d\to(\Gamma_\ida\category{T})^c$ is an exact functor for each $m>0$, we get the implication $\ref{item:dual-gen-in-Td}\implies\ref{item:dual-gen-in-Tc}$.

Conversely, assume condition~$\ref{item:dual-gen-in-Tc}$. We will apply \Cref{constr:genGammaGen} to the $\comp{R}$-linear triangulated category $\Gamma_\ida\category{T}$. Clearly, the category $\category{F}$ there contains $(\Gamma_\ida\category{T})^d$ by \Cref{thm:finite-gen-hom-dualizables} and the constructed sequence
\[
    x=x_0 \xrightarrow{\phi_0} x_1 \xrightarrow{\phi_1} x_2 \xrightarrow{\phi_2} \cdots
\]
belongs to $(\Gamma_\ida\category{T})^d$. If $k$ is the length of a chosen generating sequence for $\ida$, \Cref{lemma:ghosts-to-koszul} tells us that 
\[
\Hom(g\kos\ida^{[m]},(\phi_{i+2^{2k}-1}\circ\cdots\phi_{i+1}\circ\phi_i)\otimes\mathbb{1}\kos\ida^{[m]}) = 0
\]
for each $i\ge 0$ and $m\ge 1$. As $x\kos\ida^{[m]}\in\thick^n(g\kos\ida^{[m]})$ by assumption, \autocite[Proposition~3.3]{christensenIdeals1998} implies that
\[
\Hom(x\kos\ida^{[m]},(\phi_{2^{2k}n-1}\circ\cdots\phi_1\circ\phi_0)\otimes\mathbb{1}\kos\ida^{[m]}) = 0,
\]
or equivalently that $(\phi_{2^{2k}n-1}\circ\cdots\phi_1\circ\phi_0)\otimes\mathbb{1}\kos\ida^{[m]}=0$ for each $m\ge 1$.
Consequently, $\phi_{2^{2k}n-1}\circ\cdots\phi_1\circ\phi_0=0$ by \Cref{lemma:map-tensored-compacts-zero}, and $x\in\thick^{2^{2k}n-1}(g)$ by \Cref{lemma:constrGenGamma}~\ref{item:genGammaGen-thick}. Hence $\ref{item:dual-gen-in-Tc}\implies\ref{item:dual-gen-in-Td}$.

In order to prove $\ref{item:dual-gen-in-Tc}\implies\ref{item:dual-gen-tor-Tc}$, notice that if $g\cong\Gamma_\ida c$ and $\ida$ has a chosen generating sequence of length $k$, then we get by the same argument as in the proof of \Cref{lemma:ghosts-to-koszul} that
\[
  g\kos\ida^{[m]}\cong(\Gamma_\ida c)\kos\ida^{[m]}\cong\Gamma_\ida(c\kos\ida^{[m]})\cong c\kos\ida^{[m]}\in\thick^{2^k-1}(c).
\]
It follows that~$\ref{item:dual-gen-in-Tc}$ implies that $x\kos\ida^{[m]}\in\thick^{(n+1)(2^k-1)-1}(c)$ for each $m\ge 0$.

Finally, assume $\ref{item:dual-gen-tor-Tc}$ and that $g\cong\Gamma_\ida c$, so that $g\kos\ida^{[m]}\cong c\kos\ida^{[m]}$ for each $m$. Given any particular $m$, we have $x\kos\ida^{[m]}\in(\Gamma_\ida\category{T})^c$, so there is some $l'>0$ such that $\ida^{[l']}$ acts by zero on $x\kos\ida^{[m]}$. Then, by \Cref{lemma:koszul-split}, $x\kos\ida^{[m]}$ is a direct summand of $(x\kos\ida^{[m]})\kos\ida^{[l]}\cong x\otimes(\mathbb{1}\kos\ida^{[m]})\otimes(\mathbb{1}\kos\ida^{[l]})\cong (x\kos\ida^{[l]})\kos\ida^{[m]}$ for $l=(2^k-1)l'$. Since $x\kos\ida^{[l]}\in\thick^n(c)$ by assumption, we have $(x\kos\ida^{[l]})\kos\ida^{[m]}\in\thick^n(c\kos\ida^{[m]})$, and the above shows that $x\kos\ida^{[m]}\in\thick^n(c\kos\ida^{[m]})$. Hence $\ref{item:dual-gen-in-Tc}$ follows.
\end{proof}

\begin{remark}
If $R$ is an ordinary commutative Noetherian ring, $\category{T}=\category{D}(\catMod R)$ and $g=\Gamma_\ida R$, Balmer and Sanders \autocite{balmerSandersPerfectComplexesCompletion2024} proved that $x\in\thick(g)$ for any $x\in(\Gamma_\ida\category{T})^d$. The idea of their proof, rephrased in our terminology, is that they proved condition~$\ref{item:dual-gen-tor-Tc}$ of \Cref{thm:genDual} for any given $x$. The choice of $n$ in their case was the width of the (necessarily perfect) complex $x\otimes_R^\mathbf{L}R/\ida\in\category{D}(\catMod(R/\ida))$.
\end{remark}

\section{Reconstruction of dualizable objects from compact objects}
\label{sec:reconstruction}

Throughout all of~\Cref{sec:reconstruction}, we fix a tensor-tri\-an\-gu\-la\-ted, compactly-rigidly generated category $\category{T}$, with a Noetherian action of a Noetherian graded ring $R$, and we further fix a homogeneous ideal $\ida$ of $R$. We only be interested in the subcategory $\Gamma_\ida\category{T}$ of $\category{T}$.
We will show that the categories of torsion compact and torsion dualizable objects determine each other as triangulated closed symmetric monoidal categories (with the nuance that the tensor unit is not compact) in a way formally similar to~\cite{neemanMetrics2020}. As the category of torsion compact objects is a subcategory of the one of torsion dualizable objects, the main problem is to reconstruct the latter category from the former one.

Despite the formal similarity of the results, our techniques are completely different from~\autocite{neemanMetrics2020}. Our categories potentially admit no nontrivial $t$-structures, there cannot be any translation invariant metric associated to our completion in the sense of~\autocite[Heuristic~9]{neemanMetrics2020}, and torsion dualizable objects are typically not compactly supported with respect to our ``Cauchy sequences'' in the sense of \autocite[Definition~4]{neemanMetrics2020}.
So our class of examples seems to be of a rather different nature: we use the monoidal structure in an essential way and where the compactly supported objects would induce eventually stable inverse systems, we only get Mittag-Leffler ones by \Cref{cor:koszul-ML}.

\subsection{Compacts versus dualizables as additive categories}

Following~\autocite[\S2.5]{letzBrownRepresentabilityTriangulated2023}, we will denote by $\catMod(\Gamma_{\ida}\category{T})^{c}$ the category of graded $R$-linear (i.e.\ enriched over the category of graded modules $\catMod R$) contravariant functors
\[ (\Gamma_{\ida}\category{T})^{c,\mathrm{op}}\longrightarrow\catMod R. \]
Note that since $\id_x\in\Hom_{\Gamma_\ida\category{T}^c}(x,x)$ is annihilated by some power of $\ida$ for each $x\in(\Gamma_{\ida}\category{T})^c$, the value $F(x)$ must be an $\ida$-power torsion module for each $F\in\catMod(\Gamma_{\ida}\category{T})^{c}$.
We also have the associated $R$-linear restricted Yoneda functor
\begin{equation} \label{eq:restricted-graded-Yoneda}
  h \colon \Gamma_{\ida}\category{T} \to \catMod(\Gamma_{\ida}\category{T})^{c}, \quad
  h(x) = \Hom_{\Gamma_\ida\category{T}}(-, x)_{|(\Gamma_{\ida}\category{T})^{c}}.
\end{equation}

A key point for the reconstruction of $(\Gamma_\ida\category{T})^d$ from $(\Gamma_\ida\category{T})^c$ is the following proposition. The situation is formally similar to the one studied in~\autocite[Proposition 5.2.8]{krauseHomologicalRT2022}, where one completes the perfect derived category of a coherent ring to its bounded derived category.

\begin{proposition}\label{prop:yoneda-restricted-fully-faithful}
     The restriction of the functor $h \colon \Gamma_{\ida}\category{T} \to \catMod(\Gamma_{\ida}\category{T})^c$ to dualizable objects is fully faithful.
\end{proposition}
\begin{proof}
    Let $x, y \in (\Gamma_{\ida}\category{T})^d$. By Yoneda's lemma, $h$ is fully faithful on $(\Gamma_{\ida}\category{T})^c$.
Then, by \Cref{cor:koszul-colimit}
    \begin{align*}        
        \Hom(x,y) 
        &\cong \limit_m \Hom(\Sigma^{md - k}x \kos \ida^{[m]}, y) \\ 
        &\cong \limit_m \colim_l \Hom(\Sigma^{md-k}x \kos \ida^{[m]}, \Sigma^{ld-k}y \kos \ida^{[l]}) \\
        &\cong \limit_m \colim_l \Hom(h(\Sigma^{md - k}x \kos \ida^{[m]}), h(\Sigma^{ld - k}y \kos \ida^{[l]})) \\ 
        &\cong \limit_m \Hom(h(\Sigma^{md - k}x \kos \ida^{[m]}), \colim_l h(\Sigma^{ld - k}y \kos \ida^{[l]})) \\
        &\cong \Hom(\colim_m h(\Sigma^{md - k}x \kos \ida^{[m]}), \colim_l h(\Sigma^{ld - k}y \kos \ida^{[l]})) \\
        &\cong \Hom(h(x), h(y)).
				\qedhere
    \end{align*}
\end{proof}

While we will devote next subsections to describing the essential image of $h_{|(\Gamma_{\ida}\category{T})^d}$, it is easy to recover $(\Gamma_{\ida}\category{T})^c$ inside $(\Gamma_{\ida}\category{T})^d$ as finitely presentable objects.

\begin{corollary}\label{cor:dualizable-to-compacts}
The following are equivalent for an object $x\in(\Gamma_{\ida}\category{T})^d$:
\begin{enumerate}
\item $x$ is compact in $\Gamma_\ida\category{T}$;
\item $\Hom(x,-)\colon(\Gamma_{\ida}\category{T})^d\to\catmod R$ preserves all direct limits that exist in $(\Gamma_{\ida}\category{T})^d$.
\end{enumerate}
\end{corollary}

\begin{proof}
If $x$ is compact, then $h(x)$ is finitely presentable (even finitely generated projective) in $\catMod(\Gamma_\ida\category{T})^c$, and so $\Hom(c,-)$ preserves all existing direct limits in $(\Gamma_{\ida}\category{T})^d$ by \Cref{prop:yoneda-restricted-fully-faithful}.
Conversely, we have $x\cong\colim_m\Sigma^{md - k}x \kos \ida^{[m]}$ in $(\Gamma_{\ida}\category{T})^d$ by \Cref{cor:koszul-colimit}. So if $\Hom(x,-)$ preserves direct limits, then $x$ must be a summand of one of the $\Sigma^{md - k}x \kos \ida^{[m]}$ and, in particular, compact.
\end{proof}

\subsection{Reconstructing the closed monoidal structure}

Our next step is to lift the tensor product from $\Gamma_\ida\category{T}$ to $\catMod(\Gamma_{\ida}\category{T})^{c}$ (cp.~\autocite[\S A.13]{balmerKrauseStevensonSmashingFrame2020}). 
For that purpose, we will need 
\begin{definition}
    Let $F, G$ be in $\catMod(\Gamma_{\ida}\category{T})^{c}$. The \emph{Day convolution} of $F$ and $G$ is defined as
    \[
        (F \boxtimes G)(d) = \int^{c_1, c_2 \in (\Gamma_{\ida}\category{T})^c} F(c_1) \otimes_R G(c_2) \otimes_R \Hom_{\category{T}}(d, c_1 \otimes c_2)
    \]
    This assignment extends to a bifunctor over $\catMod R$,
    \[
        \boxtimes \colon \catMod(\Gamma_{\ida}\category{T})^{c} \times \catMod(\Gamma_{\ida}\category{T})^{c} \to \catMod(\Gamma_{\ida}\category{T})^{c}.
    \]
\end{definition}

\begin{remark}
    Since a coend is a colimit and the tensor product on $\catMod(R)$ preserves small colimits in each variable, the Day convolution bifunctor $\boxtimes$ also preserves small colimits in each variable.
		
		In fact, as the target category in which we are taking the coend is $\catMod R$, we have also another explicit formula for $(F\boxtimes G)(d)$ as the smallest quotient graded $R$-module of the direct sum
		\[ \bigoplus_{c_1, c_2 \in (\Gamma_{\ida}\category{T})^c} F(c_1) \otimes_R G(c_2) \otimes_R \Hom_{\category{T}}(d, c_1 \otimes c_2), \]
		where for each pair of maps $v_1\colon c_1'\to c_1$ and $v_2\colon c_2'\to c_2$, and an element $x\otimes y\otimes u\in F(c_1)\otimes_R G(c_2)\otimes_R \Hom_\category{T}(d,c_1'\otimes c_2')$, we identify
				\[ F(v_1)(x)\otimes F(v_2)(y) \otimes u \sim x\otimes y\otimes ((v_1\otimes v_2)\circ u). \]
\end{remark}

A basic property of the co-end is the following ``dual'' version of the Yoneda lemma. Although we state it in a general form, using elements of enriched category theory from~\autocite{kellyEnrichedCT2005}, we only need it in a much more down-to-the-earth situation where $\category{V}=\catMod R$.

\begin{lemma}\label{lemma:coend-representables}
    Let $(\category{V},\otimes,\mathbb{1})$ be a bicomplete closed symmetric monoidal category and $\category{C}$ be a small $\category{V}$-enriched category. For any $\category{V}$-functor $F \colon \category{C} \to \category{V}$, there is an isomorphism, $\category{V}$-natural both in $F$ and $b \in \category{C}$:
    \[
        F(b) \cong \int^{a \in \category{C}} \intHom_{\category C}(a,b) \otimes F(a)
    \]
    Dually, for any functor $G \colon \category{C}^{\mathrm{op}} \to \category{V}$, there is an isomorphism, $\category{V}$-natural both in $G$ and $b \in \category{C}^{\mathrm{op}}$:
    \[
        G(b) \cong \int^{a \in \category{C}} \intHom_{\category C}(b,a) \otimes G(a)
    \]
\end{lemma}
\begin{proof}
    This is standard. A version for $(\category{V},\otimes)=(\category{Set},\times)$ is a part of~\autocite[Proposition 2.2.1]{loregianCoendCalculus2021}, but the same proof applies in general using the $\category{V}$-enriched Yoneda lemma~\autocite[\S2.4]{kellyEnrichedCT2005}.
\end{proof}

\begin{remark} \label{remark:coYoneda-explicit}
The isomorphisms from the latter proposition can be made explicit. We will discuss only the first case as the second in similar.
For each $a,b\in\category{V}$, the homomorphism of the corresponding Hom-objects $\intHom_{\category{C}}(a,b) \to \intHom_{\category{V}}(F(a),F(b))$ has an adjoint form, the evaluation morphism
\[
\varepsilon_{a,b}\colon\intHom_{\category{C}}(a,b) \otimes F(a)\to F(b).
\]
Jointly, the evaluation morphisms form a $\category{V}$-natural cowedge (cf.\ \autocite[\S1.8]{kellyEnrichedCT2005}), inducing a morphism $\int^{a \in \category{C}} \intHom_{\category C}(a,b) \otimes F(a)\to F(b)$.
The inverse to this map is also explicit. One tensors the unit map $\eta_b\colon \mathbb{1}\to\intHom_{\category C}(b,b)$ with $F(b)$, which gives rise to
\[ F(b) \cong \mathbb{1}\otimes F(b) \to \intHom_{\category C}(b,b)\otimes F(b) \to \int^{a \in \category{C}} \intHom_{\category C}(a,b)\otimes F(a). \]
\end{remark}

As an immediate corollary, we obtain well-known properties of the Day convolution: it is symmetric and associative up to natural transformations satisfying the usual coherence axioms, and the restricted Yoneda functor~\eqref{eq:restricted-graded-Yoneda} is (non-unitary) strong monoidal. Getting back to our setting, we obtain the following formula.

\begin{proposition} \label{prop:day-conv-dualizable}
    For any dualizable torsion object $x \in (\Gamma_{\ida}\category{T})^d$ and any $F\in\catMod(\Gamma_{\ida}\category{T})^c$, we have an isomorphism
    \[
        h(x) \boxtimes F \cong F(\SW{x} \otimes -),
    \]
	  natural (in the $\catMod R$-enriched sense) both in $x$ and $F$.
\end{proposition}
\begin{proof}
    If $x\in(\Gamma_{\ida}\category{T})^c$ is compact, we have for any other compact object $c \in (\Gamma_{\ida}\category{T})^c$ a chain of natural isomorphisms, where the first and the last one are using \Cref{lemma:coend-representables} and the symmetry of the tensor product while the second one comes from the dualizablity of $c$:
    \begin{align*}
        (h(x) \boxtimes F)(c) &= \int^{c_1, c_2 \in (\Gamma_{\ida}\category{T})^c}\Hom_\category{T}(c_1, x) \otimes_R F(c_2) \otimes_R \Hom_\category{T}(c, c_1 \otimes c_2) \\
        &\cong \int^{c_2 \in (\Gamma_{\ida}\category{T})^c} F(c_2) \otimes_R \Hom_\category{T}(c, x \otimes c_2) \\
        &\cong \int^{c_2 \in (\Gamma_{\ida}\category{T})^c} F(c_2) \otimes_R \Hom_\category{T}(\SW{x} \otimes c, c_2) \cong F(\SW{x} \otimes c).
    \end{align*}
		This establishes the natural isomorphism for $x \in (\Gamma_{\ida}\category{T})^c$.
		
		Using \Cref{remark:coYoneda-explicit}, we get an explicit form of the isomorphism and a stronger statement. Namely, there is a natural transformation $\alpha_{x,F}\colon h(x) \boxtimes F \to F(\SW{x} \otimes -)$ defined for any $x\in(\Gamma_\ida\category{T})^d$ and $F\in\catMod(\Gamma_{\ida}\category{T})^c$, and this transformation is bijective if $x$ is compact. To describe $\alpha_{x,F}$ explicitly, it is uniquely given by a cowedge
		\[ \Hom_\category{T}(c_1, x) \otimes_R F(c_2) \otimes_R \Hom_\category{T}(c, c_1 \otimes c_2) \to F(\SW{x}\otimes c), \quad c\in(\Gamma_\ida\category{T})^c. \]
		If $u\otimes y\otimes v\in\Hom_\category{T}(c_1, x) \otimes_R F(c_2) \otimes_R \Hom_\category{T}(c, c_1 \otimes c_2)$, we denote by $\beta_{u\otimes v}$ the composition
		\[
		  \begin{tikzcd}[column sep=huge]
			  \SW{x}\otimes c \ar[r, "\id_{\SW{x}}\otimes v"] &
				\SW{x}\otimes c_1\otimes c_2 \ar[r, "\id_{\SW{x}}\otimes u\otimes \id_{c_2}"] &
				\SW{x}\otimes x\otimes c_2 \ar[r, "\varepsilon_x\otimes \id_{c_2}"] &
				c_2.
			\end{tikzcd}
		\]
	  in $(\Gamma_\ida\category{T})^c$, where $\varepsilon_{x}\colon \SW{x}\otimes x\to \Gamma_\ida\mathbb{1}$ is the canonical isomorphism. Now we define the cowedge so that we send $u\otimes y\otimes v$ to $F(\beta_{u\otimes v})(y)\in F(\SW{x}\otimes c)$.
		
		Next, given $x\in(\Gamma_\ida\category{T})^d$, we can express it in $(\Gamma_\ida\category{T})^d$ as $x \cong \colim_i \Sigma^{id-k}x \kos \ida^{[i]}$ by~\Cref{cor:koszul-colimit}, and this clearly induces an isomorphism
		\[ \colim_i \big(h(\Sigma^{id-k}x \kos \ida^{[i]})\boxtimes F\big) \overset{\sim}\to h(x)\boxtimes F. \] 
		In order to prove that $\alpha_{x,F}$ from the last paragraph is an isomorphism, it suffices to check that the colimiting cocone $(\Sigma^{id-k}x \kos \ida^{[i]} \to x)$ in $(\Gamma_\ida\category{T})^d$ also induces an isomorphism
		\begin{equation} \label{eq:SW-system-for-F}
		  \colim_i \Big(F\big(\SW{(\Sigma^{id-k}x \kos \ida^{[i]})}\otimes-\big)\Big) \overset{\sim}\to F(\SW{x}\otimes-).
		\end{equation}
					
    To see this, note that for any $c\in(\Gamma_\ida\category{T})^c$ and for $n'$ large enough (depending on $c$), the ideal $\ida^{[n']}$ acts by zero on $x\otimes\SW{c}$. Hence, by \Cref{lemma:koszul-split}, the tensor product $\Sigma^{nd - k}(x \kos \ida^{[n]}) \otimes \SW{c}\to x\otimes\SW{c}$ is a split epimorphism for each $n\ge(2^k-1)n'$. Moreover, by \Cref{constr:transitionKoszul}, if we denote by $z_n\to\Sigma^{nd - k}(x \kos \ida^{[n]})$ the kernels of the split epimorphisms in $(\Gamma_\ida\category{T})^c$, the induced maps $\phi_n\colon z_n\to z_{n+1}$ belong to $\ida\cdot\Hom(z_n,z_{n+1})$. Then we have split monomorphisms
    \[
		  \SW{x} \otimes c \cong 
			\SW{(x\otimes\SW{c})} \to
		  \SW{\big(\Sigma^{nd - k}(x\kos \ida^{[n]}) \otimes \SW{c}\big)} \cong
		  \SW{(\Sigma^{nd - k}x \kos \ida^{[n]})} \otimes c
    \]
and when we apply $F$, we obtain a direct system of split short exact sequences $\catMod R$,
\[
  0 \to F(z_n) \to
  F(\SW{(\Sigma^{nd - k}x \kos \ida^{[n]})} \otimes c) \to
	F(\SW{x} \otimes c) \to 0.
\]
Passing to a direct limit, we obtain a short exact sequence in $\catMod R$,
\[
  0 \to \colim_i F(\SW{z_i}) \to
  \colim_i F(\SW{(\Sigma^{id - k}x \kos \ida^{[i]})} \otimes c) \to
	\colim_i F(\SW{x} \otimes c) \to 0.
\]
Moreover, in the leftmost direct system
\[
  \begin{tikzcd}
     F(\SW{z_n}) \arrow[r, "\phi_{n*}"] &
     F(\SW{z_{n+1}}) \arrow[r, "\phi_{n+1*}"] &
     F(\SW{z_{n+2}}) \arrow[r, "\phi_{n+2*}"] &
		 \ldots,
	\end{tikzcd}
\]
we have $\phi_{i*}\in\ida\cdot\Hom_R(F(\SW{z_i}),F(\SW{z_{i+1}}))$ by construction. Since $F$ takes values in $\ida$-torsion modules, no non-zero element survives in this colimit and, hence, the morphism~\eqref{eq:SW-system-for-F} is an isomorphism as required.
\end{proof}

As a consequence, the Day convolution tensor product is unital, although the unit may not be finitely presentable in $\catMod(\Gamma_{\ida}\category{T})^{c}$, and we can promote the functor from \Cref{prop:yoneda-restricted-fully-faithful} to a strong monoidal functor.

\begin{proposition}\label{prop:yoneda-strong-monoidal}
    The Day convolution bifunctor $\boxtimes$ induces a closed symmetric monoidal structure on $\catMod(\Gamma_{\ida}\category{T})^{c}$ with unit
    \[
        \tilde{\mathbb{1}} := h(\Gamma_{\ida}\mathbb{1}) \cong \colim_{n} h(\Sigma^{nd - k}\mathbb{1} \kos \ida^{[n]}). 
    \]
		and the internal Hom given by $[F,G](c)=\Hom(F(\SW{c}\otimes-),G)$.
		Moreover, the restricted Yoneda functor
    \[
        h \colon (\Gamma_{\ida}\category{T}, \otimes, \Gamma_\ida\mathbb{1}) \to (\catMod(\Gamma_{\ida}\category{T})^c, \boxtimes, \tilde{\mathbb{1}})
    \]
		becomes a strong monoidal functor.
\end{proposition}
\begin{proof}
    The symmetry and associativity are classical properties of Day convolution and has been discussed above. We just have to check that the $\tilde{\mathbb{1}}$ is indeed the tensor unit. 
    Recall that $\SW{\Gamma_{\ida}\mathbb{1}} \cong \Gamma_{\ida}\mathbb{1}$, so by \Cref{prop:day-conv-dualizable},
    for any $F \in \catMod(\Gamma_{\ida}\category{T})^c$,
    \[
        \tilde{\mathbb{1}} \boxtimes F \cong F(\SW{\Gamma_{\ida}\mathbb{1}} \otimes - ) \cong F(\Gamma_{\ida}\mathbb{1} \otimes - ) \cong F.
    \]
		The required coherence properties for this isomorphism (see e.g.\ \autocite[\S1.4]{kellyEnrichedCT2005}) follow from those for $\otimes$ on $(\Gamma_\ida\category{T})^d$ (see also \Cref{remark:approximate-unitality} below). We also have the following chain of isomorphisms given by \Cref{prop:day-conv-dualizable}, the definition of $[G,H]$ and Yoneda's lemma, respectively, 
		\[
		  \Hom(h(c)\boxtimes G,H)
			\cong \Hom(G(\SW{c}\otimes-),H)
			\cong [G,H](c)
			\cong \Hom(h(c),[G,H]),
    \]
which is natural in $c\in(\Gamma_\ida\category{T})^{c,\mathrm{op}}$, $G\in(\catMod(\Gamma_\ida\category{T})^{c})^{\mathrm{op}}$ and $H\in\catMod(\Gamma_\ida\category{T})^{c}$.
By taking colimits in the first variable, this uniquely extends to a natural isomorphism $\Hom(F\boxtimes G,H)\cong\Hom(F,[G,H])$, showing that $[-,-]$ indeed is the internal Hom for $(\catMod(\Gamma_\ida\category{T})^{c},\boxtimes,\tilde{\mathbb{1}})$.

    Regarding the moreover part, we have $h(\Gamma_{\ida}\mathbb{1}) = \tilde{\mathbb{1}}$ tautologically and the fact that there is a natural isomorphism $h(x)\boxtimes h(y)\cong h(x\otimes y)$ for each pair $x,y\in\Gamma_\ida\category{T}$ follows as in~\autocite[Proposition~A.14]{balmerKrauseStevensonSmashingFrame2020}. That is, we have a natural isomorphism
    \[
        h(x) \boxtimes h(y) \cong h(y)(\SW{x} \otimes -) \cong h(x \otimes y),
    \]    
	  by \Cref{prop:day-conv-dualizable} if $x,y$ are compacts (or even dualizable), and it remains to use that for any $x\in\Gamma_{\ida}\category{T}$, there is a canonical isomorphism $h(x)\cong\colim_{c\to x}h(c)$, where the colimit is indexed by the slice category $(\Gamma_\ida\category{T})^c/x$. We again leave checking the coherence axioms to the reader.
\end{proof}

\begin{remark}\label{remark:approximate-unitality}
This shows that if we wish to reconstruct the tensor unit of $(\Gamma_\ida\category{T})^d$ from $(\Gamma_\ida\category{T})^c$, we need to remember, a direct system of ``approximate units'' in $(\Gamma_\ida\category{T})^c$ such as
\[ 
  \begin{tikzcd}
	  \Sigma^{d-k} \mathbb{1}\kos\ida
		  \arrow[r, "\phi_1"] &
	  \Sigma^{2d-k} \mathbb{1}\kos\ida^{[2]}
		  \arrow[r, "\phi_2"] &
	  \Sigma^{3d-k} \mathbb{1}\kos\ida^{[3]}
		  \arrow[r, "\phi_3"] &
		\cdots,
  \end{tikzcd}
\]
whose colimit in $(\Gamma_\ida\category{T})^d$ is $\Gamma_\ida\mathbb{1}$. Then we can for each $c\in(\Gamma_\ida\category{T})^c$ tensor this system by $c$, and \Cref{prop:torsionColimit} and \Cref{cor:koszul-ML} give us a colimiting cocone
\[
(\Sigma^{md-k} \mathbb{1}\kos\ida^{[m]}\otimes c \to
\Gamma_\ida\mathbb{1}\otimes c \overset{\sim}\to c)_{m\ge 1},
\]
which is functorial in $c$ and which can be viewed a sequence of ``approximate left unitality constraints'' for $c$. Clearly, we also have a similar sequence of ``approximate right unitality constraints'' which differ from the left ones only by composition with the symmetry constraints, and they satisfy the necessary coherence which says that
$c\otimes \Sigma^{md-k} \mathbb{1}\kos\ida^{[m]}\otimes c' \overset{\sim}\to c\otimes c'$
is the same map (up to the associativity), independently of whether we use the left or the right approximate unitality.
We can provisionally call this structure along on $((\Gamma_\ida\category{T})^c,\otimes)$ an \emph{approximate unital structure}.

Note that if we have a coherently associative and commutative bifunctor $\otimes\colon \category{C}\times\category{C}\to\category{C}$, and we have two units $\mathbb{1},\mathbb{1}'$ along with their coherent unitality constraints, they are isomorphic via $\mathbb{1}\cong\mathbb{1}\otimes\mathbb{1}'\cong\mathbb{1}'$. However, our approximate unital structure by no means enjoys such strong uniqueness. For two different ``approximate units'' $(u_1\to u_2\to \cdots)$ and $(v_1\to v_2\to \cdots)$ in $(\Gamma_\ida\category{T})^c$, we only get $\colim_m h(u_m)\cong\colim h(v_m)$ in $\catMod(\Gamma_\ida\category{T})^c$.
\end{remark}

\subsection{Dualizable functors}

In this part, we describe the essential image of the full strongly monoidal embedding $h \colon (\Gamma_{\ida}\category{T})^d \to \catMod(\Gamma_{\ida}\category{T})^c$ in a way which is intrinsic to $(\catMod(\Gamma_{\ida}\category{T})^c,\boxtimes,\tilde{\mathbb{1}})$.

\begin{proposition}\label{prop:essential-image-dualizable}
    Let $F \in \catMod(\Gamma_{\ida}\category{T})^c$. Then the following assertions are equivalent:
    \begin{enumerate}[label=(\roman*)]
   \item\label{item:eid-image-dualizable} $F \cong h(d)$ for some $d \in (\Gamma_{\ida}\category{T})^d$;
	  \item\label{item:eid-dualizable-functor} $F$ is dualizable in $(\catMod(\Gamma_{\ida}\category{T})^c,\boxtimes,\tilde{\mathbb{1}})$;
    \item\label{item:eid-boxtimes-compacts} for any compact object $c \in (\Gamma_{\ida}\category{T})^c$, $h(c) \boxtimes F \cong h(c')$ for some compact object $c' \in (\Gamma_{\ida}\category{T})^c$;
    \item\label{item:eid-boxtimes-standard-compacts} for any $m > 0$, $h(\mathbb{1} \kos \ida^{[m]})\boxtimes F \cong h(c')$ for some compact object $c' \in (\Gamma_{\ida}\category{T})^c$.
		\end{enumerate}
\end{proposition}
\begin{proof}
    The implication $\ref{item:eid-image-dualizable} \implies \ref{item:eid-dualizable-functor}$ follows from the fact that $h$ is strong monoidal by \Cref{prop:yoneda-strong-monoidal} and $\ref{item:eid-boxtimes-compacts} \implies \ref{item:eid-boxtimes-standard-compacts}$ is clear.

    $\ref{item:eid-dualizable-functor}\implies\ref{item:eid-boxtimes-compacts}$
		Suppose that $F$ is dualizable. Thus $[F,-]\cong G\boxtimes-$ for the Spanier-Whitehead dual $G\in\catMod(\Gamma_{\ida}\category{T})^c$ and, as $G\boxtimes-$ preserves all colimits, so does the functor
  \[
	  \Hom(h(c)\boxtimes F,-)\cong\Hom(h(c),[F,-])\colon \catMod(\Gamma_{\ida}\category{T})^c \to \catMod R.
	\]
for each $c\in(\Gamma_\ida\category{T})^c$.
It follows that $h(c)\boxtimes F$ is finitely generated projective in $\catMod(\Gamma_{\ida}\category{T})^c$ for each $c\in(\Gamma_\ida\category{T})^c$.

    $\ref{item:eid-boxtimes-standard-compacts}\implies\ref{item:eid-image-dualizable}$
		Assume that $F$ satisfies condition $\ref{item:eid-boxtimes-standard-compacts}$. For each $m > 0$, fix $c_m \in (\Gamma_{\ida}\category{T})^c$ such that $h(c_m) \cong h(\mathbb{1} \kos \ida^{[m]})\boxtimes F$. By Yoneda's lemma, the maps $\Sigma^{md - k}\mathbb{1} \kos \ida^{[m]} \to \Sigma^{(m+1)d - k}\mathbb{1} \kos \ida^{[m+1]}$ from \Cref{constr:transitionKoszul} for $x=\mathbb{1}$ induce maps $\Sigma^{md - k}c_m \to \Sigma^{(m+1)d - k}c_{m+1}$. Then
    \[
        F \cong \tilde{\mathbb{1}} \boxtimes F \cong
				\colim_m h(\Sigma^{md - k}\mathbb{1} \kos \ida^{[m]}) \boxtimes F \cong
				\colim_m h(\Sigma^{md - k}c_m) \cong h(\hocolim_m \Sigma^{md - k}c_n)
    \]
    Hence it suffices to prove that $d := \hocolim_n\Sigma^{nd - k}c_n$ is dualizable.
		To that end, for any $c\in(\Gamma_{\ida}\category{T})^c$ we have isomorphisms 
		\[ h(c\otimes d)\cong h(c)\boxtimes h(d) \cong h(c)\boxtimes F, \]
		using \Cref{prop:yoneda-strong-monoidal}. So $h(c')\cong h(c\otimes d)$ for some $c'$ compact and, by Yoneda's lemma, this can be lifted to an isomorphism $c'\cong c\otimes d$. It follows from \Cref{prop:compactRigid} that $d$ is dualizable.
\end{proof}

\begin{corollary}\label{cor:essential-image-dualizable}
The restricted Yoneda functor $h\colon \Gamma_\ida\category{T}\to\catMod(\Gamma_\ida\category{T})$ induces an equivalence of closed monoidal categories of dualizable objects,
\[
  h_{|(\Gamma_\ida\category{T})^d}\colon
	\big((\Gamma_\ida\category{T})^d,\otimes,\Gamma_\ida\mathbb{1}\big) \overset{\sim}\longrightarrow
	\big((\catMod(\Gamma_\ida\category{T})^c)^d,\boxtimes,\tilde{\mathbb{1}}\big).
\]
\end{corollary}

\subsection{Reconstructing the triangulated structure}

Finally, we will take up the most subtle part, which is reconstructing the triangulated structure of $(\Gamma_\ida\category{T})^d$ from that of $(\Gamma_\ida\category{T})^c$. The shift functor is the obvious one induced by the graded shift in $\catMod R$ and it will turn out that the distinguished triangles in $(\Gamma_\ida\category{T})^d$ are colimits of distinguished triangles in $(\Gamma_\ida\category{T})^c$. 
Our arguments will use the following technical lemma.

\begin{lemma}\label{lemma:sheaf-ext-sequence}
Let $(M_0\to M_1\to M_2\to\cdots)$ be a sequential direct system in an AB5 abelian category $\category{A}$ and let $N\in\category{A}$. Then we have a short exact sequence of abelian groups
\[
0 \to \limit\nolimits^{(1)}_{i}\Hom_\category{A}(M_i,N)
\to \Ext^1_\category{A}(\colim_i M_{i},N)
\to \limit\nolimits_{i}\Ext^1_\category{A}(M_i,N) \to 0
\]
\end{lemma}

\begin{proof}
Consider the exact sequence $0 \to \bigoplus_{i\ge 0}M_i \to \bigoplus_{i\ge 0}M_i \to \colim M_i \to 0$ which defines the direct limit and apply $\Hom_\category{A}(-,N)$. The desired short exact sequence is induced by the resulting long exact sequence around the term $\Ext^1_\category{A}(\colim_i M_i,N)$. One should also note that coproducts are exact in $\category{A}$, and so the contravariant functor $\Ext^1_\category{A}(-,N)$ sends coproducts in $\category{A}$ to products of abelian groups.
\end{proof}

\begin{notation} \label{notation:reconstruction-tria}
In view of \Cref{prop:yoneda-restricted-fully-faithful}, one can extend $h_{|(\Gamma_{\ida}\category{T})^d}\colon (\Gamma_{\ida}\category{T})^d\to\catMod(\Gamma_{\ida}\category{T})^c$ to a fully exact functor
\[
\tilde{h}\colon \catmod(\Gamma_{\ida}\category{T})^d\longrightarrow\catMod(\Gamma_{\ida}\category{T})^c,
\]
where $(\Gamma_{\ida}\category{T})^d$ is the category of $R$-linear finitely presentable functors $(\Gamma_{\ida}\category{T})^{d,\mathrm{op}}\to\catMod R$. This is an instance of \autocite[Lemma 2.1]{krauseTelescope2000}.

For brevity, we also put $U_m:=h(\Sigma^{md-k}\mathbb{1}\kos\ida^{[m]})$ for each $m$ and denote by $\psi_m\colon U_m\to\tilde{\mathbb{1}}$ the morphisms induced by~\eqref{eq:Koszul-telescope} (so that $U_1\to U_2\to \cdots$ is an ``approximate unit'' in $\catmod(\Gamma_\ida\category{T})^c$ in the sense of \Cref{remark:approximate-unitality}).
\end{notation}

\begin{proposition}\label{prop:ext-complete}
Let $M,N\in\catmod(\Gamma_\ida\category{T})^d$. Then the direct system $U_1\to U_2\to U_3\to\cdots$ from \Cref{notation:reconstruction-tria} induces an isomorphism of abelian groups
\[
  \Ext^1(\tilde{h}(M),\tilde{h}(N))\overset{\sim}\to
	\limit_m\Ext^1(U_m\boxtimes\tilde{h}(M),\tilde{h}(N)).
\]
\end{proposition}

\begin{proof}
In view of \Cref{lemma:sheaf-ext-sequence}, it suffices to prove that
\[
  \limit\nolimits_{m}^{(1)} \Hom(U_m\boxtimes\tilde{h}(M),\tilde{h}(N))=0.
\]
By adjunction, this is the same as proving that $\limit\nolimits_{m}^{(1)}\Hom(U_m,F)=0$ for $F=[\tilde{h}(M),\tilde{h}(N)]\in\catMod(\Gamma_{\ida}\category{T})^c$.

We claim that $F\cong \tilde{h}(L)$ for some $L\in\catmod(\Gamma_\ida\category{T})^d$. Indeed, since $\tilde{h}$ is a full exact embedding, there is an exact sequence $h(x_1)\to h(x_0)\to\tilde{h}(M)\to 0$ with $x_0,x_1\in(\Gamma_\ida\category{T})^d$, which induces an exact sequence
\[
  0 \to [\tilde{h}(M),\tilde{h}(N)] \to [h(x_0),\tilde{h}(N)] \to [h(x_1),\tilde{h}(N)].
\]
So it is sufficient to show that $[h(x),\tilde{h}(N)]$ is in the essential image of $\tilde{h}$ for each $x\in(\Gamma_\ida\category{T})^d$. Similarly, we have an exact sequence $0\to\tilde{h}(N)\to h(y^0)\to h(y^1)$ with $y^0,y^1\in(\Gamma_\ida\category{T})^d$, which induces an exact sequence 
\[
  0 \to [h(x),\tilde{h}(N)] \to [h(x),h(y^0)] \to [h(x),h(y^1)].
\]
So we are left with proving that $[h(x),h(y)]$ is in the essential image of $\tilde{h}$ for each $x,y\in(\Gamma_\ida\category{T})^d$. However, we have $[h(x),h(y)]\cong h\big(\intHom_{\Gamma_\ida\category{T}}(x,y)\big)$ by \Cref{cor:essential-image-dualizable}. This proves the claim.

Finally, we can choose a surjection $h(x)\to F$ for some $x\in(\Gamma_\ida\category{T})^d$, which for each $m\ge 1$ induces a surjection
\[
  \Hom_{(\Gamma_\ida\category{T})^d}(\Sigma^{md-k}\mathbb{1}\kos\ida^{[m]},x) \cong\Hom(U_m,h(x)) \to \Hom(U_m,F),
\]
since the $U_m$ are finitely generated projective in $\catMod(\Gamma_\ida\category{T})^c$.
As the inverse system on the left is Mittag-Leffler by \Cref{cor:koszul-ML}, so is the one on the right by~\Cref{lem:ML-closure}, and the conclusion that $\limit\nolimits_{m}^{(1)}\Hom(U_m,F)=0$ follows from \Cref{lem:ML-and-lim1}.
\end{proof}

For the sake of completeness, we also record the following consequence.

\begin{corollary}\label{cor:tilde-h-ext}
The essential image of the functor $\tilde{h}$ is extension closed.
\end{corollary}

\begin{proof}
It is equivalent to show that $\Ext^1(M,N)\to\Ext^1(\tilde{h}(M),\tilde{h}(N))$ is bijective for each $M,N\in\catmod(\Gamma_\ida\category{T})^d$. Since representable functors are both projective and injective in $\catmod(\Gamma_\ida\category{T})^d$, we can compute $\Ext^1$ using projective and injective resolutions in $\catmod(\Gamma_\ida\category{T})^d$. Furthermore, since $\tilde{h}$ is exact, our problem boils down to showing that $\Ext^1(h(x),h(y))=0$ in $\catMod(\Gamma_\ida\category{T})^c$ for each $x,y\in(\Gamma_\ida\category{T})^d$. However, we then have
\begin{align*}    
  \Ext^1(h(x),h(y))\overset{\sim}\to &
	\limit_m\Ext^1\!\big(U_m\boxtimes h(x),h(y)\big) \\
	&\cong \limit_m\Ext^1\!\big(h(\Sigma^{md-k}\mathbb{1}\kos\ida^{[m]}\otimes x),h(y)\big) = 0
\end{align*}

since all $h(\Sigma^{md-k}\mathbb{1}\kos\ida^{[m]}\otimes x)\in\catMod(\Gamma_\ida\category{T})^c$ are finitely generated projective.
\end{proof}

The Day convolution product $U_m\boxtimes-$ is exact since $U_m$ is dualizable in $\catMod(\Gamma_\ida\category{T})^c$ by \Cref{cor:essential-image-dualizable} (or directly by \Cref{prop:day-conv-dualizable}). In particular, for each $F,G\in\catMod(\Gamma_\ida\category{T})^c$, the Day convolution products with $U_m$ induce a homomorphism of abelian groups
\begin{equation} \label{eq:ext-embedding}
  \zeta_{F,G}\colon \Ext^1(F,G) \to \prod_{m\ge 0} \Ext^1\big(U_m\boxtimes F, U_m\boxtimes G).
\end{equation}

\begin{lemma}\label{lem:ext-box-injection}
If $F,G$ are in the essential image of the functor $\tilde{h}\colon \catmod(\Gamma_\ida\category{T})^d\to\catMod(\Gamma_\ida\category{T})^c$ from \Cref{notation:reconstruction-tria}, then the above map $\zeta_{F,G}$ is injective.
\end{lemma}

\begin{proof}
Suppose that $\varepsilon\colon 0\to G \to E \to F\to 0$ is such that $U_m\boxtimes\varepsilon$ splits for each $m\ge 1$. Then each $\psi_m\colon U_m\to\tilde{\mathbb{1}}$ yields a morphism of short exact sequences $U_m\boxtimes\varepsilon\to\varepsilon$, which in turn gives rise to a diagram
\[
  \begin{tikzcd}
		U_m\boxtimes\varepsilon\colon & 0 \ar[r] & U_m\boxtimes G \ar[r] \ar[d] & U_m\boxtimes E \ar[r] \ar[d] & U_m\boxtimes F \ar[r] \ar[d, equal] & 0
	  \\
		\theta_m\colon & 0 \ar[r] & G \ar[r] \ar[d, equal] & H_m \ar[r] \ar[d] & U_m\boxtimes F \ar[r] \ar[d, "\psi_m\boxtimes F"] & 0
	  \\
	  \varepsilon\colon & 0 \ar[r] & G \ar[r] & E \ar[r] & F \ar[r] & 0
	\end{tikzcd}
\]
Then also each $\theta_m$ splits, so $\Ext^1(G,\psi_m\boxtimes F)([\varepsilon])=0$ for each $m$, where $[\varepsilon]\in\Ext^1(G,F)$ denotes the element given by $\varepsilon$. Hence $[\varepsilon]=0$ by \Cref{prop:ext-complete}.
\end{proof}

Finally, we obtain the description of distinguished triangles in $(\Gamma_\ida\category{T})^d$ in terms of those in $(\Gamma_\ida\category{T})^c$.

\begin{proposition}\label{prop:koszul-triangles}
Let $x\to y\to z\to \Sigma x$ be a diagram in $(\Gamma_\ida\category{T})^d$. Then it is a distinguished triangle in $(\Gamma
_\ida\category{T})^d$ if and only if
\[
(\mathbb{1}\kos\ida^{[m]})\otimes x\to (\mathbb{1}\kos\ida^{[m]})\otimes y\to (\mathbb{1}\kos\ida^{[m]})\otimes z\to (\mathbb{1}\kos\ida^{[m]})\otimes \Sigma x\cong\Sigma(\mathbb{1}\kos\ida^{[m]}\otimes x)
\]
is a distinguished triangle in $(\Gamma_\ida\category{T})^c$ for each $m\ge 1$.
\end{proposition}

\begin{proof}
The `only if' part is clear since $-\otimes-$ is assumed to be exact in both variables.

For the 'if' part, we transport the tensor-triangulated structure from $(\Gamma_\ida\category{T})^d$ to the category of dualizable objects in $\catMod(\Gamma_\ida\category{T})^c$ using the monoidal equivalence \Cref{cor:essential-image-dualizable} and prove the statement there.  We can also assume that the shift functor $\Sigma$ on $\catMod(\Gamma_\ida\category{T})^c$ is given simply by a graded shift on $\catMod R$.
Using \Cref{notation:reconstruction-tria}, suppose that
\[
  \begin{tikzcd}
	  \tau\colon \quad h(x) \ar[r, "u"] & h(y) \ar[r, "v"] & h(z) \ar[r, "w"] & \Sigma h(x)
	\end{tikzcd}
\]
is a diagram in $\catMod(\Gamma_\ida\category{T})^c$ such that
\[
  \begin{tikzcd}
	  U_m\boxtimes \tau\colon \quad U_m\boxtimes h(x) \ar[r, "U_m\boxtimes u"] & U_m\boxtimes h(y) \ar[r, "U_m\boxtimes v"] & U_m\boxtimes h(z) \ar[r, "U_m\boxtimes w"] & \Sigma U_m\boxtimes h(x)
	\end{tikzcd}
\]
is a distinguished triangle for each $m\ge 1$ (making the identification $U_m\boxtimes \Sigma h(x)\cong \Sigma U_m\boxtimes h(x)$ implicit). In particular, $U_m\boxtimes\tau$ is an exact sequence in $\catMod(\Gamma_\ida\category{T})^c$ for each $m$. Now, recall that we have a direct system $U_1\to U_2\to\cdots$ in $\catMod(\Gamma_\ida\category{T})^c$ with a colimiting cocone $(\psi_m\colon U_m \to \tilde{\mathbb{1}})$, so using the unitality constraint, $\tau$ identifies with $\colim_m(U_m\boxtimes \tau)$ in $\catMod(\Gamma_\ida\category{T})^c$. In particular, $\tau$ is also an exact sequence in $\catMod(\Gamma_\ida\category{T})^c$.

Since we know that $(\Gamma_\ida\category{T})^d$ is a triangulated category, we can complete $h(u)$ to a distinguished triangle
\[
  \begin{tikzcd}
	  \tau'\colon \quad h(x) \ar[r, "u"] & h(y) \ar[r, "v'"] & h(z') \ar[r, "w'"] & \Sigma h(x),
	\end{tikzcd}
\]
giving rise to a similar direct system of distinguished triangles $U_m\boxtimes\tau'$ whose colimit is $\tau'$. For each single $m\ge 1$, the distinguished triangles $U_m\boxtimes\tau$ and $U_m\boxtimes\tau'$ must be isomorphic, i.e.\ we have a commutative diagram of the form
\begin{equation} \label{eq:local-comparison-tria}
  \begin{tikzcd}
	  U_m\boxtimes h(x) \ar[r, "U_m\boxtimes u"] \ar[d, equal] & U_m\boxtimes h(y) \ar[r, "U_m\boxtimes v"] \ar[d, equal] & U_m\boxtimes h(z) \ar[r, "U_m\boxtimes w"] \ar[d, "\sim"] & \Sigma U_m\boxtimes h(x)\phantom{.} \ar[d, equal]
    \\
	  U_m\boxtimes h(x) \ar[r, "U_m\boxtimes u"] & U_m\boxtimes h(y) \ar[r, "U_m\boxtimes v'"] & U_m\boxtimes h(z') \ar[r, "U_m\boxtimes w'"] & \Sigma U_m\boxtimes h(x).
	\end{tikzcd}
\end{equation}

Finally, let us reinterpret the above isomorphisms. Denoting $G=\coker(u)$ and $F=\ker(\Sigma u)$, the sequence $\tau$ induces a short exact sequence
\[
\varepsilon\colon \quad 0 \to G \xrightarrow{\bar{v}} h(z) \xrightarrow{\bar{w}} F \to 0
\]
and $\tau'$ induces a similar sequence
\[
\varepsilon'\colon \quad 0 \to G \xrightarrow{\bar{v}'} h(z') \xrightarrow{\bar{w}'} F \to 0.
\]
So both $\varepsilon$ and $\varepsilon'$ produce elements $[\varepsilon], [\varepsilon']\in\Ext^1(F,G)$ in $\catMod(\Gamma_\ida\category{T})^c$. Now, \eqref{eq:local-comparison-tria} precisely tells us that the images of $[\varepsilon]$ and $[\varepsilon']$ under the map~\eqref{eq:ext-embedding} are equal. It follows from \Cref{lem:ext-box-injection} that $[\varepsilon]=[\varepsilon']$, and this in turn implies the existence a commutative diagram of the form
\[
  \begin{tikzcd}
	  h(x) \ar[r, "u"] \ar[d, equal] & h(y) \ar[r, "v"] \ar[d, equal] & h(z) \ar[r, "w"] \ar[d, "\sim"] & \Sigma h(x) \ar[d, equal]
		\\
	  h(x) \ar[r, "u"] & h(y) \ar[r, "v'"] & h(z') \ar[r, "w'"] & \Sigma h(x).
	\end{tikzcd}
\]
Hence $\tau$ must be a distinguished triangle since $\tau'$ was chosen as such.
\end{proof}

We conclude by summarizing the results in a section in a compact form.

\begin{theorem}\label{thm:reconstruction}
Let $\category{T}$ be an $R$-linear, Noetherian, tensor-tri\-an\-gu\-lated, compactly-rigidly generated category, and fix a homogeneous ideal $\ida$ of $R$.

Then we can, up to equivalence, recover the tensor-triangulated $\comp{R}$-linear category $(\Gamma_\ida\category{T})^d$ from the approximate unital (in the sense of \Cref{remark:approximate-unitality}) $R$-linear tensor-triangulated category $(\Gamma_\ida\category{T})^c$.

Conversely, $(\Gamma_\ida\category{T})^c$ is precisely the full triangulated subcategory of $(\Gamma_\ida\category{T})^d$ consisting of finitely presentable objects (i.e.\ those $x\in(\Gamma_\ida\category{T})^d$ such that the functor $\Hom(x,-)\colon (\Gamma_\ida\category{T})^d\to\catmod R$ preserves all existing direct limits).
\end{theorem}

\begin{proof}
The (closed) symmetric monoidal category $(\Gamma_\ida\category{T})^d$ can be reconstructed from $(\Gamma_\ida\category{T})^c$ using \Cref{cor:essential-image-dualizable}, and the triangulation on $(\Gamma_\ida\category{T})^d$ is determined by~\Cref{prop:koszul-triangles}. The converse reconstruction of $(\Gamma_\ida\category{T})^c$ from $(\Gamma_\ida\category{T})^d$ was given in \Cref{cor:dualizable-to-compacts}.
\end{proof}

\section{Consequences of strong generation}
\label{sec:strong-gen}

Now we will focus on compactly-rigidly generated Noetherian $R$-linear categories $\category{T}$ such that the category of compact objects $\category{T}^c$ admits a strong generator (recall \Cref{definition:strong-generator}). As explained in \Cref{example:strongly-generated=regular-finite-dim}, this can be viewed as a regularity condition in the context of commutative Noetherian rings.

\subsection{Strong generators for torsion and complete dualizable objects}

First, we are going to prove that if $\category{T}^c$ has a strong generator, then so have it the categories of dualizable $\ida$-torsion and dualizable $\ida$-complete objects. So the slogan is that if $\category{T}^c$ is ``regular'', so is the small tt-completion $(\Gamma_\ida\category{T})^d\simeq(\Lambda^\ida\category{T})^d$.
Note that, in contrast, the smaller categories of compact objects $(\Gamma_\ida\category{T})^c$ and $(\Lambda^\ida\category{T})^c$ are typically \emph{not} strongly generated due to \autocite[Theorem~4.1]{steenStevensonStrongGenerators2015}.

\begin{theorem}\label{thm:strongGenDual}
    Let $\category{T}$ be an $R$-linear, Noetherian, tensor-triangulated, compactly-rigidly generated category. Let $g$ be a strong generator of $\category{T}^c$. Then $\Gamma_{\ida}g$ is a strong generator of $(\Gamma_{\ida}\category{T})^d$ and $\Lambda^{\ida}g$ is a strong generator of $(\Lambda^{\ida}\category{T})^d$.
\end{theorem}
\begin{proof}
    Since the equivalence of categories $\Lambda^{\ida} \colon \Gamma_{\ida}\category{T} \xrightarrow{\sim} \Lambda^{\ida}\category{T}$ maps the subcategories $\thick(\Gamma_{\ida}g)$ to $\thick(\Lambda^{\ida}g)$ and $(\Gamma_{\ida}g)^d$ to $(\Lambda^{\ida}g)^d$, it suffices to prove the version for $(\Gamma_{\ida}g)^d$.
     This is then an immediate consequence of \Cref{thm:genDual}, since condition~$\ref{item:dual-gen-tor-Tc}$ there is always satisfied by the strong generation assumption.
		
		However, we can also give a more direct and streamlined argument.
		Fix $D \in \mathbb{N}$ such that $\thick^D(g) = \category{T}^c$.
		Then for any $c \in (\Gamma_{\ida}\category{T})^c$, we have $c \in \thick^D(g)$ and, by applying the exact functor $\Gamma_{\ida}$, also
    \[
        c \cong \Gamma_{\ida}c \in \thick^D(\Gamma_{\ida}g).
    \]
    Hence $(\Gamma_{\ida}\category{T})^c \subset \thick^D(\Gamma_{\ida}g)$.
    Now fix $d \in (\Gamma_{\ida}\category{T})^d$ and recall that $\Gamma_{\ida}\category{T}$ has a canonical action of $\comp{R}$ by \Cref{remark:homFromTorsionIsComplete}. Then for any $k \geq 1$,
    \[
        d \kos \ida^{[k]} \in (\Gamma_{\ida}\category{T})^c \subset \thick^D(\Gamma_{\ida}g)
    \]
    Since $d\cong\Gamma_\ida d$ is a homotopy colimit of certain shifts of $d \kos \ida^{[k]}$ by \Cref{prop:torsionColimit}, the defining triangle for the homotopy colimit shows that $d \in \Loc^{2D+1}(\Gamma_{\ida}g)$. Hence, by \Cref{prop:finHomGen} applied to the $\comp{R}$-linear category $\Gamma_\ida\category{T}$, the thick subcategory $\category{U}=(\Gamma_\ida\category{T})^d$ and the object $\Gamma_{\ida}g\in\category{U}$, we obtain that
    \[
        \thick^{2D+1}(\Gamma_{\ida}g) = \thick(\Gamma_{\ida}g) = (\Gamma_{\ida}\category{T})^d.
    \]
    Here we also used that $\Hom_{\Gamma_\ida\category{T}}(\Gamma_\ida g, d')$ is a finitely generated $\comp{R}$-module for any $d'\in\category{U}$ by \Cref{cor:finiteGenOverCompR}.
\end{proof}

\subsection{Characterization of dualizable objects via Hom-groups}

Putting all the previous results together, we obtain several characterizations of dualizable objects in $\Gamma_{\ida}\category{T}$ and $\Lambda^{\ida}\category{T}$ in terms of Hom-groups from and to compacts.
First of all, in a strictly Noetherian category (in particular in a $\category{T}$ where the category of compact objects is strongly generated, recall \Cref{prop:strictlyRightNoetherian}), the dualizable objects are even characterized by finite generatedness of homomorphism $R$-modules.
This can be viewed as a partial converse to \Cref{prop:dualNoethComp} and \Cref{cor:dualNoethTor}.

\begin{proposition}\label{prop:strictRightNoethDual}
    Assume that $\category{T}$ is strictly Noetherian.
    Then:
    \begin{enumerate}[label=(\alph*)]
    \item\label{item:sRNoethDualComp} Let $x \in \Lambda^{\ida}\category{T}$ be an object such that, for any compact object $c \in (\Lambda^{\ida}\category{T})^c$, $\Hom_{\category{T}}(c, x)$ is Noetherian over $R$. Then $x$ is dualizable in $\Lambda^{\ida}\category{T}$. 
    \item\label{item:sRNoethDualTor} Let $x \in \Gamma_{\ida}\category{T}$ be an object such that, for any compact object $c \in (\Gamma_{\ida}\category{T})^c$, $\Hom_{\category{T}}(c, x)$ is Noetherian over $R$. Then $x$ is dualizable in $\Gamma_{\ida}\category{T}$. 
    \end{enumerate}
\end{proposition}
\begin{proof}
    We first prove \ref{item:sRNoethDualComp}. Let $c \in (\Lambda^{\ida}\category{T})^c = (\Gamma_{\ida}\category{T})^c$. By \Cref{prop:compactRigid}, it suffices to show that $c \otimes_{\Lambda^{\ida}\category{T}} x$ is compact. To this end, suppose that $c' \in \category{T}^c$. Then 
    \[
        \Hom_{\category{T}}(c', c \otimes_{\Lambda^{\ida}\category{T}} x) =
				\Hom_{\category{T}}(c', \Lambda^{\ida}(c \otimes x)) \cong
				\Hom_{\category{T}}(\Gamma_{\ida} c', c \otimes x),
    \]
where the first equality comes from \Cref{prop:tensorCompTor}\,\ref{item:tensorCompTor-complete} and the second isomorphism from the adjunction $\Gamma_{\ida} \dashv \Lambda^{\ida}$. Since $c$ is dualizable both in $\Gamma_\ida\category{T}$ and $\category{T}$ and the Spanier-Whitehead duals coincide, we further have
    \[
				\Hom_{\category{T}}(\Gamma_{\ida} c', c \otimes x) \cong
				\Hom_{\category{T}}((\Gamma_{\ida} c') \otimes \SW{c}, x) \cong
				\Hom_{\category{T}}(c' \otimes \SW{c}, x),
    \]
		where the last isomorphism follows from the fact that $\SW{c}\in\Gamma_\ida\category{T}$.
    In particular, $\Hom_{\category{T}}(c', c \otimes_{\Lambda^{\ida}\category{T}} x)$ is Noetherian over $R$, since $c' \otimes \SW{c}$ is in $(\Gamma_{\ida}\category{T})^c=(\Lambda^{\ida}\category{T})^c$. As $c'$ was chosen arbitrarily, our hypothesis on $\category{T}$ then implies that $c \otimes_{\Lambda^{\ida}\category{T}} x$ is compact.

    We now turn to \ref{item:sRNoethDualTor}. Note that $\Lambda^{\ida}x$ satisfies the hypothesis of \ref{item:sRNoethDualComp}, hence is dualizable in $\Lambda^{\ida}\category{T}$. Hence $x$ is dualizable in the equivalent category $\Gamma_{\ida}\category{T}$.
\end{proof}

As for the converse to the other halves of \Cref{prop:dualNoethComp} and \Cref{cor:dualNoethTor}, we can prove it only for some categories $\category{T}$, namely those which are compactly cogenerated in the sense of \Cref{definition:compactly-cogenerated} (see also \Cref{examples:compactly-cogenerated}).

\begin{proposition}\label{prop:strictLeftNoethDual}
    Assume that $\category{T}$ is compactly cogenerated.
    Then:
    \begin{enumerate}[label=(\alph*)]
    \item\label{item:sLNoethDualComp} Let $x \in \Lambda^{\ida}\category{T}$ be an object such that, for any compact object $c \in (\Lambda^{\ida}\category{T})^c$, $\Hom_{\category{T}}(x, c)$ is Noetherian over $R$. Then $x$ is dualizable in $\Lambda^{\ida}\category{T}$. 
    \item\label{item:sLNoethDualTor} Let $x \in \Gamma_{\ida}\category{T}$ be an object such that, for any compact object $c \in (\Gamma_{\ida}\category{T})^c$, $\Hom_{\category{T}}(x, c)$ is Noetherian over $R$. Then $x$ is dualizable in $\Gamma_{\ida}\category{T}$. 
    \end{enumerate}
\end{proposition}

\begin{proof}
The argument is completely dual to the one for \Cref{prop:strictRightNoethDual}. Only instead of strict Noetherianity, we use \Cref{corollary:compactness-by-cogeneration}.
\end{proof}

Now we can give a characterization of torsion dualizable objects. If $R$ is graded local and $\ida\subset R$ is the maximal graded ideal, there are also other equivalent conditions by \autocite[Proposition 7.7]{bensonLocallyDualisableModular2024}.

\begin{theorem}\label{thm:characDualizableTor}
    Let $\category{T}$ be an $R$-linear, Noetherian, tensor-triangulated, compactly-rigidly generated category. Assume $\category{T}^c$ has a strong generator $g$. Then the following conditions are equivalent for $x$ in $\Gamma_{\ida}\category{T}$ and a graded ideal $\ida\subset R$:
    \begin{enumerate}[label=(\roman*)]
        \item\label{item:cDTDualizable} $x$ is dualizable in $\Gamma_{\ida}\category{T}$;
        \item\label{item:cDTFinGenR} for any compact object $c \in (\Gamma_{\ida}\category{T})^c$, $\Hom_{\category{T}}(c, x)$ is a finitely generated $R$-module;
        \item\label{item:cDTThick} $x \in \thick(\Gamma_{\ida}g)$.
    \end{enumerate}
    If, moreover, $\category{T}$ is compactly cogenerated, the statements are also equivalent to:
    \begin{enumerate}[label=(\roman*)]
		    \setcounter{enumi}{3}
        \item\label{item:cDTFinGenRb} for any compact object $c \in (\Gamma_{\ida}\category{T})^c$, $\Hom_{\category{T}}(x, c)$ is a finitely generated $R$-module;
        \item\label{item:cDTFinGenRComp} for any compact object $c \in \category{T}^c$, $\Hom_{\category{T}}(x, c)$ is finitely generated as an $\comp{R}$-module.
    \end{enumerate}
\end{theorem}
\begin{proof}
    The implications $\ref{item:cDTDualizable} \Rightarrow \ref{item:cDTFinGenR}$ and $\ref{item:cDTDualizable} \Rightarrow \ref{item:cDTFinGenRb}$ are given by \Cref{cor:dualNoethTor}. Since $\category{T}$ is compactly generated with a strong generator of compacts, it is strictly Noetherian, thus \Cref{prop:strictRightNoethDual} provides the converse implication $\ref{item:cDTFinGenR} \Rightarrow \ref{item:cDTDualizable}$. Statement $\ref{item:cDTThick}$ implies $\ref{item:cDTDualizable}$ by \Cref{prop:tensorCompTor}. The converse implication $\ref{item:cDTDualizable} \Rightarrow \ref{item:cDTThick}$ is \Cref{thm:strongGenDual}.
		
		The equivalence $\ref{item:cDTFinGenRb} \Leftrightarrow \ref{item:cDTFinGenRComp}$ comes from \Cref{prop:torAllCompactNoeth}. 
    When $\category{T}$ is compactly cogenerated, the implication $\ref{item:cDTFinGenRb} \Rightarrow \ref{item:cDTDualizable}$ holds by \Cref{prop:strictLeftNoethDual}.
\end{proof}

The dual characterizations hold for dualizable complete objects in $\Lambda^{\ida}\category{T}$; compare to \autocite[Proposition 7.10]{bensonLocallyDualisableModular2024}. 

\begin{theorem}\label{thm:characDualizableComp}
    Let $\category{T}$ be a $R$-linear, Noetherian, tensor-triangulated, compactly-rigidly generated category. Assume $\category{T}^c$ has a strong generator $g$.
Then the following conditions are equivalent for $x$ in $\Gamma_{\ida}\category{T}$ and a graded ideal $\ida\subset R$:
    \begin{enumerate}[label=(\roman*)]
    \item\label{item:cDCDualizable} x is dualizable in $\Lambda^{\ida}\category{T}$;
    \item\label{item:cDCFinGenR} for any compact object $c \in (\Lambda^{\ida}\category{T})^c$, $\Hom_{\category{T}}(c, x)$ is a finitely generated $R$-module;
    \item\label{item:cDCFinGenRComp} for any compact object $c \in \category{T}^c$, $\Hom_{\category{T}}(c, x)$ is a finitely generated $\comp{R}$-module;
    \item\label{item:cDCThick} $x \in \thick(\Lambda^{\ida}g)$. 
    \end{enumerate}
    If, moreover, $\category{T}$ is compactly cogenerated, the statements are also equivalent to:
    \begin{enumerate}[label=(\roman*)]
		\setcounter{enumi}{4}
    \item\label{item:cDCFinGenRb} for any compact object $c \in (\Lambda^{\ida}\category{T})^c$, $\Hom_{\category{T}}(x, c)$ is a finitely generated $R$-module;
    \end{enumerate}
\end{theorem}

\begin{proof}
    The implications $\ref{item:cDCDualizable} \Rightarrow \ref{item:cDCFinGenR}$ and $\ref{item:cDCDualizable} \Rightarrow \ref{item:cDCFinGenRb}$ are given by \Cref{prop:dualNoethComp}. Since $\category{T}$ is compactly generated with a strong generator of compacts, it is strictly Noetherian, thus \Cref{prop:strictRightNoethDual} provides the converse implication $\ref{item:cDCFinGenR} \Rightarrow \ref{item:cDCDualizable}$.
		The equivalence $\ref{item:cDCFinGenR} \Leftrightarrow \ref{item:cDCFinGenRComp}$ comes from \Cref{prop:complAllCompactNoeth}. 
		Statement $\ref{item:cDCDualizable}$ implies $\ref{item:cDCThick}$ by \Cref{thm:strongGenDual}, while the converse implication $\ref{item:cDCDualizable} \Rightarrow \ref{item:cDCThick}$ is given by \Cref{prop:tensorCompTor}.
		
    When $\category{T}$ is compactly cogenerated, the implication $\ref{item:cDCFinGenRb} \Rightarrow \ref{item:cDCDualizable}$ holds by \Cref{prop:strictLeftNoethDual}.
\end{proof}

\subsection{Representability theorems}

In the presence of strong generator, we also have another way to recover (as $R$-linear categories) torsion dualizable modules from the torsion compact ones and vice versa, complementary to~\S\ref{prop:completion-on-homs}. This is closely related to representability theorems of Bondal and Van den Bergh \autocite[Theorem~1.3 and Appendix~A]{bondalVanDenBergh2003}, Rouquier~\autocite[Corollary~7.50]{rouquierDimensions2008} and Neeman~\autocite{neeman2Brown2025} and makes the relation between $(\Gamma_\ida\category{T})^c$ and $(\Gamma_\ida\category{T})^d$ looks similar as the one between $\category{D}^\per(X)$ and $\category{D}^\bnd(\category{coh}X)$ for a nice enough Noetherian scheme. However, our methods are very different as, for example unlike \autocite[Theorem 0.4]{neeman2Brown2025}, we cannot expect the existence of non-trivial $t$-structures since some power of the suspension functor in our setting may well be isomorphic to the identity.

Let again $\category{T}$ be a compactly-rigidly generated tensor-triangulated Noetherian $R$-linear category such that $\category{T}^c$ admits a strong generator. Let $\ida\subset R$ be a homogeneous ideal and denote by $\category{tors}R\subset\catmod R$ the full subcategory of finitely generated modules which are annihilated by some power of $\ida$. We will prove that the graded $R$-linear (= enriched over $\catMod R$ in the sense of \autocite{kellyEnrichedCT2005}) bifunctor
\begin{equation}\label{eq:perfect-pairing}
  \Hom_\category{T}(-,-)\colon (\Gamma_\ida\category{T})^{c,\mathrm{op}}\times (\Gamma_\ida\category{T})^d \to \category{tors} R
\end{equation}
(which is well-defined by \Cref{cor:dualNoethTor}) is a ``perfect pairing''. In the following theorem, we will denote the graded $R$-linear category of graded $R$-linear functors between two graded $R$-linear categories $\category{C},\category{D}$ by $\category{Fun}_R(\category{C},\category{D})$.

\begin{theorem}\label{thm:perfectPairing}
The following holds in the setting of the previous paragraph:

\begin{enumerate}[label=(\arabic*)]
\item\label{item:repres-dual} The restricted Yoneda functor induced by~\eqref{eq:perfect-pairing},
\begin{align*}
  h\colon (\Gamma_\ida\category{T})^d &\to \category{Fun}_R((\Gamma_\ida\category{T})^{c,\mathrm{op}}, \category{tors} R),
  \\
	d &\mapsto \Hom_\category{T}(-,d)_{|(\Gamma_\ida\category{T})^{c}},
\end{align*}
is fully faithful, and the essential image consists precisely of the cohomological functors. In particular, each cohomological graded $R$-linear functor $H\colon (\Gamma_\ida\category{T})^{c,\mathrm{op}}\to \category{tors} R$ is representable by a unique (up to a canonical isomorphism) object $d\in (\Gamma_\ida\category{T})^d$.

\item\label{item:repres-comp} The restricted Yoneda functor induced by~\eqref{eq:perfect-pairing},
\begin{align*}
  \tilde{h}\colon (\Gamma_\ida\category{T})^{c,\mathrm{op}} &\to \category{Fun}_R((\Gamma_\ida\category{T})^d, \category{tors} R),
  \\
	c &\mapsto \Hom_\category{T}(c,-)_{|(\Gamma_\ida\category{T})^{d}},
\end{align*}
is fully faithful, and the essential image consists precisely of the homological functors. In particular, each homological graded $R$-linear functor $\tilde{H}\colon (\Gamma_\ida\category{T})^d\to \category{tors} R$ is corepresentable by a unique (up to a canonical isomorphism) object $c\in (\Gamma_\ida\category{T})^c$.
\end{enumerate}
\end{theorem}

\begin{proof}
The functor $h$ is fully faithful by \Cref{prop:yoneda-restricted-fully-faithful}, while $\tilde{h}$ is fully faithful by the (enriched) Yoneda lemma since $(\Gamma_\ida\category{T})^c\subset(\Gamma_\ida\category{T})^d$. Clearly the essential images of both $h$ and $\tilde{h}$ are contained in the full subcategories of (co)homological functors. So we are left with proving the (co)representability parts of~\ref{item:repres-dual} and~\ref{item:repres-comp}.

Let us focus on part~\ref{item:repres-dual} first. If $H$ is a cohomological functor as above, it suffices to prove by~\Cref{prop:day-conv-dualizable,prop:essential-image-dualizable} that for each $c\in(\Gamma_\ida\category{T})^{c}$, the functor $h(c)\boxtimes H\cong H(\SW{c}\otimes-)\colon (\Gamma_\ida\category{T})^{c,\mathrm{op}}\to\category{tors}R$ is representable by a torsion compact object. However, the latter functor is in fact well-defined on the category of all dualizable objects by \Cref{prop:compactRigid}, so we have
\[
  (\Gamma_\ida\category{T})^{d,\mathrm{op}} \xrightarrow{\SW{c}\otimes-}
	(\Gamma_\ida\category{T})^{c,\mathrm{op}} \xrightarrow{\quad F\quad}
	\category{tors}R \subset \catmod R.
\]
Now $(\Gamma_\ida\category{T})^{d,\mathrm{op}}$ is strongly generated by \Cref{thm:strongGenDual}, and so there is some $x\in (\Gamma_\ida\category{T})^d$ such that $H(\SW{c}\otimes-)\cong\Hom_\category{T}(-,x)$ on $(\Gamma_\ida\category{T})^d$ by \autocite[Theorem 2.7]{letzBrownRepresentabilityTriangulated2023}. Finally, since $\ida^{[n']}$ annihilates $\emorp(x)\cong H(\SW{c}\otimes x)$ for large enough $n'$, it follows from \Cref{lemma:koszul-split} that $x$ is a direct summand of $x\kos\ida^{[n]}$ for a potentially even larger $n$, so $x$ is compact.

Finally, we prove the representability in~\ref{item:repres-comp}. If $\tilde{H}\colon (\Gamma_\ida\category{T})^d\to \category{tors} R$ is homological, we can again invoke the strong generation of $(\Gamma_\ida\category{T})^{d,\mathrm{op}}$ and \autocite[Theorem 2.7]{letzBrownRepresentabilityTriangulated2023} to obtain $y\in(\Gamma_\ida\category{T})^d$ such that $\tilde{H}\cong\Hom_\category{T}(y,-)$. Now $\emorp(y)\cong\tilde{H}(y)$ is $\ida$-power torsion, which implies that $y\in(\Gamma_\ida\category{T})^c$ by the same argument as in the previous paragraph.
\end{proof}

\subsection{Local regularity}\label{subsec:local-regularity}

Let $\category{T}$ be an $R$-linear, Noetherian, compactly-rigidly generated category. The \emph{local cohomology} at point $\ideal{p} \in \Spec R$ is the functor $\Gamma_{\mathcal{V}(\ideal{p})}L_{\mathcal{Z}(\ideal{p})}$ with
\[
    \mathcal{Z}(\ideal{p}) =\{ \ideal{q} \in \Spec R \: | \: \ideal{q} \nsubseteq \ideal{p} \} \qquad \mathcal{V}(\ideal{p}) = \{ \ideal{q} \in \Spec R \: | \: \ideal{p} \subset \ideal{q} \}
\]
The category $\category{T}$ is \emph{locally regular} if there is a compact generator $g$ of $\category{T}$ such that for each $\ideal{p} \in \Spec R$, the category $\thick(L_{\mathcal{Z}(\ideal{p})}\Gamma_{\mathcal{V}(\ideal{p})}g)$ is strongly generated~\autocite[Definition 7.1]{bensonLocallyDualisableModular2024}.

This relatively technical definition is crucial for the arguments in \autocite{bensonLocallyDualisableModular2024}.
Here we provide a simple criterion to check whether a category is locally regular. For instance, by applying it to $\category{KInj}(\field G)$, with $\field$ a field and $G$ a finite group, we recover \autocite[Theorem 8.3]{bensonLocallyDualisableModular2024}, which had been proved using more sophisticated arguments, and our argument applies equally well to finite group schemes.

\begin{proposition}\label{prop:stongGenCompactLocalRegularity}
    Assume that $\category{T}^c$ is strongly generated. Then $\category{T}$ is locally regular.
\end{proposition}
\begin{proof}
    Fix $\ideal{p} \in \Spec R$. Note that by the discussion in \autocite[\S2]{bensonSpectrumLocalDualisable2025}, the category $\category{L} = L_{\mathcal{Z}(\ideal{p})}\category{T}$ is $R_{\ideal{p}}$-linear, Noetherian (over $R_{\ideal{p}}$) and compactly-rigidly generated. Moreover, if $g$ is a strong generator of $\category{T}^c$, $L_{\mathcal{Z}(\ideal{p})}g$ is a strong generator of $\category{L}^c$ since the functor $\category{T}^c\to\category{L}^c$ is essentially surjective up to summands (see e.g.\ \autocite[\S4]{balmerFaviTelescope2011}). Hence $\thick(\Gamma_{\mathcal{V}(\ideal{p})}L_{\mathcal{Z}(\ideal{p})}g)$ is precisely the class of dualizable objects of $\Gamma_{\mathcal{V}(\ideal{p})}\category{L}$ and it is strongly generated by \Cref{thm:strongGenDual}.
\end{proof}

\subsection{Case study: finite group algebras}
\label{subsec:KInjkG}

Let $G$ be a finite group. We consider its group algebra $\field G$ over a field $\field$. The cohomology ring $R=H^*(G, \field)$ is known to be Noetherian \autocite{EvensCohomologyRing1961}. The homotopy category of injective $\field G$-modules $\category{KInj}(\field G)$, endowed with the canonical action of $H^*(G, \field)$, is then Noetherian, tensor-triangulated and compactly-rigidly generated; see \autocite[\S10]{bensonLocalCohomologySupport2008}. The category of compacts can be identified with $\category{D}^\bnd(\catmod(\field G))$, the bounded category of finitely generated modules over $\field G$, via the functor
\begin{equation}\label{eq:inj-res-functor}
\mathbf{i}\colon \category{D}^\bnd(\catmod(\field G)) \to \category{KInj}(\field G)
\end{equation}
which assigns to a complex $x\in\category{D}^\bnd(\catmod(\field G))$ an injective resolution $\mathbf{i}x$ (see~\autocite[Proposition~2.3]{krauseStableDerived2005}).
An injective resolution $\mathbf{i}g$ of the direct sum $g$ of all simple modules is a strong generator of the compact objects.

Note that the graded ring $H^*(G; \field)$ is concentrated in non-negative degrees, with $H^0(G; \field) \cong \field$. Thus, any proper ideal $\ida$ is concentrated in positive degrees, and any module whose degree is bounded below is $\ida$-complete. As a consequence, we can observe that the dualizable $\ida$-complete objects of $\category{KInj}(\field G)$ are precisely the compacts, independent of $\ida$.
This generalizes \autocite[Example~3.29]{balmerSandersTateIVT2025} by showing that, following the philosophy there, $\category{KInj}(\field G)$ should be considered as `$\ida$-complete' with respect to any proper homogeneous ideal $\ida$, not just the maximal one.

\begin{proposition}\label{prop:complete-independent}
Let $G$ be a finite group, $\field$ a field and $\ida\subset H^*(G, \field)$ a any proper ideal. Then
\[
    (\Lambda^{\ida}\category{KInj}(\field G))^d = \thick(\mathbf{i}g) = \operatorname{\mathbf{i}}\category{D}^\bnd(\catmod(\field G)).
\]
\end{proposition}

\begin{remark}
Let us stress that the category of compact objects $(\Lambda^\ida)$ is typically smaller as
\[
(\Lambda^\ida\category{KInj}(\field G))^c=
(\Gamma_\ida\category{KInj}(\field G))^c=
\operatorname{\mathbf{i}}\category{D}^\bnd(\catmod(\field G))\cap\Gamma_\ida\category{KInj}(\field G);
\]
recall \Cref{remark:torsion-compacts}.
\end{remark}

\begin{proof}[Proof of \Cref{prop:complete-independent}]
It follows from the discussion just before the lemma that the $R$-module $\Hom(\mathbf{i}g, \mathbf{i}g) \cong \bigoplus_{s,s' \text{ simple}} \Ext^*_{\field G}(s, s')$ is $\ida$-complete. Hence $\Hom(\mathbf{i}c, \mathbf{i}g)$ is $\ida$-complete for any $c\in\category{D}^\bnd(\catmod(\field G))$ and so $g\cong\Lambda^{\ida}g$ by \Cref{thm:charPTandDC}\,\ref{item:charDerivedComplete}. We conclude by \Cref{thm:strongGenDual}.
\end{proof}

We also deduce a characterization of dualizable $\ida$-torsion objects as an immediate consequence of \Cref{thm:characDualizableTor}.

\begin{proposition}
    Let $\ida$ be a homogeneous ideal of $H^*(G, \field)$ and $x$ be an object of $\Gamma_{\ida}\category{KInj}(\field G)$. The following propositions are equivalent:
    \begin{enumerate}
    \item $x$ is dualizable in $\Gamma_{\ida}\category{KInj}(\field G)$;
    \item for any compact object $c \in (\Gamma_{\ida}\category{KInj}(\field G))^c$, $\Hom_{\category{KInj}(\field G)}(c, x)$ is a finitely generated $H^*(G, \field)$-module;
    \item $x \in \thick(\Gamma_{\ida}(\mathbf{i}s), \, s \text{ simple } \field G\text{-module})$.
    \end{enumerate}
\end{proposition}

Since $(\Gamma_{\ida}\category{KInj}(\field G))^d\simeq(\Lambda^{\ida}\category{KInj}(\field G))^d$, we obtain for any closed subset $V \subset \Spec H^*(G; \field)$ a subcategory $(\Gamma_V\category{KInj}(\field G))^d$ of $\category{KInj}(\field G)$, equivalent as a tensor-triangulated category to $\category{D}^\bnd(\catmod(\field G))$. Distinct closed subsets yield distinct subcategories.
Let us illustrate this phenomenon on an example.

\begin{example}
Let $G$ be a finite group, $\field$ a field and $\ideal{m}=H^{>0}(G, \field)$ be the unique maximal graded ideal. Then $(\Gamma_\ideal{m}\category{KInj}(\field G))^c=\thick(\field G)$ and
\[
L_\ideal{m}\category{KInj}(\field G)=\category{KInj}_\mathrm{ac}(\field G),
\]
the category of acyclic complexes of injectives. It follows that $\Gamma_\ideal{m}\category{KInj}(\field G)$ consist of complexes of injectives that are \emph{left} Hom-orthogonal to the acyclic complexes, while $\Lambda^\ideal{m}\category{KInj}(\field G)$ consist of complexes of injectives that are \emph{right} Hom-orthogonal to $\category{KInj}_\mathrm{ac}(\field G)$. So $\Gamma_\ideal{m}\category{KInj}(\field G)$ consists of what it usually called the homotopy projective complexes, while $\Lambda^\ideal{m}\category{KInj}(\field G)$ is the class of homotopy injective complexes (and both the categories are equivalent to $\category{D}(\catMod\field G)$).

Turning to dualizable objects, we know from \Cref{prop:complete-independent} that dualizable $\ideal{m}$-complete objects are precisely the compacts, i.e. the essential image of the functor~\eqref{eq:inj-res-functor}.
On the other hand, the $\ideal{m}$-torsion dualizable objects coincide with the essential image of the functor
\[
\mathbf{p}\colon \category{D}^\bnd(\catmod(\field G)) \to \category{KInj}(\field G),
\]
which sends a bounded complex $x\in\category{D}^\bnd(\catmod(\field G))$ to a projective resolution $\mathbf{p}x$ of~$x$.
\end{example}

\printbibliography

@ARTICLE{balmerSpectraCube2010,
    AUTHOR = {Balmer, Paul},
     TITLE = {Spectra, spectra, spectra---tensor triangular spectra versus
              {Z}ariski spectra of endomorphism rings},
   JOURNAL = {Algebr. Geom. Topol.},
  FJOURNAL = {Algebraic \& Geometric Topology},
    VOLUME = {10},
      YEAR = {2010},
    NUMBER = {3},
     PAGES = {1521--1563},
      ISSN = {1472-2747,1472-2739},
       DOI = {10.2140/agt.2010.10.1521},
}

@article {balmerTTGeometryFoundation,
    AUTHOR = {Balmer, Paul},
     TITLE = {The spectrum of prime ideals in tensor triangulated
              categories},
   JOURNAL = {J. Reine Angew. Math.},
  FJOURNAL = {Journal f\"ur die Reine und Angewandte Mathematik. [Crelle's
              Journal]},
    VOLUME = {588},
      YEAR = {2005},
     PAGES = {149--168},
      ISSN = {0075-4102,1435-5345},
       DOI = {10.1515/crll.2005.2005.588.149},
}

@INPROCEEDINGS{balmerTensorTriangularGeometry2010,
  AUTHOR = {Balmer, Paul},
  PUBLISHER = {Hindustan Book Agency, New Delhi},
  BOOKTITLE = {Proceedings of the {{International Congress}} of {{Mathematicians}}. {{Volume II}}},
  DATE = {2010},
  ISBN = {978-81-85931-08-3; 978-981-4324-32-8; 981-4324-32-9},
  PAGES = {85--112},
  TITLE = {Tensor Triangular Geometry},
}

@article {balmerKrauseStevensonSmashingFrame2020,
    AUTHOR = {Balmer, Paul and Krause, Henning and Stevenson, Greg},
     TITLE = {The frame of smashing tensor-ideals},
   JOURNAL = {Math. Proc. Cambridge Philos. Soc.},
  FJOURNAL = {Mathematical Proceedings of the Cambridge Philosophical
              Society},
    VOLUME = {168},
      YEAR = {2020},
    NUMBER = {2},
     PAGES = {323--343},
      ISSN = {0305-0041,1469-8064},
       DOI = {10.1017/s0305004118000725},
       URL = {https://doi.org/10.1017/s0305004118000725},
}

@ONLINE{balmerSandersPerfectComplexesCompletion2024,
  AUTHOR = {Balmer, Paul and Sanders, Beren},
  URL = {http://arxiv.org/abs/2411.14761},
  DATE = {2024-11-22},
  EPRINT = {2411.14761},
  EPRINTCLASS = {math},
  EPRINTTYPE = {arXiv},
  PUBSTATE = {prepublished},
  TITLE = {Perfect Complexes and Completion},
}

@article {balmerSandersTateIVT2025,
    AUTHOR = {Balmer, Paul and Sanders, Beren},
     TITLE = {The {T}ate {I}ntermediate {V}alue {T}heorem},
   JOURNAL = {Adv. Math.},
  FJOURNAL = {Advances in Mathematics},
    VOLUME = {483},
      YEAR = {2025},
     PAGES = {Paper No. 110675},
      ISSN = {0001-8708,1090-2082},
       DOI = {10.1016/j.aim.2025.110675},
}

@ARTICLE{balmerFaviTelescope2011,
    AUTHOR = {Balmer, Paul and Favi, Giordano},
     TITLE = {Generalized tensor idempotents and the telescope conjecture},
   JOURNAL = {Proc. Lond. Math. Soc. (3)},
  FJOURNAL = {Proceedings of the London Mathematical Society. Third Series},
    VOLUME = {102},
      YEAR = {2011},
    NUMBER = {6},
     PAGES = {1161--1185},
      ISSN = {0024-6115,1460-244X},
       DOI = {10.1112/plms/pdq050},
}

@ONLINE{barthelLattices2023,
  AUTHOR = {Barthel, Tobias and Benson, David J. and Iyengar, Srikanth B. and Krause, Henning and Pevtsova, Julia},
  URL = {http://arxiv.org/abs/2307.16271},
  DATE = {2023-10-14},
  EPRINT = {2307.16271},
  EPRINTCLASS = {math},
  EPRINTTYPE = {arXiv},
  LANGID = {english},
  PUBSTATE = {prepublished},
  TITLE = {Lattices over finite group schemes and stratification},
}

@ARTICLE{beligiannisRelativeHA2000,
    AUTHOR = {Beligiannis, Apostolos},
     TITLE = {Relative homological algebra and purity in triangulated
              categories},
   JOURNAL = {J. Algebra},
  FJOURNAL = {Journal of Algebra},
    VOLUME = {227},
      YEAR = {2000},
    NUMBER = {1},
     PAGES = {268--361},
      ISSN = {0021-8693,1090-266X},
       DOI = {10.1006/jabr.1999.8237},
}

@ARTICLE{bensonColocalizingSubcategoriesCosupport2012,
  AUTHOR = {Benson, David J. and Iyengar, Srikanth B. and Krause, Henning},
  DATE = {2012},
  ISSN = {0075-4102,1435-5345},
  JOURNAL = {J. Reine Angew. Math.},
  FJOURNAL = {Journal Fur Die Reine Und Angewandte Mathematik},
  PAGES = {161--207},
  TITLE = {Colocalizing Subcategories and Cosupport},
  VOLUME = {673},
}

@ARTICLE{bensonLocalCohomologySupport2008,
  TITLE = {Local Cohomology and Support for Triangulated Categories},
  AUTHOR = {Benson, Dave and Iyengar, Srikanth B. and Krause, Henning},
  DATE = {2008},
  JOURNAL = {Ann. Sci. Éc. Norm. Supér. (4)},
  FJOURNAL = {Annales Scientifiques de l'École Normale Supérieure. Quatrième Série},
  VOLUME = {41},
  NUMBER = {4},
  PACES = {573--619},
  ISSN = {0012-9593,1873-2151},
  DOI = {10.24033/asens.2076},
}

@ONLINE{bensonLocallyDualisableModular2024,
  AUTHOR = {Benson, David J. and Iyengar, Srikanth B. and Krause, Henning and Pevtsova, Julia},
  URL = {http://arxiv.org/abs/2404.14672},
  DATE = {2024-04-23},
  PUBSTATE = {prepublished},
  TITLE = {Locally Dualisable Modular Representations and Local Regularity},
}

@incollection {bensonLocallyDualisableLocalAlgebra2024,
    AUTHOR = {Benson, Dave and Iyengar, Srikanth B. and Krause, Henning and
              Pevtsova, Julia},
     TITLE = {Local dualisable objects in local algebra},
 BOOKTITLE = {Triangulated categories in representation theory and
              beyond---the {A}bel {S}ymposium 2022},
    SERIES = {Abel Symp.},
    VOLUME = {17},
     PAGES = {85--103},
 PUBLISHER = {Springer, Cham},
      YEAR = {2024},
      ISBN = {978-3-031-57788-8; 978-3-031-57789-5},
       DOI = {10.1007/978-3-031-57789-5\_3},
}

@article {bensonStratifyingTria2011,
    AUTHOR = {Benson, Dave and Iyengar, Srikanth B. and Krause, Henning},
     TITLE = {Stratifying triangulated categories},
   JOURNAL = {J. Topol.},
  FJOURNAL = {Journal of Topology},
    VOLUME = {4},
      YEAR = {2011},
    NUMBER = {3},
     PAGES = {641--666},
      ISSN = {1753-8416,1753-8424},
       DOI = {10.1112/jtopol/jtr017},
}

@online{bensonSpectrumLocalDualisable2025,
  TITLE = {The Spectrum of Local Dualisable Modular Representations},
  AUTHOR = {Benson, Dave and Iyengar, Srikanth B. and Krause, Henning and Pevtsova, Julia},
  DATE = {2025-05-25},
  eprint = {2505.19368},
  eprinttype = {arXiv},
  eprintclass = {math},
  url = {http://arxiv.org/abs/2505.19368},
  pubstate = {prepublished},
}

@article {bondalVanDenBergh2003,
    AUTHOR = {Bondal, A. and van den Bergh, M.},
     TITLE = {Generators and representability of functors in commutative and
              noncommutative geometry},
   JOURNAL = {Mosc. Math. J.},
  FJOURNAL = {Moscow Mathematical Journal},
    VOLUME = {3},
      YEAR = {2003},
    NUMBER = {1},
     PAGES = {1--36, 258},
      ISSN = {1609-3321,1609-4514},
       DOI = {10.17323/1609-4514-2003-3-1-1-36},
}

@book {brunsHerzogCMBook1993,
    AUTHOR = {Bruns, Winfried and Herzog, J\"urgen},
     TITLE = {Cohen-{M}acaulay rings},
    SERIES = {Cambridge Studies in Advanced Mathematics},
    VOLUME = {39},
 PUBLISHER = {Cambridge University Press, Cambridge},
      YEAR = {1993},
     PAGES = {xii+403},
      ISBN = {0-521-41068-1},
}

@ARTICLE{christensenIdeals1998,
    AUTHOR = {Christensen, J. Daniel},
     TITLE = {Ideals in triangulated categories: phantoms, ghosts and
              skeleta},
   JOURNAL = {Adv. Math.},
  FJOURNAL = {Advances in Mathematics},
    VOLUME = {136},
      YEAR = {1998},
    NUMBER = {2},
     PAGES = {284--339},
      ISSN = {0001-8708,1090-2082},
       DOI = {10.1006/aima.1998.1735},
}

@article {dellAmbrogioStanleyWeaklyRegular2016,
    AUTHOR = {Dell'Ambrogio, Ivo and Stanley, Donald},
     TITLE = {Affine weakly regular tensor triangulated categories},
   JOURNAL = {Pacific J. Math.},
  FJOURNAL = {Pacific Journal of Mathematics},
    VOLUME = {285},
      YEAR = {2016},
    NUMBER = {1},
     PAGES = {93--109},
      ISSN = {0030-8730,1945-5844},
       DOI = {10.2140/pjm.2016.285.93},
}

@article {EvensCohomologyRing1961,
    AUTHOR = {Evens, Leonard},
     TITLE = {The cohomology ring of a finite group},
   JOURNAL = {Trans. Amer. Math. Soc.},
  FJOURNAL = {Transactions of the American Mathematical Society},
    VOLUME = {101},
      YEAR = {1961},
     PAGES = {224--239},
      ISSN = {0002-9947,1088-6850},
       DOI = {10.2307/1993372},
}

@book {foxbyHolmJoergensen2024,
    AUTHOR = {Christensen, Lars Winther and Foxby, Hans-Bj\o rn and Holm,
              Henrik},
     TITLE = {Derived category methods in commutative algebra},
    SERIES = {Springer Monographs in Mathematics},
 PUBLISHER = {Springer, Cham},
      YEAR = {2024},
     PAGES = {xxiii+1119},
      ISBN = {978-3-031-77452-2; 978-3-031-77453-9},
       DOI = {10.1007/978-3-031-77453-9},
}

@ARTICLE{geigleLenzingPerpendicular1991,
    AUTHOR = {Geigle, Werner and Lenzing, Helmut},
     TITLE = {Perpendicular categories with applications to representations
              and sheaves},
   JOURNAL = {J. Algebra},
  FJOURNAL = {Journal of Algebra},
    VOLUME = {144},
      YEAR = {1991},
    NUMBER = {2},
     PAGES = {273--343},
      ISSN = {0021-8693,1090-266X},
       DOI = {10.1016/0021-8693(91)90107-J},
}

@incollection {greenleesTate2001,
    AUTHOR = {Greenlees, J. P. C.},
     TITLE = {Tate cohomology in axiomatic stable homotopy theory},
 BOOKTITLE = {Cohomological methods in homotopy theory ({B}ellaterra, 1998)},
    SERIES = {Progr. Math.},
    VOLUME = {196},
     PAGES = {149--176},
 PUBLISHER = {Birkh\"auser, Basel},
      YEAR = {2001},
      ISBN = {3-7643-6588-9},
       DOI = {10.1007/978-3-0348-8312-2\_12},
}

@incollection {greenleesMayCompletions1995,
    AUTHOR = {Greenlees, J. P. C. and May, J. P.},
     TITLE = {Completions in algebra and topology},
 BOOKTITLE = {Handbook of algebraic topology},
     PAGES = {255--276},
 PUBLISHER = {North-Holland, Amsterdam},
      YEAR = {1995},
      ISBN = {0-444-81779-4},
       DOI = {10.1016/B978-044481779-2/50008-0},
}

@article {greenleesMayLocalHomology1992,
    AUTHOR = {Greenlees, J. P. C. and May, J. P.},
     TITLE = {Derived functors of {$I$}-adic completion and local homology},
   JOURNAL = {J. Algebra},
  FJOURNAL = {Journal of Algebra},
    VOLUME = {149},
      YEAR = {1992},
    NUMBER = {2},
     PAGES = {438--453},
      ISSN = {0021-8693,1090-266X},
       DOI = {10.1016/0021-8693(92)90026-I},
}

@ARTICLE{grothendieckEGA3a,
    AUTHOR = {Grothendieck, A.},
     TITLE = {\'El\'ements de g\'eom\'etrie alg\'ebrique. {III}. \'Etude
              cohomologique des faisceaux coh\'erents. {I}.},
   JOURNAL = {Inst. Hautes \'Etudes Sci. Publ. Math.},
  FJOURNAL = {Institut des Hautes \'Etudes Scientifiques. Publications
              Math\'ematiques},
    NUMBER = {11},
      YEAR = {1961},
     PAGES = {167},
      ISSN = {0073-8301,1618-1913},
       URL = {http://www.numdam.org/item?id=PMIHES_1961__11__167_0},
}

@ARTICLE{hoveyLockridge2011,
    AUTHOR = {Hovey, Mark and Lockridge, Keir},
     TITLE = {The ghost and weak dimensions of rings and ring spectra},
   JOURNAL = {Israel J. Math.},
  FJOURNAL = {Israel Journal of Mathematics},
    VOLUME = {182},
      YEAR = {2011},
     PAGES = {31--46},
      ISSN = {0021-2172,1565-8511},
       DOI = {10.1007/s11856-011-0022-8},
}

@ARTICLE {hoveyPalmieriStricklandMemoir1997,
    AUTHOR = {Hovey, Mark and Palmieri, John H. and Strickland, Neil P.},
     TITLE = {Axiomatic stable homotopy theory},
   JOURNAL = {Mem. Amer. Math. Soc.},
  FJOURNAL = {Memoirs of the American Mathematical Society},
    VOLUME = {128},
      YEAR = {1997},
    NUMBER = {610},
     PAGES = {x+114},
      ISSN = {0065-9266,1947-6221},
       DOI = {10.1090/memo/0610},
}

@article {hoveyStricklandMoravaKTheories1999,
    AUTHOR = {Hovey, Mark and Strickland, Neil P.},
     TITLE = {Morava {$K$}-theories and localisation},
   JOURNAL = {Mem. Amer. Math. Soc.},
  FJOURNAL = {Memoirs of the American Mathematical Society},
    VOLUME = {139},
      YEAR = {1999},
    NUMBER = {666},
     PAGES = {viii+100},
      ISSN = {0065-9266,1947-6221},
       DOI = {10.1090/memo/0666},
}

@BOOK{jensenDerivedLim1972,
    AUTHOR = {Jensen, C. U.},
     TITLE = {Les foncteurs d\'eriv\'es de {$\varprojlim$} et leurs
              applications en th\'eorie des modules},
    SERIES = {Lecture Notes in Mathematics},
    VOLUME = {Vol. 254},
 PUBLISHER = {Springer-Verlag, Berlin-New York},
      YEAR = {1972},
     PAGES = {iv+103},
}

@article {kellyEnrichedCT2005,
    AUTHOR = {Kelly, G. M.},
     TITLE = {Basic concepts of enriched category theory},
      NOTE = {Reprint of the 1982 original [Cambridge Univ. Press,
              Cambridge; MR0651714]},
   JOURNAL = {Repr. Theory Appl. Categ.},
  FJOURNAL = {Reprints in Theory and Applications of Categories},
    NUMBER = {10},
      YEAR = {2005},
     PAGES = {vi+137},
}

@article {krauseBrownRepresentabilityCoherent2002,
    AUTHOR = {Krause, Henning},
     TITLE = {A {B}rown representability theorem via coherent functors},
   JOURNAL = {Topology},
  FJOURNAL = {Topology. An International Journal of Mathematics},
    VOLUME = {41},
      YEAR = {2002},
    NUMBER = {4},
     PAGES = {853--861},
      ISSN = {0040-9383},
       DOI = {10.1016/S0040-9383(01)00010-6},
}

@article {krauseCompletingDper2020,
    AUTHOR = {Krause, Henning},
     TITLE = {Completing perfect complexes},
      NOTE = {With appendices by Tobias Barthel, Bernhard Keller and Krause},
   JOURNAL = {Math. Z.},
  FJOURNAL = {Mathematische Zeitschrift},
    VOLUME = {296},
      YEAR = {2020},
    NUMBER = {3-4},
     PAGES = {1387--1427},
      ISSN = {0025-5874},
       DOI = {10.1007/s00209-020-02490-z},
}

@book {krauseHomologicalRT2022,
    AUTHOR = {Krause, Henning},
     TITLE = {Homological theory of representations},
    SERIES = {Cambridge Studies in Advanced Mathematics},
    VOLUME = {195},
 PUBLISHER = {Cambridge University Press, Cambridge},
      YEAR = {2022},
     PAGES = {xxxiv+482},
      ISBN = {978-1-108-83889-4},
}

@article {krauseTelescope2000,
    AUTHOR = {Krause, Henning},
     TITLE = {Smashing subcategories and the telescope conjecture---an
              algebraic approach},
   JOURNAL = {Invent. Math.},
  FJOURNAL = {Inventiones Mathematicae},
    VOLUME = {139},
      YEAR = {2000},
    NUMBER = {1},
     PAGES = {99--133},
      ISSN = {0020-9910,1432-1297},
       DOI = {10.1007/s002229900022},
}

@ARTICLE {krauseStableDerived2005,
    AUTHOR = {Krause, Henning},
     TITLE = {The stable derived category of a {N}oetherian scheme},
   JOURNAL = {Compos. Math.},
  FJOURNAL = {Compositio Mathematica},
    VOLUME = {141},
      YEAR = {2005},
    NUMBER = {5},
     PAGES = {1128--1162},
      ISSN = {0010-437X,1570-5846},
       DOI = {10.1112/S0010437X05001375},
}

@article {lauBSpectrumStacks2023,
    AUTHOR = {Lau, Eike},
     TITLE = {The {B}almer spectrum of certain {D}eligne-{M}umford stacks},
   JOURNAL = {Compos. Math.},
  FJOURNAL = {Compositio Mathematica},
    VOLUME = {159},
      YEAR = {2023},
    NUMBER = {6},
     PAGES = {1314--1346},
      ISSN = {0010-437X,1570-5846},
       DOI = {10.1112/s0010437x23007200},
}

@ARTICLE{letzBrownRepresentabilityTriangulated2023,
  AUTHOR = {Letz, Janina C.},
  TITLE = {Brown Representability for Triangulated Categories with a Linear Action by a Graded Ring},
  JOURNAL = {Arch. Math. (Basel)},
  FJOURNAL = {Archiv der Mathematik},
  VOLUME = {120},
  YEAR = {2023},
  NUMBER = {2},
  PAGES = {135--146},
  ISSN = {0003-889X,1420-8938},
  DOI = {10.1007/s00013-022-01800-7},
}

@article {letzStephanGenerationTime2025,
    AUTHOR = {Letz, Janina C. and Stephan, Marc},
     TITLE = {Generation time for biexact functors and {K}oszul objects in
              triangulated categories},
   JOURNAL = {Doc. Math.},
  FJOURNAL = {Documenta Mathematica},
    VOLUME = {30},
      YEAR = {2025},
    NUMBER = {2},
     PAGES = {379--415},
      ISSN = {1431-0635,1431-0643},
       DOI = {10.4171/dm/1000},
}

@book {lewisMaySteinbergerEquivariantHomotopy1986,
    AUTHOR = {Lewis, Jr., L. G. and May, J. P. and Steinberger, M. and
              McClure, J. E.},
     TITLE = {Equivariant stable homotopy theory},
    SERIES = {Lecture Notes in Mathematics},
    VOLUME = {1213},
      NOTE = {With contributions by J. E. McClure},
 PUBLISHER = {Springer-Verlag, Berlin},
      YEAR = {1986},
     PAGES = {x+538},
      ISBN = {3-540-16820-6},
       DOI = {10.1007/BFb0075778},
}

@book{loregianCoendCalculus2021,
  title = {({{Co}})End Calculus},
  author = {Loregian, Fosco},
  date = {2021},
  series = {London Mathematical Society Lecture Note Series},
  volume = {468},
  pages = {xxi+308},
  publisher = {Cambridge University Press, Cambridge},
  doi = {10.1017/9781108778657},
  isbn = {978-1-108-74612-0},
}

@ARTICLE{mathewGaloisHomotopy2016,
    AUTHOR = {Mathew, Akhil},
     TITLE = {The {G}alois group of a stable homotopy theory},
   JOURNAL = {Adv. Math.},
  FJOURNAL = {Advances in Mathematics},
    VOLUME = {291},
      YEAR = {2016},
     PAGES = {403--541},
      ISSN = {0001-8708,1090-2082},
       DOI = {10.1016/j.aim.2015.12.017},
}

@ARTICLE{nakamuraCosupports2019,
    AUTHOR = {Nakamura, Tsutomu},
     TITLE = {Cosupports and minimal pure-injective resolutions of affine
              rings},
   JOURNAL = {J. Algebra},
  FJOURNAL = {Journal of Algebra},
    VOLUME = {540},
      YEAR = {2019},
     PAGES = {306--316},
      ISSN = {0021-8693,1090-266X},
       DOI = {10.1016/j.jalgebra.2019.08.023},
}

@ARTICLE{naumannPolSeparableAlgebras2024,
    AUTHOR = {Naumann, Niko and Pol, Luca},
     TITLE = {Separable commutative algebras and {G}alois theory in stable
              homotopy theories},
   JOURNAL = {Adv. Math.},
  FJOURNAL = {Advances in Mathematics},
    VOLUME = {449},
      YEAR = {2024},
     PAGES = {Paper No. 109736, 67},
      ISSN = {0001-8708,1090-2082},
       DOI = {10.1016/j.aim.2024.109736},
}

@ONLINE{naumanPolRamziFractureSquare2024,
  AUTHOR = {Naumann, Niko and Pol, Luca and Ramzi, Maxime},
  URL = {http://arxiv.org/abs/2411.05467},
  DATE = {2024-11-08},
  EPRINT = {2411.05467},
  EPRINTCLASS = {math},
  EPRINTTYPE = {arXiv},
  LANGID = {english},
  PUBSTATE = {prepublished},
  TITLE = {A symmetric monoidal fracture square},
}

@ARTICLE {neemanColocalizing2011,
    AUTHOR = {Neeman, Amnon},
     TITLE = {Colocalizing subcategories of {$\mathbf{D}(R)$}},
   JOURNAL = {J. Reine Angew. Math.},
  FJOURNAL = {Journal f\"ur die Reine und Angewandte Mathematik. [Crelle's
              Journal]},
    VOLUME = {653},
      YEAR = {2011},
     PAGES = {221--243},
      ISSN = {0075-4102,1435-5345},
       DOI = {10.1515/CRELLE.2011.028},
}

@article {neemanMetrics2020,
    AUTHOR = {Neeman, Amnon},
     TITLE = {Metrics on triangulated categories},
   JOURNAL = {J. Pure Appl. Algebra},
  FJOURNAL = {Journal of Pure and Applied Algebra},
    VOLUME = {224},
      YEAR = {2020},
    NUMBER = {4},
     PAGES = {106206, 13},
      ISSN = {0022-4049,1873-1376},
       DOI = {10.1016/j.jpaa.2019.106206},
}

@article {neemanStrongGenerators2021,
    AUTHOR = {Neeman, Amnon},
     TITLE = {Strong generators in {$\mathbf{D}^{\mathrm {perf}}(X)$} and
              {$\mathbf{D}^b_{\mathrm {coh}}(X)$}},
   JOURNAL = {Ann. of Math. (2)},
  FJOURNAL = {Annals of Mathematics. Second Series},
    VOLUME = {193},
      YEAR = {2021},
    NUMBER = {3},
     PAGES = {689--732},
      ISSN = {0003-486X,1939-8980},
       DOI = {10.4007/annals.2021.193.3.1},
}

@ARTICLE{neemanBousfieldTechniques1996,
    AUTHOR = {Neeman, Amnon},
     TITLE = {The {G}rothendieck duality theorem via {B}ousfield's
              techniques and {B}rown representability},
   JOURNAL = {J. Amer. Math. Soc.},
  FJOURNAL = {Journal of the American Mathematical Society},
    VOLUME = {9},
      YEAR = {1996},
    NUMBER = {1},
     PAGES = {205--236},
      ISSN = {0894-0347,1088-6834},
       DOI = {10.1090/S0894-0347-96-00174-9},
}

@BOOK{neemanBook2001,
 author = {Neeman, Amnon},
 title = {Triangulated categories},
 fseries = {Annals of Mathematics Studies},
 series = {Ann. Math. Stud.},
 volume = {148},
 isbn = {0-691-08685-0; 0-691-08686-9},
 year = {2001},
 publisher = {Princeton, NJ: Princeton University Press},
 doi = {10.1515/9781400837212},
}

@online{neeman2Brown2025,
  title = {Triangulated categories with a single compact generator and two Brown representability theorems},
  author = {Neeman, Amnon},
  date = {2025-05-14},
  eprint = {1804.02240v5},
  eprinttype = {arXiv},
  eprintclass = {math},
  url = {http://arxiv.org/abs/1804.02240v5},
  pubstate = {prepublished},
}

@ARTICLE{portaShaulYekutieli2015,
    AUTHOR = {Porta, Marco and Shaul, Liran and Yekutieli, Amnon},
     TITLE = {Cohomologically cofinite complexes},
   JOURNAL = {Comm. Algebra},
  FJOURNAL = {Communications in Algebra},
    VOLUME = {43},
      YEAR = {2015},
    NUMBER = {2},
     PAGES = {597--615},
      ISSN = {0092-7872,1532-4125},
       DOI = {10.1080/00927872.2013.822506},
}

@ARTICLE{positselskiContraadjustedModulesContramodules2017,
  AUTHOR = {Positselski, Leonid},
  DOI = {10.17323/1609-4514-2017-17-3-385-455},
  DATE = {2017},
  ISSN = {1609-3321,1609-4514},
  FJOURNAL = {Moscow Mathematical Journal},
  NUMBER = {3},
  PAGES = {385--455},
  JOURNAL = {Mosc. Math. J.},
  TITLE = {Contraadjusted Modules, Contramodules, and Reduced Cotorsion Modules},
  VOLUME = {17},
}

@ARTICLE{positselskiDedualizingMGM2016,
    AUTHOR = {Positselski, Leonid},
     TITLE = {Dedualizing complexes and {MGM} duality},
   JOURNAL = {J. Pure Appl. Algebra},
  FJOURNAL = {Journal of Pure and Applied Algebra},
    VOLUME = {220},
      YEAR = {2016},
    NUMBER = {12},
     PAGES = {3866--3909},
      ISSN = {0022-4049,1873-1376},
       DOI = {10.1016/j.jpaa.2016.05.019},
}

@ARTICLE{positselskiRemarksDerivedComplete2023,
  AUTHOR = {Positselski, Leonid},
  DOI = {doi.org/10.1002/mana.202000140},
  DATE = {2023},
  ISSN = {0025-584X,1522-2616},
  FJOURNAL = {Mathematische Nachrichten},
  NUMBER = {2},
  PAGES = {811--839},
  JOURNAL = {Math. Nachr.},
  TITLE = {Remarks on Derived Complete Modules and Complexes},
  VOLUME = {296},
}

@article {rouquierDimensions2008,
    AUTHOR = {Rouquier, Rapha\"el},
     TITLE = {Dimensions of triangulated categories},
   JOURNAL = {J. K-Theory},
  FJOURNAL = {Journal of K-Theory. K-Theory and its Applications in Algebra,
              Geometry, Analysis \& Topology},
    VOLUME = {1},
      YEAR = {2008},
    NUMBER = {2},
     PAGES = {193--256},
      ISSN = {1865-2433,1865-5394},
       DOI = {10.1017/is007011012jkt010},
}

@MISC{thestacksprojectauthorsStacksProject,
  AUTHOR = {{The Stacks Project Authors}},
  URL = {https://stacks.math.columbia.edu},
  SHORTHAND = {Stacks},
  TITLE = {Stacks {{Project}}},
}

@ONLINE{salch2023,
  AUTHOR = {Salch, Andrew},
  URL = {http://arxiv.org/abs/1006.0048v4},
  DATE = {2023-05-07},
  EPRINT = {1006.0048v4},
  EPRINTCLASS = {math},
  EPRINTTYPE = {arXiv},
  PUBSTATE = {prepublished},
  TITLE = {Approximation of subcategories by abelian subcategories},
}

@article {sarochStovicekTelescope2008,
    AUTHOR = {{\v S}aroch, Jan and {\v S}{\v t}ov{\'\i}{\v c}ek, Jan},
     TITLE = {The countable telescope conjecture for module categories},
   JOURNAL = {Adv. Math.},
  FJOURNAL = {Advances in Mathematics},
    VOLUME = {219},
      YEAR = {2008},
    NUMBER = {3},
     PAGES = {1002--1036},
      ISSN = {0001-8708,1090-2082},
       DOI = {10.1016/j.aim.2008.05.012},
}

@ARTICLE{steenStevensonStrongGenerators2015,
    AUTHOR = {Steen, Johan and Stevenson, Greg},
     TITLE = {Strong generators in tensor triangulated categories},
   JOURNAL = {Bull. Lond. Math. Soc.},
  FJOURNAL = {Bulletin of the London Mathematical Society},
    VOLUME = {47},
      YEAR = {2015},
    NUMBER = {4},
     PAGES = {607--616},
      ISSN = {0024-6093,1469-2120},
       DOI = {10.1112/blms/bdv037},
}

@ARTICLE {stevensonRouquierGlobalDimension2025,
    AUTHOR = {Stevenson, Greg},
     TITLE = {Rouquier dimension versus global dimension},
   JOURNAL = {J. Pure Appl. Algebra},
  FJOURNAL = {Journal of Pure and Applied Algebra},
    VOLUME = {229},
      YEAR = {2025},
    NUMBER = {1},
     PAGES = {Paper No. 107827, 4},
      ISSN = {0022-4049,1873-1376},
       DOI = {10.1016/j.jpaa.2024.107827},
}

@ONLINE{stricklandLocCompSpectra,
  AUTHOR = {Strickland, Neil},
  URL = {http://arxiv.org/abs/2412.09361},
  DATE = {2024-12-12},
  EPRINT = {2412.09361},
  EPRINTCLASS = {math},
  EPRINTTYPE = {arXiv},
  LANGID = {english},
  PUBSTATE = {prepublished},
  TITLE = {Arithmetic localisation and completion of spectra},
}

@ARTICLE{thompsonCosupportComputations2018,
    AUTHOR = {Thompson, Peder},
     TITLE = {Cosupport computations for finitely generated modules over
              commutative noetherian rings},
   JOURNAL = {J. Algebra},
  FJOURNAL = {Journal of Algebra},
    VOLUME = {511},
      YEAR = {2018},
     PAGES = {249--269},
      ISSN = {0021-8693,1090-266X},
       DOI = {10.1016/j.jalgebra.2018.06.014},
}

@ARTICLE{vdKallenNoetherianBase2023,
    AUTHOR = {van der Kallen, Wilberd},
     TITLE = {A {F}riedlander-{S}uslin Theorem over a {N}oetherian Base Ring},
   JOURNAL = {Transform. Groups},
  FJOURNAL = {Transformation Groups},
      YEAR = {2023},
       DOI = {10.1007/s00031-023-09817-0},
}

@book {weibelIntro1994,
    AUTHOR = {Weibel, Charles A.},
     TITLE = {An introduction to homological algebra},
    SERIES = {Cambridge Studies in Advanced Mathematics},
    VOLUME = {38},
 PUBLISHER = {Cambridge University Press, Cambridge},
      YEAR = {1994},
     PAGES = {xiv+450},
      ISBN = {0-521-43500-5; 0-521-55987-1},
       DOI = {10.1017/CBO9781139644136},
}

@article {williamsonPolHtpyComplete2022,
    AUTHOR = {Pol, Luca and Williamson, Jordan},
     TITLE = {The homotopy theory of complete modules},
   JOURNAL = {J. Algebra},
  FJOURNAL = {Journal of Algebra},
    VOLUME = {594},
      YEAR = {2022},
     PAGES = {74--100},
      ISSN = {0021-8693,1090-266X},
       DOI = {10.1016/j.jalgebra.2021.11.030},
}

\end{document}